\documentclass[11pt]{amsart}
\usepackage{amsmath}
\usepackage{amsfonts}
\usepackage{amssymb}
\usepackage{mathtools}
\newtheorem{teo}{Theorem}[section]
\newtheorem{lem}[teo]{Lemma}
\newtheorem{cor}[teo]{Corollary}
\newtheorem{prop}[teo]{Proposition}

\newcommand\restrict[1]{\raisebox{-.5ex}{$|$}_{#1}}

\theoremstyle{remark}
\newtheorem{oss}[teo]{Remark}

\theoremstyle{definition}
\newtheorem{defi}[teo]{Definition}

\DeclareMathOperator*{\diver}{div}

\newcommand{\average}{{\mathchoice {\kern1ex\vcenter{\hrule height.4pt
width 6pt
depth0pt} \kern-9.7pt} {\kern1ex\vcenter{\hrule height.4pt width 4.3pt
depth0pt}
\kern-7pt} {} {} }}

\def\vint_#1{\mathchoice%
          {\mathop{\kern 0.2em\vrule width 0.6em height 0.69678ex depth -0.58065ex
                  \kern -0.8em \intop}\nolimits_{\kern -0.4em#1}}%
          {\mathop{\kern 0.1em\vrule width 0.5em height 0.69678ex depth -0.60387ex
                  \kern -0.6em \intop}\nolimits_{#1}}%
          {\mathop{\kern 0.1em\vrule width 0.5em height 0.69678ex
              depth -0.60387ex
                  \kern -0.6em \intop}\nolimits_{#1}}%
          {\mathop{\kern 0.1em\vrule width 0.5em height 0.69678ex depth -0.60387ex
                  \kern -0.6em \intop}\nolimits_{#1}}}
\def\vintslides_#1{\mathchoice%
          {\mathop{\kern 0.1em\vrule width 0.5em height 0.697ex depth -0.581ex
                  \kern -0.6em \intop}\nolimits_{\kern -0.4em#1}}%
          {\mathop{\kern 0.1em\vrule width 0.3em height 0.697ex depth -0.604ex
                  \kern -0.4em \intop}\nolimits_{#1}}%
          {\mathop{\kern 0.1em\vrule width 0.3em height 0.697ex depth -0.604ex
                  \kern -0.4em \intop}\nolimits_{#1}}%
          {\mathop{\kern 0.1em\vrule width 0.3em height 0.697ex depth -0.604ex
                  \kern -0.4em \intop}\nolimits_{#1}}}

\newcommand{\kintint}[2]{\mathchoice%
          {\mathop{\kern 0.2em\vrule width 0.6em height 0.69678ex depth -0.58065ex
                  \kern -0.8em \intop}\nolimits_{\kern -0.45em#1}^{#2}}%
          {\mathop{\kern 0.1em\vrule width 0.5em height 0.69678ex depth -0.60387ex
                  \kern -0.6em \intop}\nolimits_{#1}^{#2}}%
          {\mathop{\kern 0.1em\vrule width 0.5em height 0.69678ex depth -0.60387ex
                  \kern -0.6em \intop}\nolimits_{#1}^{#2}}%
          {\mathop{\kern 0.1em\vrule width 0.5em height 0.69678ex depth -0.60387ex
                  \kern -0.6em \intop}\nolimits_{#1}^{#2}}}

\makeatletter
\def\cleardoublepage{\clearpage\if@twoside \ifodd\c@page\else
\hbox{}
\thispagestyle{empty}
\newpage
\if@twocolumn\hbox{}\newpage\fi\fi\fi}
\makeatother
\title{Some remarks about the existence of an Alt-Caffarelli-Friedman monotonicity formula in the Heisenberg group}
\author{Fausto Ferrari}
\address{Fausto Ferrari: Dipartimento di Matematica\\ Universit\`a di Bologna\\ Piazza di Porta S.Donato 5\\ 40126, Bologna-Italy}
\email{fausto.ferrari@unibo.it }
\author{Nicol\`o Forcillo}
\address{Nicol\`o Forcillo: Dipartimento di Matematica\\ Universit\`a di Bologna\\ Piazza di Porta S.Donato 5\\ 40126, Bologna-Italy}
\email{nicolo.forcillo2@unibo.it }
\thanks{}
\date{\today}

\begin{document}

\begin{abstract}
The aim of this paper is to study the existence of an Alt-Caffarelli-Friedman monotonicity type formula in the Heisenberg group.
\end{abstract}
\maketitle

\tableofcontents

\section{Introduction}
In this paper, starting from the structure of the functions that satisfy the following problem
$$
\left\{\begin{array}{ll}
\Delta_{\mathbb{H}^1}u=0& \mbox{in }\mathcal{P}_\Gamma\cap B_R^{\mathbb{H}^1}(0),\\
u=0&\mbox{on }\Gamma,
\end{array}
\right.
$$ 
where
$$
\mathcal{P}_\Gamma\coloneqq \{(x,y,t)\in \mathbb{H}^1:\quad (x,y,t)=\delta_{\lambda} (\xi,\eta,\tau),\:\:\lambda>0,\:\: (\xi,\eta,\tau)\in \Gamma \subset \partial B_1^{{\mathbb{H}^1}}(0) \},
$$
we examine the existence of an Alt-Caffarelli-Friedman formula in the Heisenberg group.

Here
$\delta_{\lambda} (\xi,\eta,\tau)\coloneqq (\lambda \xi,\lambda \eta,\lambda^2\tau),$ $\lambda>0,$ is the dilation semigroup in the smallest Heisenberg group $\mathbb{H}^1,$ and
  $$B_R^{\mathbb{H}^1}(0)\coloneqq \{(x,y,t)\in \mathbb{H}^1:\quad (x^2+y^2)^2+t^2<R^4 \}$$ is the Koranyi ball  centered at $(0,0,0)$ of radius $R.$   
  
For instance, if $\mathcal{C}=\{(x,y,t)\in \mathbb{H}^1:\quad x^2+y^2<Mt\}$
 is a paraboloid, being $M>0$  constant,  then
$\mathcal{C}= \mathcal{P}_\Gamma$ for $\Gamma=\{(x,y,t)\in \partial B_1^{{\mathbb{H}^1}}(0) :\quad x^2+y^2<Mt\}.$ This type of problem has been faced (the authors having in mind different applications respect to our ones) in \cite{Jerison} and \cite{Biri}. However in those papers, the authors deal with Heisenberg group $\mathbb{H}^{n},$
a more abstract approach, with respect to our computation developed in $\mathbb{H}^1$ only, has been applied. For this reason we think that our approach permits to understanding better some details about our problems and represents by itself a useful tools note for further applications, see Section \ref{laplace_beltrami}.

In fact, as a consequence of our analysis, we obtain necessary and sufficient conditions concerning the existence of an Alt-Caffarelli-Friedman monotonicity formula in $\mathbb{H}^1.$

The Alt-Caffarelli-Friedman  monotonicity formula, was introduced in \cite{ACF} as one of the fundamental tools for studying the main properties of the solutions of some two-phase free boundary problems.

We recall briefly the result in the Euclidean setting. Roughly saying, see \cite{ACF}, there exists $r_0>0$ such that for every given non-negative $u_1,u_2\in C(B_1(0))\cap H^1(B_1(0)),$ if  $u_1u_2=0$ in $B_1(0),$ $u_i(0)=0$ and $\Delta u_i\geq 0,$ $i=1,2,$ where  $B_1(0)$ is the Euclidean ball centered at $0$ of radius $1$ in $\mathbb{R}^{n},$  the function
\begin{equation}\label{expression-of-Phi}
\Phi(r)\coloneqq r^{-4}\int_{B_r(0)}\frac{|\nabla u_1(x)|^2}{|x|^{n-2}}dx\int_{B_r(0)}\frac{|\nabla u_2(x)|^2}{|x|^{n-2}}dx
\end{equation}
is well defined, bounded and monotone increasing in $[0,r_0),$ see also \cite{CS}.

 The original motivation for proving the previous result, as we have just recalled, was associated with the global regularity of the solutions of the following two-phase free boundary problem:
\begin{equation}\label{two_phase_classical}
\left\{\begin{array}{ll}
\Delta u=0& \mbox{in }\Omega^+(u)\coloneqq  \{x\in \Omega:\quad u(x)>0\},\\
\Delta u=0& \mbox{in }\Omega^-(u)\coloneqq \mbox{Int}(\{x\in \Omega:\quad u(x)\leq 0\}),\\
|\nabla u^+|^2-|\nabla u^-|^2=1&\mbox{on }\mathcal{F}(u)\coloneqq \partial \Omega^+(u)\cap \Omega,
\end{array}
\right.
\end{equation}
starting from its variational formulation.
%

In particular, the monotonicity formula was applied for obtaining the global Lipschitz continuous regularity of the solutions of \eqref{two_phase_classical}.

After \cite{ACF}, many other important papers on this topic appeared.  For instance, in \cite{C3} it was proved that monotonicity formula holds for linear uniformly elliptic operators in divergence form with H\"older continuous coefficients, in \cite{CJK} a formula for non-homogeneous free boundary problems was discovered,  in \cite{TZ} the Riemannian case was treated, while in \cite{MP} it was faced the non-divergence form case. Moreover, some very partial results have been obtained also in the non-linear case in dimension $n=2$: see \cite{DiKar} for the $p-$Laplacian case.

Further applications to two-phase free boundary problems may be found in \cite{C1}, dealing with the elliptic homogeneous case, in \cite{ACS1} and \cite{FS_p} in the parabolic homogeneous setting, as well as in \cite{DFS_APDE} for the linear elliptic non-homogeneous problems. In addition, we also recall its applications to some segregation problems, see e.g. \cite{NTV}, \cite{Q}, \cite{TVZ} and \cite{TTV}.

The existence of such a tool for elliptic degenerate operators, for instance sublaplacians on groups, is not yet understood. Anyhow, concerning other formulas about sublaplacians on groups, we find in literature some important contributions to the degenerate linear operators, see \cite{Garofalo_Lanconelli} and in particular \cite{Garofalo_Rotz}, where the authors deal with the frequency function of Almgren in Carnot groups. Moreover, considering further contributions in the non-commutative setting about other free boundary problems, for instance the obstacle problem, we point out \cite{DGS} and \cite{DGP_ob}. 

As regards our research, we have been mainly motivated by our interest in studying the solutions of the following two-phase free boundary problem in the Heisenberg group
\begin{equation}\label{two_phase_Heisenberg}
\begin{cases}
\Delta_{\mathbb{H}^n} u=f& \mbox{in }\Omega^+(u)\coloneqq  \{x\in \Omega:\hspace{0.1cm} u(x)>0\},\\
\Delta_{\mathbb{H}^n} u=f& \mbox{in }\Omega^-(u)\coloneqq \mbox{Int}(\{x\in \Omega:\hspace{0.1cm} u(x)\leq 0\}),\\
|\nabla_{\mathbb{H}^n} u^+|^2-|\nabla_{\mathbb{H}^n} u^-|^2=1&\mbox{on }\mathcal{F}(u)\coloneqq \partial \Omega^+(u)\cap \Omega,
\end{cases}
\end{equation}
where $f\in C(\Omega)\cap L^{\infty}(\Omega).$ In \cite{Fe} the authors obtain this formulation starting from the notion of domain variation solution of a Bernoulli type functional like
 \[
 \mathcal{E}_{\mathbb{H}^n}(v)\coloneqq \int_{\Omega}\left(|\nabla_{\mathbb{H}^n} v|^2+\chi_{\{v>0\}}+2fv\right)dx.
\]


Concerning our contribution to this subject in the Heisenberg group $\mathbb{H}^1,$ see Section \ref{notation_Heisenberg} for the basic notation about this Carnot group, we remark  that we obtain only some partial results based on the size of the following Rayleigh quotient type:
\begin{equation*}\label{def-lambda-varphixx}
\lambda_{\varphi}(\Sigma)=\inf_{v\hspace{0.05cm}\in\hspace{0.05cm}H^{1}_{0}(\Sigma)}\frac{\displaystyle \int_{\Sigma}\frac{\left|\nabla^{\varphi}_{\mathbb{H}^1}v(\xi)\right|^{2}}{\sqrt{x^{2}+y^{2}}}\hspace{0.1cm}d\sigma_{\mathbb{H}^1}(\xi)}{\displaystyle \int_{\Sigma}v^{2}(\xi)\sqrt{x^{2}+y^{2}}\hspace{0.1cm}d\sigma_{\mathbb{H}^1}(\xi)},
\end{equation*}
where $\Sigma\subset \partial B_1^{\mathbb{H}^1}(0)$ is a rectifiable set.
Here $\sigma_{\mathbb{H}^1}$ denotes the perimeter measure in the Heisenberg group, see Section \ref{notation_Heisenberg} for a short introduction, and \cite{GN},  \cite{FSSC_houston} for a detailed exposition.

Furthermore, we denote by
$\nabla^\varphi_{\mathbb{H}^1}v(\xi)=\langle \nabla_{\mathbb{H}^1}v(\xi),e_\varphi(\xi)\rangle e_\varphi(\xi),$ being
$e_\varphi\coloneqq \frac{\nabla_{\mathbb{H}^1}\varphi}{|\nabla_{\mathbb{H}^1}\varphi|},$ where $\varphi$ is the variable associated with the point $\xi\in \mathbb{H}^1$ via the following appropriate spherical coordinates given in the Heisenberg group $\mathbb{H}^1$ by
\begin{equation*}
\begin{cases}
x=\rho \sqrt{\sin \varphi}\cos\theta \\
y=\rho \sqrt{\sin \varphi}\sin \theta\\
t=\rho^{2}\cos \varphi,
\end{cases}
\end{equation*}
see Section \ref{laplace_beltrami} for the notation and for the discussion of some remarks about these coordinates.

We precisely get the following result.

\begin{teo}\label{lowerbound_f}
Let  $u_1,u_2\in C(B_1^{\mathbb{H}^1}(0))\cap H^1_{\mathbb{H}^1}(B_1^{\mathbb{H}^1}(0)),$   $u_1u_2=0$ in $B_1^{\mathbb{H}^1}(0),$ $u_i(0)=0$ and $\Delta_{\mathbb{H}^1} u_i\geq 0,$ $i=1,2.$ Then
\begin{equation}\label{general-term-condition-monotonicity-lower-bound_x}
\sum_{i=1}^2\frac{\displaystyle \int_{\partial B_1^{\mathbb{H}^1}(0)}\frac{\left|\nabla_{\mathbb{H}^1}u_i\right|^{2}}{\sqrt{x^{2}+y^{2}}}\hspace{0.1cm}d \sigma_{\mathbb{H}^1}(\xi)}{\displaystyle \int_{B_1^{\mathbb{H}^1}(0)}\frac{\left|\nabla_{\mathbb{H}^1}u_i\right|^{2}}{\left|\xi\right|_{\mathbb{H}^1}^{2}}\hspace{0.1cm}d\xi}\geq 2\sum_{i=1}^2\left(\sqrt{1+\lambda_{\varphi}\left(\Sigma_i\right)}-1\right).
\end{equation}
\end{teo}
Let us also define
$$
J_{\beta,\mathbb{H}^1}(r)= r^{-\beta}\int_{B_r^{\mathbb{H}^1}(0)}\frac{\mid\nabla_{\mathbb{H}^1} u_1(\xi)\mid^2}{\left|\xi\right|_{\mathbb{H}^1}^{2}}d\xi\int_{B_r^{\mathbb{H}^1}(0)}\frac{\mid\nabla_{\mathbb{H}^1} u_2(\xi)\mid^2}{\left|\xi\right|_{\mathbb{H}^1}^{2}}d\xi,
$$
where $\beta>0$ is a parameter. The function $J_{\beta,\mathbb{H}^1}$ takes the place of $\Phi,$ defined in \eqref{expression-of-Phi}, in the Heisenberg group $\mathbb{H}^1,$ using the parameter $\beta>0$ instead of $\beta=4,$ substituting the Euclidean balls with Koranyi  balls and recalling that the fundamental solution of the Kohn-Laplace operator $\Delta_{\mathbb{H}^1}$ in the Heisenberg group $\mathbb{H}^1$ is, up to a constant, $|\xi|_{\mathbb{H}_1}^{-2}.$
Analogously,  in $\mathbb{H}^n$ the function  $J_{\beta,\mathbb{H}^1}$ becomes:
$$
J_{\beta,\mathbb{H}^n}(r)= r^{-\beta}\int_{B_r^{{\mathbb{H}^n}}(0)}\frac{\mid\nabla_{\mathbb{H}^n} u_1(\zeta)\mid^2}{|\zeta|_{\mathbb{H}^n}^{Q-2}}d\zeta\int_{B_r^{{\mathbb{H}^n}}(0)}\frac{\mid\nabla_{\mathbb{H}^n} u_2(\zeta)\mid^2}{|\zeta|_{\mathbb{H}^n}^{Q-2}}d\zeta,
$$
where $Q\coloneqq 2n+2$ is the homogeneous dimension in $\mathbb{H}^n.$ We state our results in $\mathbb{H}^1,$ even if the proof holds in every $\mathbb{H}^n,$ only because the result we obtain is an intermediate step with respect to the existence of a monotonicity formula in the Heisenberg group. Indeed, see the next Corollary \ref{corcor1}, it still remains open the problem of determining the best configuration in splitting the Koranyi ball in two parts. This fact depends on the best profile of the set that realizes the equality in the isoperimetric inequality in the Heisenberg group. Concerning the same problem in the Euclidean setting, the question is well understood, see the proof of the Alt-Caffarelli-Friedman formula, \cite{ACF}, \cite{FH}, \cite{Sperner}, \cite{CS} and \cite{FeFo} for a recent review of the problem.

In order to state an application of Theorem \ref{lowerbound_f}, let us introduce the following function, see Section \ref{phidependence},
\begin{equation}\label{h_func_introduction}
\begin{split}
&h(\varphi)\coloneqq 2(\sqrt{1+\lambda_0(\varphi)}-1)+2(\sqrt{1+\lambda_0(\pi-\varphi)}-1).
\end{split}
\end{equation}
where $\lambda_0(\varphi)$ is associated with a cap on the Koranyi ball around the $t-$axis of half-opening $\varphi$.
About this fact,  it is worth to recall that the Koranyi ball is not symmetric along all the directions like the Euclidean ball. For instance, we can not obtain $\partial B_1^{\mathbb{H}^1}\cap\{(x,y,t)\in \mathbb{H}^1:\:\:t>0\}$ from $\partial B_1^{\mathbb{H}^1}\cap\{(x,y,t)\in \mathbb{H}^1:\:\:x>0\}$ via anyone rotation in $\mathbb{H}^1\equiv\mathbb{R}^3$.
 
Nevertheless, the function $h$ is symmetric with respect to $\frac{\pi}{2}$ in $[0,\pi].$ Unfortunately, however,  we do not know if the minimum of the function $h$ is realized by $\varphi=\frac{\pi}{2},$ even if this fact would result almost expected. In any case, if it were true that the two half parts of the Koranyi ball split in two half parts by the plane $t=0$ realize  the minimum for $h$, then $\min_{\varphi\in [0,\pi]}{h(\varphi)}\geq 16.$ 

As a consequence, the following corollary holds, see also Corollary \ref{corolmain}.

\begin{cor}\label{corcor1} If there exists a positive number $\beta$ for which $J_{\beta,\mathbb{H}^1}$
 is monotone 
 for every $u_1, u_2\in C(B_1^{\mathbb{H}^1}(0))\cap H_{\mathbb{H}^1}^1(B_1^{{\mathbb{H}^1}}(0))$, such that $\Delta_{\mathbb{H}^1}u_i\geq 0,$ $u_i(0)=0,$ $i=1,2$ and $u_1u_2=0,$   then $\beta\leq 4.$
 Moreover, if the minimum of $J_{\beta,\mathbb{H}^1}$ were realized by two functions like $u_1=(ax+by)^+$ and $u_2=(ax+by)^-$, $a,b\in \mathbb{R},$ $a^2+b^2>0,$ or that
  $$\min_{\varphi\in [0,\pi]}h(\varphi)\geq 16,$$
  then for $\beta=4$ the function $J_{\beta,\mathbb{H}^1}$ is monotone, that is there exists $r_0>0$ such that
   $$
J_{4,\mathbb{H}^1}(r)= r^{-4}\int_{B_r^{{\mathbb{H}^1}}(0)}\frac{\mid\nabla_{\mathbb{H}^1} u_1(\zeta)\mid^2}{|\zeta|_{\mathbb{H}^1}^{2}}d\zeta\int_{B_r^{{\mathbb{H}^1}}(0)}\frac{\mid\nabla_{\mathbb{H}^1} u_2(\zeta)\mid^2}{|\zeta|_{\mathbb{H}^1}^{2}}d\zeta
$$ 
is monotone increasing in $[0,r_0).$
\end{cor}
We comment this last result pointing out that our result moves the problem, about the existence of a monotonicity formula in the Heisenberg group $\mathbb{H}^1,$ in solving the new problem of knowing the {\it smallest} configuration, on the Koranyi sphere of radius one, determined by the lower bound described by the left hand side of \eqref{general-term-condition-monotonicity-lower-bound_x}, that is by the value of:
$$
\inf_{\cup_{i=1}^2\overline{\Sigma}_i=\partial B^{\mathbb{H}^1}_1(0),\:\:\mbox{int}\Sigma_1 \cap \mbox{int}\Sigma_2=\emptyset }2\sum_{i=1}^2\left(\sqrt{1+\lambda_{\varphi}\left(\Sigma_i\right)}-1\right).
$$
More precisely, in case this smallest configuration is realized by the half spheres on the Koranyi ball of radius one, obtained by splitting the ball in two half parts determined by the plane $t=0,$ then $J_{4,\mathbb{H}^1}$ is monotone increasing.

\section{The main notation in the Heisenberg group}\label{notation_Heisenberg}
Let
$\mathbb{H}^n$ be the set $\mathbb{R}^{2n+1},$ $n\in \mathbb{N},$ $n\geq 1,$ endowed with the following non-commutative inner law:

 for every $(x_1,y_1,t_1)\in \mathbb{R}^{2n+1},$ $(x_2,y_2,t_2)\in \mathbb{R}^{2n+1},$ $x_{i}\in \mathbb{R}^n,$ $y_{i}\in \mathbb{R}^n,$ $i=1,2:$ 
$$
(x_1,y_1,t_1)\circ(x_2,y_2,t_2)=(x_1+x_2,y_1+y_2,t_1+t_2+2(\langle x_2, y_1\rangle-\langle x_1,y_2\rangle)),
$$
where $\langle x_i,y_i\rangle$ denotes the usual inner product in $\mathbb{R}^n.$

Let  $X_i=(e_i,0,2y_i)$ and $Y_i=(0,e_i,-2x_i),$ $i=1,\dots,n,$ where $\{e_i\}_{1\leq i\leq n}$ is the canonical basis for $\mathbb{R}^n.$

We use the same symbols to denote the vector fields associated with the previous vectors so that for $i=1,\dots,n$
$$
X_i=\partial_{x_i}+2y_i\partial_t,\quad Y_i=\partial_{y_i}-2x_i\partial_t.
$$
The commutator between the vector fields is for every $i=1,\dots, n:$
$$
[X_i,Y_i]=-4\partial_t,
$$
otherwise is $0.$ 
The intrinsic gradient of a smooth function $u$ in a point $P$ is  
$$
\nabla_{\mathbb{H}^n}u(P)=\sum_{i=1}^n(X_iu(P)X_i(P)+Y_iu(P)Y_i(P)).
$$
There exists a unique metric on $H\mathbb{H}^n_P=\mbox{span}\{X_1,\dots,X_n,Y_1,\dots,Y_n\}$ which makes orthonormal the set of vectors $\{X_1,\dots,X_n,Y_1,\dots,Y_n\}.$ Thus for every $P\in \mathbb{H}^n$ and  for every $U,W\in H\mathbb{H}^n_P,$ $U=\sum_{j=1}^n(\alpha_{1,j}X_{j}(P)+\beta_{1,j}Y_j(P)),$
$V=\sum_{j=1}^n(\alpha_{2,j}X_{j}(P)+\beta_{2,j}Y_j(P))$
$$
\langle U,V\rangle=\sum_{j=1}^n(\alpha_{1,j}\alpha_{2,j}+\beta_{1,j}\beta_{2,j}).
$$  
In particular, we get a norm associated with the metric on $\mbox{span}\{X_1,\dots,X_n,Y_1,\dots,Y_n\}$  and
$$
\mid U\mid=\sqrt{\sum_{i=1}^n\left(\alpha_{1,j}^2+\beta_{1,j}^2\right)}.
$$
For example, the norm of the intrinsic gradient of a smooth function $u$ in $P$ is
$$
\mid \nabla_{\mathbb{H}^n} u(P)\mid=\sqrt{\sum_{i=1}^n\left((X_iu(P))^2+(Y_iu(P))^2\right)}.
$$
Moreover, if $\nabla_{\mathbb{H}^n} u(P)\not=0,$ the norm of
$$
\frac{\nabla_{\mathbb{H}^n}u(P)}{\mid \nabla_{\mathbb{H}^n}u(P)\mid}
$$
is one. If $\nabla_{\mathbb{H}^n} u(P)=0,$ instead, we say that the point $P$ is characteristic for the smooth surface $\{u=u(P)\}.$
Hence, for every point $M\in \{u=u(P)\},$ which is not characteristic, is well defined the intrinsic normal to the surface $\{u=u(P)\}$ as follows: 
$$
\nu(M)=\frac{\nabla_{\mathbb{H}^n} u(M)}{\mid \nabla_{\mathbb{H}^n} u(M)\mid}.
$$
At this point, we introduce in the Heisenberg group $\mathbb{H}^n$ the following gauge norm:
$$
\mid(x,y,t)\mid_{\mathbb{H}^n}\coloneqq \sqrt[4]{(\mid x\mid^2+\mid y\mid^2)^2+t^2},\quad P=(x,y,t)\in \mathbb{H}^n.
$$
In particular, for every positive number $r,$ the gauge ball of radius $r$ centered at $0$ is
$$
B^{\mathbb{H}^n}_r(0)\coloneqq \{P\in \mathbb{H}^n :\:\:|P|_{\mathbb{H}^n}<r\}.
$$
In the Heisenberg group, a dilation semigroup is defined as follows: for every $r>0$ and for every $P=(x,y,t)\in \mathbb{H}^n,$ let
$$
\delta_r(P)\coloneqq (rx,ry,r^2t).
$$

It is well known that, fixed $R\in \mathbb{H}^n,$
$|R^{-1}\circ P|_{\mathbb{H}^n}^{2-Q},$ $P\in \mathbb{H}^n,$ is, up to a multiplicative constant, the fundamental solution of the sublaplacian $\Delta_{\mathbb{H}^n}$ in the Heisenberg group, where $Q=2n+2$  is the homogeneous dimension in $\mathbb{H}^n$ introduced before.

The definition of $\mathbb{H}^n-$subharmonic function, as well as the one of $\mathbb{H}^n-$superharmonic function in a set $\Omega\subset \mathbb{H}^{n},$ can be stated, as usual, in the classical way, requiring respectively that $\Delta_{\mathbb{H}^n}u(P)\geq 0$ for every $P\in \Omega,$  for the $\mathbb{H}^n-$subharmonicity, and that $\Delta_{\mathbb{H}^n}u(P)\leq 0$ for every $P\in \Omega$ for having $\mathbb{H}^n-$superharmonicity. We refer to \cite{BLU} for further details.  

Concerning the natural Sobolev spaces to consider in the Heisenberg group $\mathbb{H}^n$, we refer to the literature, see for instance \cite{GN}. Here, we simply recall that: 
$$
\mathcal{L}^{1,2}(\Omega)\coloneqq \{f\in L^{2}(\Omega):\:\:X_if,\:Y_if\in L^{2}(\Omega),\:\:i=1,\dots, n\}
$$ 
is a Hilbert space with respect to the norm
$$
|f|_{\mathcal{L}^{1,2}(\Omega)}=\left(\int_{\Omega}\bigg(\sum_{i}^n(\left|X_if\right|^2+\left|Y_if\right|^2))+|f|^2\bigg)dx\right)^{\frac{1}{2}}.
$$
Moreover
$$
 H_{\mathbb{H}^n}^1(\Omega)=\overline{C^{\infty}(\Omega)\cap \mathcal{L}^{1,2}(\Omega)}^{|\cdot |_{\mathcal{L}^{1,2}(\Omega)}}.
$$

Now, if $E\subset\mathbb{H}^n$ is a measurable set, a notion of $\mathbb{H}^n$-perimeter measure
$|\partial E|_{\mathbb{H}^n}$ has been introduced in \cite{GN} in a more general setting, even if here we recall some results in the framework of the Heisenberg group, the simplest  non-trivial example of Carnot group. We refer to \cite{GN},
\cite{FSSC_houston}, \cite{FSSC_CAG}, \cite{FSSC_step2} for a detailed presentation. For
our applications, we restrict ourselves to remind that, if $E$ has locally finite $\mathbb{H}^n$-perimeter
(is a $\mathbb{H}^n$-Caccioppoli set),
then $|\partial E|_{\mathbb{H}^n}$ is a Radon measure in $\mathbb{H}^n$, invariant under
group translations and $\mathbb{H}^n$-homogeneous of degree $Q-1$. In addition, the following
representation theorem holds (see  \cite{capdangar}).

\begin{prop}\label{perimetro regolare}
If $E$ is a $\mathbb{H}^{n}\coloneqq \mathbb{R}^{2n+1}$-Caccioppoli set with Euclidean ${\mathbf C}^1$
boundary, then there is an explicit representation of the
$\mathbb{H}^{n}$-perimeter in terms of the Euclidean $2n$-dimensional
Hausdorff measure $\mathcal H^{2n}$
\begin{equation*}
\sigma_{\mathbb{H}^n}^{\Omega, E}(\partial E)=\int_{\partial
E\cap\Omega}\bigg(\sum_{j=1}^{n}\left(\langle
X_j,n_E\rangle_{\mathbb{R}^{2n+1}}^2+\langle
Y_j,n_E\rangle_{\mathbb{R}^{2n+1}}^2\right)\bigg)^{1/2}d{\mathcal {H}}^{2n},
\end{equation*}
where $n_E=n_E(x)$ is the Euclidean unit outward normal to $\partial
E$.
\end{prop}

We also have:
\begin{prop}\label{divergence}
If $E$ is a regular bounded open set with Euclidean ${\mathbf C}^1$
boundary and $\phi$ is a horizontal vector field,
continuously differentiable on $ \overline{E} $, then
$$
\int_E \mathrm{div}_{\mathbb{H}^n}\ \phi\, dx = \int_{\partial E} \langle \phi, \nu_{\mathbb{H}^n}\rangle d \sigma_{\mathbb{H}^n}^{E},
$$
where $\nu_{\mathbb{H}^n}(x)$ is the intrinsic horizontal unit outward normal to $\partial E$,
given by the (normalized) projection of $n_E(x)$ on the fiber $H\mathbb{H}^n_x$ of
the horizontal fiber bundle $H\mathbb{H}^n$.
\end{prop}

\begin{oss}
The definition of $\nu_{\mathbb{H}^n}$ is well done, since $H\mathbb{H}^n_x$ is transversal to
the tangent space of $E$ at $x,$ for  $\sigma_{\mathbb{H}^n}^{E}(\partial E)$-a.e. $x\in\partial E$
(see \cite{magnani}).
\end{oss}

Following \cite{ACF} and \cite{Fe} in the Heisenberg framework, we conclude, by applying the definition of solution in the sense of the domain variation to the functional
 \[
 \mathcal{E}_{\mathbb{H}^n}(v)\coloneqq \int_{\Omega}\left(|\nabla_{\mathbb{H}^n} v|^2+\chi_{\{v>0\}}\right)dx,
\]
$\Omega \subset\mathbb{H}^n,$
 that the parallel two-phase problem to (\ref{two_phase_classical})
  is:
\begin{equation}\label{two_phase_Heisenberg}
\begin{cases}
\Delta_{\mathbb{H}^n} u=0& \mbox{in }\Omega^+(u)\coloneqq  \{x\in \Omega:\hspace{0.1cm} u(x)>0\},\\
\Delta_{\mathbb{H}^n} u=0& \mbox{in }\Omega^-(u)\coloneqq \mbox{Int}(\{x\in \Omega:\hspace{0.1cm} u(x)\leq 0\}),\\
|\nabla_{\mathbb{H}^n} u^+|^2-|\nabla_{\mathbb{H}^n} u^-|^2=1&\mbox{on }\mathcal{F}(u)\coloneqq \partial \Omega^+(u)\cap \Omega.
\end{cases}
\end{equation}
Thus, the candidate to give an Alt-Caffarelli-Friedman monotonicity formula in the Heisenberg group is the following function:
\begin{equation}\label{monofondformula}
J_{\beta,\mathbb{H}^n}(r)= r^{-\beta}\int_{B_r^{{\mathbb{H}^n}}(0)}\frac{\mid\nabla_{\mathbb{H}^n} u^+\mid^2}{|\zeta |_{\mathbb{H}^n}^{Q-2}}d\zeta\int_{B_r^{{\mathbb{H}^n}}(0)}\frac{\mid\nabla_{\mathbb{H}^n} u^-\mid^2}{|\zeta|_{\mathbb{H}^n}^{Q-2}}d\zeta,
\end{equation}
where $\beta>0$ is a suitable fixed exponent and $u^+\coloneqq \sup\{u,0\}$ and $u^-\coloneqq \sup\{-u,0\},$ being $0\in \mathcal{F}(u).$

 We refer to \cite{BLU} for the following statements:
\begin{defi}\label{definition-mollifier}
Let $u:\mathbb{H}^n\to [-\infty,+\infty)$ be a function. Let $J\in C_0^{\infty}(\mathbb{H}^n),$ $J\geq 0$ such that $\mbox{supp}(J)\subset D_1(0)$ and $\int_{\mathbb{H}^n}J=1.$ For every positive number $\varepsilon,$ we define $u_{\varepsilon}$ to be the Friedrichs mollifier of $u$ as:
$$
u_{\varepsilon}(x)=\varepsilon^{-Q}\int_{\mathbb{H}^n}u(-y\circ x)J(\delta_{\varepsilon^{-1}}(y))dy.
$$ 
\end{defi}
\begin{prop}\label{subharmonicity-mollifier}
Let $u:\mathbb{H}^n\to [-\infty,+\infty)$ be a $\mathbb{H}^n-$ subharmonic function. Then $u_\varepsilon\in C^\infty(\mathbb{H}^n)$ is $\mathbb{H}^n-$ subharmonic, and $u_\varepsilon \to u$ in $L^1_{\mathrm{loc}}(\mathbb{H}^n)$ as $\varepsilon\to 0^+.$
\end{prop}
Let $\mathcal{S}(\mathbb{H}^n)$ be the set of the $\mathbb{H}^n-$subharmonic functions in $\mathbb{H}^n.$ Then if $u\in \mathcal{S}(\mathbb{H}^n),$  $L_u:C_0^\infty(\mathbb{H}^n)\to \mathbb{R},$
$$
 L_u(\varphi)=\int_{\mathbb{H}^n}u(x)\Delta_{\mathbb{H}^n}\varphi(x)dx 
 $$
 is positive, i.e. if $u\in \mathcal{S}(\mathbb{H}^n)$ then $\Delta_{\mathbb{H}^n}u\geq 0$ in the distributional sense.
 
 \section{Some key steps in the Euclidean case}
Since we closely follow the Euclidean proof, we recall the main steps.
After a straightforward differentiation, it results, in view of \eqref{expression-of-Phi},
$$
\Phi'(r)=I_{1}(r)I_{2}(r)r^{-5}\left(-4+r\left(\frac{I_{1}'}{I_{1}}+\frac{I_{2}'}{I_{2}}\right)\right),
$$
where for $i=1,2:$
$$
I_i(r)=\int_{B_r(0)}\frac{|\nabla u_i(x)|^2}{|x|^{n-2}}dx.
$$
By a rescaling argument the problem may be reduced to 
\begin{equation}\label{derivative-of-Phi-rescaled}
\Phi'(r)=I_{1}(r)I_{2}(r)r^{-5}\left(-4+\frac{\displaystyle\int_{\partial B_1(0)}|\nabla u_1(x)|^2d\sigma}{\displaystyle\int_{ B_1(0)}\frac{|\nabla u_1(x)|^2}{|x|^{n-2}}dx}+\frac{\displaystyle\int_{\partial B_1(0)}|\nabla u_2(x)|^2d\sigma}{\displaystyle\int_{ B_1(0)}\frac{|\nabla u_2(x)|^2}{|x|^{n-2}}dx}\right).
\end{equation}
Thus, if 
$$
-4+\frac{\displaystyle\int_{\partial B_1(0)}|\nabla u_1(x)|^2d\sigma}{\displaystyle\int_{ B_1(0)}\frac{|\nabla u_1(x)|^2}{|x|^{n-2}}dx}+\frac{\displaystyle\int_{\partial B_1(0)}|\nabla u_2(x)|^2d\sigma}{\displaystyle\int_{ B_1(0)}\frac{|\nabla u_2(x)|^2}{|x|^{n-2}}dx}\geq 0,
$$
then, from \eqref{derivative-of-Phi-rescaled}, $\Phi'(r)\geq 0.$ 
Hence, in order to prove that previous inequality holds,  the following ratios
$$
\frac{\displaystyle\int_{\partial B_1(0)}|\nabla u_i(x)|^2d\sigma}{\displaystyle\int_{ B_1(0)}\frac{|\nabla u_i(x)|^2}{|x|^{n-2}}dx}
$$
for $i=1,2,$ have to be estimated.

Since the gradient may split in two orthogonal parts involving the radial part and the tangential part, respectively denoted by $\nabla^\rho u_i$ and $\nabla^\theta u_i,$ it holds
$$
|\nabla u_i(x)|^2=|\nabla^\rho u_i(x)|^2+|\nabla^\theta u_i(x)|^2,
$$
so that, using a further integration by parts, H\"older and Cauchy inequality and the facts that $\left|x\right|^{2-n}$ is, up to a multiplicative constant, the fundamental solution of $\Delta$ and $0\in \mathcal{F}(u_i),$ $i=1,2,$ we achieve, for every $\beta_i\in (0,1),$
\begin{equation*}\begin{split}
&\frac{\displaystyle\int_{\partial B_1(0)}|\nabla u_i(x)|^2d\sigma}{\displaystyle\int_{ B_1(0)}\frac{|\nabla u_i(x)|^2}{|x|^{n-2}}dx}=\frac{\displaystyle\int_{\Gamma_i}\left(|\nabla^\rho u_i(x)|^2+|\nabla^\theta u_i(x)|^2\right)d\sigma}{\displaystyle\int_{ \Gamma_i}\Big(u_i(x)|\nabla^\rho u_i(x)|+\frac{n-2}{2} u_i^2(x)\Big)d\sigma}\\
&\geq \frac{2\bigg(\displaystyle\int_{\Gamma_i}|\nabla^\rho u_i(x)|^2d\sigma\bigg)^{\frac{1}{2}}\bigg(\int_{\Gamma_i}\beta_i \lambda(\Gamma_i)  u_i^2(x)d\sigma\bigg)^{\frac{1}{2}}+(1-\beta_i)\lambda(\Gamma_i)\int_{\Gamma_i}u_i^2(x)d\sigma}{\bigg(\displaystyle \int_{ \Gamma_i} |\nabla^\rho u_i(x)|^2d\sigma\bigg)^{\frac{1}{2}}\bigg(\int_{ \Gamma_i}u_i^2(x)d\sigma\bigg)^{\frac{1}{2}}+\frac{n-2}{2} \int_{ \Gamma_i}u_i^2(x)d\sigma}\\
&=:\frac{2(\beta_i\lambda(\Gamma_i))^{\frac{1}{2}}\xi_i\eta_i+(1-\beta_i)\lambda (\Gamma_i)\eta_i^2}{\xi_i\eta_i+\frac{n-2}{2} \eta_i^2}\\.
\end{split}
\end{equation*}
Therefore,
\begin{equation*}\begin{split}
&\frac{\displaystyle\int_{\partial B_1(0)}|\nabla u_i(x)|^2d\sigma}{\displaystyle\int_{ B_1(0)}\frac{|\nabla u_i(x)|^2}{|x|^{n-2}}dx}= \frac{2(\beta_i\lambda(\Gamma_i))^{\frac{1}{2}}+(1-\beta_i)\lambda (\Gamma_i)\frac{\eta_i}{\xi_i}}{1+\frac{n-2}{2} \frac{\eta_i}{\xi_i}}\\
&\geq \inf_{z\geq 0} \frac{2(\beta_i\lambda(\Gamma_i))^{\frac{1}{2}}+(1-\beta_i)\lambda (\Gamma_i)z}{1+\frac{n-2}{2} z}=2\min\left\{\frac{\lambda (\Gamma_i)}{n-2}(1-\beta_i),(\beta_i\lambda(\Gamma_i))^{\frac{1}{2}}\right\},
\end{split}
\end{equation*}
where $\Gamma_i\coloneqq \{x\in \partial B_1(0):\hspace{0.2cm} u_i(x)>0\}$ and
$$
\lambda(\Gamma_i)\coloneqq \inf_{v\hspace{0.05cm}\in\hspace{0.05cm} H_0^1(\Gamma_i)}\frac{\displaystyle\int_{\Gamma_i}|\nabla^\theta v(x)|^2d\sigma}{\displaystyle\int_{\Gamma_i}v^2(x)\hspace{0.05cm}d\sigma}
$$
are the Rayleigh quotients for $i=1,2.$

At this point, if we choose $\beta_i$ in such a way that 
$$
\frac{\lambda (\Gamma_i)}{n-2}(1-\beta_i)=(\beta_i\lambda(\Gamma_i))^{\frac{1}{2}}
$$
we realize, by denoting  $\alpha_i\coloneqq (\beta_i\lambda(\Gamma_i))^{\frac{1}{2}},$ that previous equation is satisfied if and only if
$$
 \alpha_i^2+(n-2)\alpha_i-\lambda(\Gamma_i)=0.
$$
On the other hand, because a function $u=\rho^{\alpha}g(\theta),$ $\theta\coloneqq (\theta_1,\dots,\theta_{n-1}),$ is harmonic in a cone determined by a domain $\Gamma$ whenever
$$
\rho^{\alpha-1}\left(( \alpha(\alpha-1)+\alpha (n-1))g(\theta)+\Delta_{\theta} g\right)=0,
$$
we deduce that there exists $\alpha_i$ such that
$$
\alpha_i(\alpha_i-1)+\alpha_i (n-1)=\lambda(\Gamma_i).
$$
By the structure of the equation, it immediately comes that there always exists a strictly positive solution, so that we have to prove the existence of $\beta_i\in (0,1)$ such that
\begin{equation}\label{betaf}
\frac{-(n-2)+\sqrt{(n-2)^2+4\lambda(\Gamma_i)}}{2}=(\beta_i\lambda(\Gamma_i))^{\frac{1}{2}}.
\end{equation}
Specifically, the last relationship is equivalent to solve
$$
\frac{4\lambda(\Gamma_i)}{(n-2)+\sqrt{(n-2)^2+4\lambda(\Gamma_i)}}=2(\beta_i\lambda(\Gamma_i))^{\frac{1}{2}},
$$
that is
$$
\frac{2\lambda(\Gamma_i)^{\frac{1}{2}}}{(n-2)+\sqrt{(n-2)^2+4\lambda(\Gamma_i)}}=\beta_i^{\frac{1}{2}}.
$$
Now, since the continuous positive function defined in $[0+\infty)$ as 
$$
z\to\frac{z}{(n-2)+\sqrt{(n-2)^2+z^2}},
$$
is strictly increasing, $\bigg(\frac{z}{(n-2)+\sqrt{(n-2)^2+z^2}}\bigg)(0)=0$ and $\sup_{[0,+\infty)} \frac{z}{(n-2)+\sqrt{(n-2)^2+z^2}}=1,$ we conclude that for every $\lambda(\Gamma_i)>0,$ there exists $\beta_i$ such that
(\ref{betaf}) holds. In particular, we get
$$
\beta_i=\left(\frac{2\lambda(\Gamma_i)^{\frac{1}{2}}}{(n-2)+\sqrt{(n-2)^2+4\lambda(\Gamma_i)}}\right)^2.
$$
Hence, with previous choice of $\beta_i,$ if we denote by
$$
\alpha_i\coloneqq \min\left\{\frac{\lambda (\Gamma_i)}{n-2}(1-\beta_i),(\beta_i\lambda(\Gamma_i))^{\frac{1}{2}}\right\},
$$
which is also the exponent corresponding to the eigenvalue given by the Rayleigh quotient $\lambda (\Gamma_i),$ we conclude that, whenever
\begin{equation}\label{belowestimate}
\alpha_1+\alpha_2\geq 2,
\end{equation}
then $\Phi'\geq 0.$
The number $\alpha_i$ is called the characteristic constant of $\Gamma_i,$ see (\cite{FH}). So, for concluding this proof we would need to know that (\ref{belowestimate}) holds.
Indeed, assuming that the minimal configuration  in splitting the surface ball is given by two convex components that are symmetric, the result in the Euclidean setting is proved, see \cite{BKP} and \cite{CS}. This last part passes through a non-trivial discussion based on a rearrangement argument that we omit here, see \cite{CS}, \cite{Sperner}, \cite{FH} and \cite{FeFo} for further details.

\section{Some estimates in the Heisenberg group}
In this section, we provide some partial steps in order to deal with a monotonicity formula associated with $J_{\beta,\mathbb{H}^n}$.

\begin{lem}\label{limitato}
There exists a positive constant $c=c(Q)$ such that for every non-negative $\mathbb{H}^n-$subharmonic function $u\in C(B^{\mathbb{H}^n}_1(0)),$ if $u(0)=0$ then 
$$
\int_{B^{\mathbb{H}^n}_\rho(0)}\frac{\mid\nabla_{\mathbb{H}^n} u(\zeta)\mid^2}{|\zeta|_{\mathbb{H}^n}^{Q-2}}d\zeta\leq  c\rho^{-Q}\int_{{B^{\mathbb{H}^n}_{2\rho}(0)}\setminus{B^{\mathbb{H}^n}_\rho(0)}}u^2(\zeta)\hspace{0.05cm}d\zeta.
$$
\end{lem}

\begin{proof}
Let $u_\varepsilon$ be the Friedrichs mollifier of $u.$ Then, by hypothesis, Definition \ref{definition-mollifier} and Proposition \ref{subharmonicity-mollifier}, we have
$$
\Delta_{\mathbb{H}^n}u_\varepsilon^2=2u_\varepsilon\Delta_{\mathbb{H}^n}u_\varepsilon+2\mid\nabla_{\mathbb{H}^n}u_\varepsilon\mid^2\geq 2\mid\nabla_{\mathbb{H}^n}u_\varepsilon\mid^2.
$$
Hence, for every positive test function $\varphi\in C_0^\infty(B^{\mathbb{H}^n}_1(0))$

\begin{equation*}
\begin{split}
&\int_{B^{\mathbb{H}^n}_1(0)}u^2\Delta_{\mathbb{H}^n}\varphi \hspace{0.1cm}d\zeta=\lim_{\varepsilon \to 0}\int_{B^{\mathbb{H}^n}_1(0)}u_\varepsilon^2\Delta_{\mathbb{H}^n}\varphi\hspace{0.1cm}d\zeta=\lim_{\varepsilon \to 0}\int_{B^{\mathbb{H}^n}_1(0)}\varphi\Delta_{\mathbb{H}^n}u_\varepsilon^2\hspace{0.1cm}d\zeta\\
&\geq \lim_{\varepsilon \to 0} 2 \int_{B^{\mathbb{H}^n}_1(0)}\varphi|\nabla_{\mathbb{H}^n}u_\varepsilon|^2d\zeta = 2 \int_{B^{\mathbb{H}^n}_1(0)}\varphi|\nabla_{\mathbb{H}^n}u|^2d\zeta.
\end{split}
\end{equation*}
Thus
\begin{equation}\label{inequality-distributional-sense}
\int_{B^{\mathbb{H}^n}_1(0)}u^2\Delta_{\mathbb{H}^n}\varphi \hspace{0.1cm}d\zeta\geq 2 \int_{B^{\mathbb{H}^n}_1(0)}\varphi|\nabla_{\mathbb{H}^n}u|^2d\zeta.
\end{equation}
As a consequence, $u\in H^1_{{\mathbb{H}^1}_{\mathrm{loc}}}(B^{\mathbb{H}^n}_1(0))$ and $\Delta_{\mathbb{H}^n}u^2\geq 2\mid\nabla_{\mathbb{H}^n}u\mid^2$ as a distribution. 
Let now $\psi$ be a cutoff function, $\psi\equiv 1$ in $B^{\mathbb{H}^n}_{\rho}(0),$ $\psi\equiv 0$ in $B^{\mathbb{H}^n}_{\rho}(0),$ $0<\rho<\frac{1}{2}.$ We also set
\begin{equation}\label{definition-gamma-epsilon}
\gamma_\varepsilon=\eta_\varepsilon\ast \gamma,
\end{equation}
where $\gamma(\zeta)=|\zeta|_{\mathbb{H}^n}^{2-Q}$ and $\eta_\varepsilon$ is an approximation of the identity. Then, $\psi\gamma_\varepsilon$ is a positive test function in $B^{\mathbb{H}^n}_{\rho}(0),$ thus, in view of \eqref{inequality-distributional-sense} and \eqref{definition-gamma-epsilon}, we achieve
\begin{equation*}
\begin{split}
&2\int_{B^{\mathbb{H}^n}_{\rho}(0)}\psi\gamma_\varepsilon\mid\nabla_{\mathbb{H}^n}u\mid^2 d\zeta\leq 2\int_{B^{\mathbb{H}^n}_{2\rho}(0)}\psi\gamma_\varepsilon\mid\nabla_{\mathbb{H}^n}u\mid^2 d\zeta\leq\int_{B^{\mathbb{H}^n}_{2\rho}(0)}\psi\gamma_\varepsilon\Delta_{\mathbb{H}^n}u^2\hspace{0.1cm}d\zeta\\
&=\int_{B^{\mathbb{H}^n}_{2\rho}(0)}\Delta_{\mathbb{H}^n}(\psi\gamma_\varepsilon)u^2\hspace{0.1cm}d\zeta= \int_{B^{\mathbb{H}^n}_{2\rho}(0)}\psi\Delta_{\mathbb{H}^n}(\gamma_\varepsilon)u^2\hspace{0.1cm}d\zeta+2\int_{B^{\mathbb{H}^n}_{2\rho}(0)\setminus B^{\mathbb{H}^n}_{\rho}(0)}\langle\nabla_{\mathbb{H}^n}\psi,\nabla_{\mathbb{H}^n}\gamma_\varepsilon\rangle u^2\hspace{0.1cm}d\zeta\\
&+\int_{B^{\mathbb{H}^n}_{2\rho}(0)\setminus B^{\mathbb{H}^n}_{\rho}(0)}\Delta_{\mathbb{H}^n}(\psi)\gamma_\varepsilon u^2\hspace{0.1cm}d\zeta=\int_{B^{\mathbb{H}^n}_{2\rho}(0)}\psi(\eta_\varepsilon\ast\Delta_{\mathbb{H}^n}(\gamma))u^2\hspace{0.1cm}d\zeta\\
&+2\int_{B^{\mathbb{H}^n}_{2\rho}(0)\setminus B^{\mathbb{H}^n}_{\rho}(0)}\langle\nabla_{\mathbb{H}^n}\psi,\eta_\varepsilon\ast\nabla_{\mathbb{H}^n}\gamma\rangle u^2\hspace{0.1cm}d\zeta+\int_{B^{\mathbb{H}^n}_{2\rho}(0)\setminus B^{\mathbb{H}^n}_{\rho}(0)}\Delta_{\mathbb{H}^n}(\psi)\gamma_\varepsilon u^2\hspace{0.1cm}d\zeta,
\end{split}
\end{equation*}
which yields, letting $\varepsilon$ go to $0,$ because $\gamma_\varepsilon \to \gamma$ in $L^1_{\mathrm{loc}}(\mathbb{H}^n),$ $\psi\equiv 1$ in $B^{\mathbb{H}^n}_{\rho}(0),$ $\gamma$ is, up to a multiplicative constant, the fundamental solution of $\Delta_{\mathbb{H}^n}$ and $u(0)=0:$
\begin{align*}
&\int_{B^{\mathbb{H}^n}_{\rho}(0)}\frac{\mid\nabla_{\mathbb{H}^n}u\mid^2}{|\zeta|_{\mathbb{H}^n}^{Q-2}} d\zeta\leq \int_{B^{\mathbb{H}^n}_{2\rho}(0)}\psi\Delta_{\mathbb{H}^n}(\gamma)u^2\hspace{0.1cm}d\zeta+2\int_{B^{\mathbb{H}^n}_{2\rho}(0)\setminus B^{\mathbb{H}^n}_{\rho}(0)}\langle\nabla_{\mathbb{H}^n}\psi,\nabla_{\mathbb{H}^n}\gamma\rangle u^2\hspace{0.1cm}d\zeta\\
&+\int_{B^{\mathbb{H}^n}_{2\rho}(0)\setminus B^{\mathbb{H}^n}_{\rho}(0)}\Delta_{\mathbb{H}^n}(\psi)|\zeta|_{\mathbb{H}^n}^{2-Q} u^2\hspace{0.1cm}d\zeta\leq c\rho^{-Q}\int_{B^{\mathbb{H}^n}_{2\rho}(0)\setminus B^{\mathbb{H}^n}_{\rho}(0)}u^2,
\end{align*}
that is
\[\int_{B^{\mathbb{H}^n}_{\rho}(0)}\frac{\mid\nabla_{\mathbb{H}^n}u\mid^2}{|\zeta|_{\mathbb{H}^n}^{Q-2}} d\zeta\leq c\rho^{-Q}\int_{B^{\mathbb{H}^n}_{2\rho}(0)\setminus B^{\mathbb{H}^n}_{\rho}(0)}u^2.\]
\end{proof}


Let us define now the following function

\begin{equation}\label{expression-of-J-beta-H-^-n}
J_{\beta,\mathbb{H}^n}(r)= r^{-\beta}\int_{B_r^{{\mathbb{H}^n}}(0)}\frac{\mid\nabla_{\mathbb{H}^n} u_1(\zeta)\mid^2}{|\zeta|_{\mathbb{H}^n}^{Q-2}}d\zeta\int_{B_r^{{\mathbb{H}^n}}(0)}\frac{\mid\nabla_{\mathbb{H}^n} u_2(\zeta)\mid^2}{|\zeta|_{\mathbb{H}^n}^{Q-2}}d\zeta,
\end{equation}
where $\beta>0$ is a parameter.

\begin{lem}\label{lemma-quotient-derivative-of-J-value-of-J}
For every non-negative $\mathbb{H}^n-$subharmonic functions  $u_i\in C(B^{\mathbb{H}^n}_1(0)),$ $i=1,2,$  such that $u_1u_2=0$ and  $u_1(0)=u_2(0)=0,$ we have
\begin{equation}\label{quotient-derivative-of-J-value-of-J}
 \frac{J_{\beta,\mathbb{H}^n}'(1)}{J_{\beta,\mathbb{H}^n}(1)}=\frac{\displaystyle\int_{\partial B_1^{{\mathbb{H}^n}}(0)}\frac{\mid\nabla_{\mathbb{H}^n} u_1(\kappa)\mid^2}{\sqrt{\mid x\mid^2+\mid y\mid^2}}d\sigma_{\mathbb{H}^n}(\kappa)}{\displaystyle\int_{B_1^{{\mathbb{H}^n}}(0)}\frac{\mid\nabla_{\mathbb{H}^n} u_1(\kappa)\mid^2}{|\kappa|_{{\mathbb{H}^n}}^{Q-2}}d\kappa}+\frac{\displaystyle\int_{\partial B_1^{{\mathbb{H}^n}}(0)}\frac{\mid\nabla_{\mathbb{H}^n} u_2(\kappa)\mid^2}{\sqrt{\mid x\mid^2+\mid y\mid^2}}d\sigma_{\mathbb{H}^n}(\kappa)}{\displaystyle\int_{B_1^{{\mathbb{H}^n}}(0)}\frac{\mid\nabla_{\mathbb{H}^n} u_2(\kappa)\mid^2}{|\kappa|_{{\mathbb{H}^n}}^{Q-2}}d\kappa}-\beta.
\end{equation}
Moreover, if $\frac{J'_{\beta,\mathbb{H}^n}(1)}{J_{\beta,\mathbb{H}^n}(1)}\geq 0,$ then there exists $r_0>0$ such that $J_{\beta,\mathbb{H}^n}$ is monotone increasing in $[0,r_0)$.
\end{lem}

 \begin{proof}
 It follows from Lemma \ref{limitato} that $J_{\beta,\mathbb{H}^n}$ is well defined in $(0,1).$ Differentiating with respect to $r$ and recalling the co-area formula in the Heisenberg group, we get
 
 \begin{equation}
\begin{split}
J_{\beta,\mathbb{H}^n}'(r)&=-\beta r^{-\beta-1}\int_{B_r^{{\mathbb{H}^n}}(0)}\frac{\mid\nabla_{\mathbb{H}^n} u_1(\zeta)\mid^2}{|\zeta|_{{\mathbb{H}^n}}^{Q-2}}d\zeta \int_{B_r^{{\mathbb{H}^n}}(0)}\frac{\mid\nabla_{\mathbb{H}^n} u_2(\zeta)\mid^2}{|\zeta|_{{\mathbb{H}^n}}^{Q-2}}d\zeta\\
&+ r^{-\beta}\int_{\partial B_r^{{\mathbb{H}^n}}(0)}|\zeta|_{\mathbb{H}^n}\frac{\mid\nabla_{\mathbb{H}^n} u_1(\zeta)\mid^2}{r^{Q-2}\sqrt{\mid x\mid^2+\mid y\mid^2}}d\sigma_{\mathbb{H}^n}(\zeta)\int_{B_r^{{\mathbb{H}^n}}(0)}\frac{\mid\nabla_{\mathbb{H}^n}
	 u_2(\zeta)\mid^2}{|\zeta|_{{\mathbb{H}^n}}^{Q-2}}d\zeta 
\\
&+r^{-\beta}\int_{B_r^{{\mathbb{H}^n}}(0)}\frac{\mid\nabla_{\mathbb{H}^n} u_1(\zeta)\mid^2}{|\zeta|_{{\mathbb{H}^n}}^{Q-2}}d\zeta \int_{\partial B_r^{{\mathbb{H}^n}}(0)}|\zeta|_{\mathbb{H}^n}\frac{\mid\nabla_{\mathbb{H}^n}u_2(\zeta)\mid^2}{r^{Q-2}\sqrt{\mid x\mid^2+\mid y\mid^2}}d\sigma_{\mathbb{H}^n}(\zeta)\\
&=-\beta r^{-\beta-1}\int_{B_r^{{\mathbb{H}^n}}(0)}\frac{\mid\nabla_{\mathbb{H}^n} u_1(\zeta)\mid^2}{|\zeta|_{{\mathbb{H}^n}}^{Q-2}}d\zeta \int_{B_r^{{\mathbb{H}^n}}(0)}\frac{\mid\nabla_{\mathbb{H}^n} u_2(\zeta)\mid^2}{|\zeta|_{{\mathbb{H}^n}}^{Q-2}}d\zeta\\
&+ r^{-\beta}\int_{\partial B_r^{{\mathbb{H}^n}}(0)}\frac{\mid\nabla_{\mathbb{H}^n} u_1(\zeta)\mid^2}{r^{Q-3}\sqrt{\mid x\mid^2+\mid y\mid^2}}d\sigma_{\mathbb{H}^n}(\zeta)\int_{B_r^{{\mathbb{H}^n}}(0)}\frac{\mid\nabla_{\mathbb{H}^n}
	u_2(\zeta)\mid^2}{|\zeta|_{{\mathbb{H}^n}}^{Q-2}}d\zeta 
\\
&+r^{-\beta}\int_{B_r^{{\mathbb{H}^n}}(0)}\frac{\mid\nabla_{\mathbb{H}^n} u_1(\zeta)\mid^2}{|\zeta|_{{\mathbb{H}^n}}^{Q-2}}d\zeta \int_{\partial B_r^{{\mathbb{H}^n}}(0)}\frac{\mid\nabla_{\mathbb{H}^n}u_2(\zeta)\mid^2}{r^{Q-3}\sqrt{\mid x\mid^2+\mid y\mid^2}}d\sigma_{\mathbb{H}^n}(\zeta)
 \end{split}
 \end{equation}
 
 Notice that by a change of variables, denoting $\kappa \in \partial B^{\mathbb{H}^n}_1(0)$ as $\kappa=(\kappa_x,\kappa_y,k_t)$ with $\kappa_x,$ $\kappa_y\in \mathbb{R}^n$ and $\kappa_t\in \mathbb{R},$ we obtain
 
 \begin{equation*}
\begin{split}
J_{\beta,\mathbb{H}^n}'(r)&= -\beta r^{-\beta-1}r^{Q}\int_{B_1^{{\mathbb{H}^n}}(0)}\frac{\mid\nabla_{\mathbb{H}^n} u_1(\delta_r(\kappa))\mid^2}{|\delta_r(\kappa)|^{Q-2}_{\mathbb{H}^n}}d\kappa\hspace{0.3cm} r^{Q}\int_{B_1^{{\mathbb{H}^n}}(0)}\frac{\mid\nabla_{\mathbb{H}^n} u_2(\delta_r(\kappa))\mid^2}{|\delta_r(\kappa)|^{Q-2}_{\mathbb{H}^n}}d\kappa\\
&+r^{-\beta}r^{Q-1}\int_{\partial B_1^{{\mathbb{H}^n}}(0)}\frac{\mid\nabla_{\mathbb{H}^n} u_1(\delta_r(\kappa))\mid^2}{r^{Q-3}\sqrt{\mid r \kappa_x\mid^2+\mid r\kappa_y\mid^2}}d\sigma_{\mathbb{H}^n}(\kappa)\hspace{0.3cm}r^Q\int_{B_1^{{\mathbb{H}^n}}(0)}\frac{\mid\nabla_{\mathbb{H}^n} u_2(\delta_r(\kappa))\mid^2}{|\delta_r(\kappa)|^{Q-2}_{\mathbb{H}^n}}d\kappa\\
&+r^{-\beta}r^Q\int_{B_1^{{\mathbb{H}^n}}(0)}\frac{\mid\nabla_{\mathbb{H}^n} u_1(\delta_r(\kappa))\mid^2}{|\delta_r(\kappa)|^{Q-2}_{\mathbb{H}^n}}d\kappa \hspace{0.3cm}r^{Q-1}\int_{\partial B_1^{{\mathbb{H}^n}}(0)}\frac{\mid\nabla_{\mathbb{H}^n}u_2(\delta_r(\kappa))\mid^2}{r^{Q-3}\sqrt{\mid r\kappa_x\mid^2+\mid r\kappa_y\mid^2}}d\sigma_{\mathbb{H}^n}(\kappa)\\
&=-\beta r^{-\beta-1}r^Q\int_{B_1^{{\mathbb{H}^n}}(0)}\frac{\mid\nabla_{\mathbb{H}^n} u_1(\delta_r(\kappa))\mid^2}{r^{Q-2}|\kappa|_{\mathbb{H}^n}^{Q-2}}d\kappa\hspace{0.3cm}r^Q \int_{B_1^{{\mathbb{H}^n}}(0)}\frac{\mid\nabla_{\mathbb{H}^n} u_2(\delta_r(\kappa))\mid^2}{r^{Q-2}|\kappa|_{\mathbb{H}^n}^{Q-2}}d\kappa\\
&+r^{-\beta}r^{2}\int_{\partial B_1^{{\mathbb{H}^n}}(0)}\frac{\mid\nabla_{\mathbb{H}^n} u_1(\delta_r(\kappa))\mid^2}{r\sqrt{\mid \kappa_x\mid^2+\mid \kappa_y\mid^2}}d\sigma_{\mathbb{H}^n}(\kappa)\hspace{0.3cm}r^Q\int_{B_1^{{\mathbb{H}^n}}(0)}\frac{\mid\nabla_{\mathbb{H}^n} u_2(\delta_r(\kappa))\mid^2}{r^{Q-2}|\kappa|_{\mathbb{H}^n}^{Q-2}}d\kappa\\
&+r^{-\beta}r^Q\int_{B_1^{{\mathbb{H}^n}}(0)}\frac{\mid\nabla_{\mathbb{H}^n} u_1(\delta_r(\kappa))\mid^2}{r^{Q-2}|\kappa|_{\mathbb{H}^n}^{Q-2}}d\kappa \hspace{0.3cm}r^{2}\int_{\partial B_1^{{\mathbb{H}^n}}(0)}\frac{\mid\nabla_{\mathbb{H}^n}u_2(\delta_r(\kappa))\mid^2}{r\sqrt{\mid \kappa_x\mid^2+\mid \kappa_y\mid^2}}d\sigma_{\mathbb{H}^n}(\kappa),
\end{split}
\end{equation*}
which gives
\begin{equation}\label{derivative-of-J-beta-H-^-n-1}
\begin{split}
J'_{\beta,\mathbb{H}^n}(r)&=r^{-\beta-1}r^4\Big(-\beta \int_{B_1^{{\mathbb{H}^n}}(0)}\frac{\mid\nabla_{\mathbb{H}^n} u_1(\delta_r(\kappa))\mid^2}{|\kappa|_{\mathbb{H}^n}^{Q-2}}d\kappa\hspace{0.3cm} \int_{B_1^{{\mathbb{H}^n}}(0)}\frac{\mid\nabla_{\mathbb{H}^n} u_2(\delta_r(\kappa))\mid^2}{|\kappa|_{\mathbb{H}^n}^{Q-2}}d\kappa,\\
&+\int_{\partial B_1^{{\mathbb{H}^n}}(0)}\frac{\mid\nabla_{\mathbb{H}^n} u_1(\delta_r(\kappa))\mid^2}{\sqrt{\mid \kappa_x\mid^2+\mid \kappa_y\mid^2}}d\sigma_{\mathbb{H}^n}(\kappa)\hspace{0.3cm}\int_{B_1^{{\mathbb{H}^n}}(0)}\frac{\mid\nabla_{\mathbb{H}^n} u_2(\delta_r(\kappa))\mid^2}{|\kappa|_{\mathbb{H}^n}^{Q-2}}d\kappa\\
&+\int_{B_1^{{\mathbb{H}^n}}(0)}\frac{\mid\nabla_{\mathbb{H}^n} u_1(\delta_r(\kappa))\mid^2}{|\kappa|_{\mathbb{H}^n}^{Q-2}}d\kappa \hspace{0.3cm}\int_{\partial B_1^{{\mathbb{H}^n}}(0)}\frac{\mid\nabla_{\mathbb{H}^n}u_2(\delta_r(\kappa))\mid^2}{\sqrt{\mid \kappa_x\mid^2+\mid \kappa_y\mid^2}}d\sigma_{\mathbb{H}^n}(\kappa)\Big).\\
 \end{split}
 \end{equation}
 Let now $v_i(\kappa)=\frac{u_i(\delta_r(\kappa))}{r},$ $i=1,2.$ Then $\nabla_{\mathbb{H}^n} v_i(\kappa)=(\nabla_{\mathbb{H}^n}u_i)(\delta_r(\kappa)).$
 Hence, by \eqref{derivative-of-J-beta-H-^-n-1},
 
 \begin{equation}\label{derivative-of-J-beta-H-^-n-2}
\begin{split}
J_{\beta,\mathbb{H}^n}'(r)=&r^{-\beta+3}\Big(-\beta \int_{B_1^{{\mathbb{H}^n}}(0)}\frac{\mid\nabla_{\mathbb{H}^n} v_1(\kappa)\mid^2}{|\kappa|_{\mathbb{H}^n}^{Q-2}}d\kappa\hspace{0.3cm} \int_{B_1^{{\mathbb{H}^n}}(0)}\frac{\mid\nabla_{\mathbb{H}^n} v_2(\kappa)\mid^2}{|\kappa|_{\mathbb{H}^n}^{Q-2}}d\kappa\\
&+\int_{\partial B_1^{{\mathbb{H}^n}}(0)}\frac{\mid\nabla_{\mathbb{H}^n} v_1(\kappa)\mid^2}{\sqrt{\mid \kappa_x\mid^2+\mid \kappa_y\mid^2}}d\sigma_{\mathbb{H}^n}(\kappa)\hspace{0.3cm}\int_{B_1^{{\mathbb{H}^n}}(0)}\frac{\mid\nabla_{\mathbb{H}^n} v_2(\kappa)\mid^2}{|\kappa|_{\mathbb{H}^n}^{Q-2}}d\kappa\\
&+\int_{B_1^{{\mathbb{H}^n}}(0)}\frac{\mid\nabla_{\mathbb{H}^n} v_1(\kappa)\mid^2}{|\kappa|_{\mathbb{H}^n}^{Q-2}}d\kappa \hspace{0.3cm}\int_{\partial B_1^{{\mathbb{H}^n}}(0)}\frac{\mid\nabla_{\mathbb{H}^n}v_2(\kappa)\mid^2}{\sqrt{\mid \kappa_x\mid^2+\mid \kappa_y\mid^2}}d\sigma_{\mathbb{H}^n}(\kappa)\Big).
\end{split}
 \end{equation}
 
 In particular, $J_{\beta,\mathbb{H}^n}'(r)r^{\beta-3}=J_{\beta,\mathbb{H}^n}'(1),$ thus it is enough to prove that $J_{\beta,\mathbb{H}^n}'(1)\geq 0.$
 Moreover, using \eqref{derivative-of-J-beta-H-^-n-2} and \eqref{expression-of-J-beta-H-^-n}, we have
 $$
 \frac{J_{\beta,\mathbb{H}^n}'(1)}{J_{\beta,\mathbb{H}^n}(1)}=\frac{\displaystyle\int_{\partial B_1^{{\mathbb{H}^n}}(0)}\frac{\mid\nabla_{\mathbb{H}^n} v_1(\kappa)\mid^2}{\sqrt{\mid \kappa_x\mid^2+\mid \kappa_y\mid^2}}d\sigma_{\mathbb{H}^n}(\kappa)}{\displaystyle\int_{B_1^{{\mathbb{H}^n}}(0)}\frac{\mid\nabla_{\mathbb{H}^n} v_1(\kappa)\mid^2}{|\kappa|_{\mathbb{H}^n}^{Q-2}}d\kappa}+\frac{\displaystyle\int_{\partial B_1^{{\mathbb{H}^n}}(0)}\frac{\mid\nabla_{\mathbb{H}^n} v_2(\kappa)\mid^2}{\sqrt{\mid \kappa_x\mid^2+\mid \kappa_y\mid^2}}d\sigma_{\mathbb{H}^n}(\kappa)}{\displaystyle\int_{B_1^{{\mathbb{H}^n}}(0)}\frac{\mid\nabla_{\mathbb{H}^n} v_2(\kappa)\mid^2}{|\kappa|_{\mathbb{H}^n}^{Q-2}}d\kappa}-\beta.
 $$
  \end{proof}
 In the next section, we reduce ourselves to the simplest case given by $\mathbb{H}^1.$ 
  \section{Laplace--Kohn operator on $\mathbb{H}^1$ Koranyi ball boundary}\label{laplace_beltrami}
  
  This section is devoted to represent the Laplace-Kohn operator in spherical coordinates in the Heisenberg group. An analogous computation has been faced in \cite{Jerison} by using an abstract and more elegant approach, even if very theoretical, see also \cite{Greiner} and \cite{Biri}. Here we describe with explicit computations  the $\mathbb{H}^1$ case.
  
Specifically, we consider the following coordinates in $\mathbb{H}^{1}:$ 
\begin{equation}\label{pol-coord}
\begin{cases}
x=\rho \sqrt{\sin \varphi}\cos\theta \\
y=\rho \sqrt{\sin \varphi}\sin \theta\\
t=\rho^{2}\cos \varphi,
\end{cases}
\end{equation}
see \cite{Greiner}.

They mimic the classical polar coordinates in $\mathbb{R}^3.$
From \eqref{pol-coord}, we obtain the expression of $\rho,$ $\varphi$ and $\theta$ with respect to the cartesian coordinates $x,$ $y$ and $t,$ that is: 
\begin{equation}\label{cartes-coord}
\begin{cases}
\rho=((x^{2}+y^{2})^{2}+t^{2})^{1/4}\\
\theta=\arctan\left({\frac{y}{x}}\right)\\
\varphi=\arccos\left(\frac{t}{\rho^{2}}\right).
\end{cases}
\end{equation}
Recalling the vector fields
\begin{equation}\label{heisen-vect-fields}
\begin{cases}
X=\frac{\partial}{\partial x}+2y\frac{\partial}{\partial t}\\
Y=\frac{\partial}{\partial y}-2x\frac{\partial}{\partial t},
\end{cases}
\end{equation}
and the operators:
\begin{equation}\label{grad-lapl-horiz}
\nabla_{\mathbb{H}^1}\equiv(X,Y),\quad \Delta_{\mathbb{H}^1}=X^{2}+Y^{2},
\end{equation}
we want to determine the following: $\nabla_{\mathbb{H}^1}\rho,$ $\nabla_{\mathbb{H}^1}\theta,$ $\nabla_{\mathbb{H}^1}\varphi,$ using \eqref{heisen-vect-fields}, \eqref{cartes-coord} and \eqref{grad-lapl-horiz}.
\begin{lem}\label{lemma-nabla-horiz-rho-theta-phi}
It results that:
\begin{align*}
&\nabla_{\mathbb{H}^{1}}\varphi=\frac{2}{\rho(x^{2}+y^{2})}\left(t\nabla_{\mathbb{H}^{1}}\rho+\rho(-y,x)\right)\nonumber,\\
&\nabla_{\mathbb{H}^{1}}\rho=\rho^{-3}((x^{2}+y^{2})x+ty,(x^{2}+y^{2})y-tx)\nonumber,\\
&\nabla_{\mathbb{H}^{1}}\theta=\frac{1}{x^{2}+y^{2}}(-y,x).
\end{align*}
\end{lem}
\begin{proof}
Let us begin by calculating:
\begin{align*}
X\varphi&=X\left(\arccos\left(\frac{t}{\rho^{2}}\right)\right)=-\frac{1}{\sqrt{1-\left(\frac{t}{\rho^{2}}\right)^{2}}}X\left(\frac{t}{\rho^{2}}\right)=-\frac{1}{\sqrt{1-\frac{t^{2}}{\rho^{4}}}}\left(\frac{2y}{\rho^{2}}-2\rho^{-3}tX\rho\right)\\
&=-\frac{2}{\sqrt{\frac{\rho^{4}-t^{2}}{\rho^{4}}}}\frac{1}{\rho^{2}}\left(y-\frac{tX\rho}{\rho}\right)=-\frac{2\rho^{2}}{\sqrt{\rho^{4}-t^{2}}}\frac{1}{\rho^{2}}\left(\frac{\rho y-tX\rho}{\rho}\right)\\
&=-\frac{2}{\rho\sqrt{\rho^{4}-t^{2}}}\left(\rho y-tX\rho\right)=\frac{2}{\rho\sqrt{\rho^{4}-t^{2}}}\left(tX\rho-\rho y\right)
\end{align*}
and 
\begin{align*}
Y\varphi&=-\frac{1}{\sqrt{1-\left(\frac{t}{\rho^{2}}\right)^{2}}}Y\left(\frac{t}{\rho^{2}}\right)=-\frac{\rho^2}{\sqrt{\rho^{4}-t^2}}\left(-\frac{2x}{\rho^{2}}-2\rho^{-3}tY\rho\right)=-\frac{2\rho^2}{\sqrt{\rho^{4}-t^2}}\frac{1}{\rho^{2}}\left(-x-\frac{tY\rho}{\rho}\right)\\
&=-\frac{2}{\sqrt{\rho^{4}-t^{2}}}\left(\frac{-x\rho-tY\rho}{\rho}\right)=\frac{2}{\rho\sqrt{\rho^{4}-t^{2}}}\left(x\rho+tY\rho\right),
\end{align*}
which give
\begin{align}\label{grad-horiz-varphi}
\nabla_{\mathbb{H}^{1}}\varphi&=(X\varphi,Y\varphi)=\left(\frac{2}{\rho\sqrt{\rho^{4}-t^{2}}}\left(tX\rho-\rho y\right),\frac{2}{\rho\sqrt{\rho^{4}-t^{2}}}\left(x\rho+tY\rho\right)\right)\nonumber\\
&=\frac{2}{\rho\sqrt{\rho^{4}-t^{2}}}\left(tX\rho-\rho y,x\rho+tY\rho\right)=\frac{2}{\rho\sqrt{(x^{2}+y^{2})^{2}+t^{2}-t^{2}}}\left(\left(tX\rho,tY\rho\right)+(-\rho y,\rho x)\right)\nonumber\\
&=\frac{2}{\rho(x^{2}+y^{2})}\left(t\nabla_{\mathbb{H}^{1}}\rho+\rho(-y,x)\right).
\end{align}
At this point, we compute
\begin{align*}
X\rho&=X(((x^{2}+y^{2})^{2}+t^{2})^{1/4})=\frac{1}{4}((x^{2}+y^{2})^{2}+t^{2})^{\frac{1}{4}-1}X((x^{2}+y^{2})^{2}+t^{2})\\
&=\frac{1}{4}((x^{2}+y^{2})^{2}+t^{2})^{-3/4}(2(x^{2}+y^{2})2x+2y2t)=\rho^{-3}((x^{2}+y^{2})x+yt)
\end{align*}
and
\begin{align*}
Y\rho&=Y(((x^{2}+y^{2})^{2}+t^{2})^{1/4})=\frac{1}{4}((x^{2}+y^{2})^{2}+t^{2})^{-3/4}Y((x^{2}+y^{2})^{2}+t^{2})\\
&=\frac{1}{4}((x^{2}+y^{2})^{2}+t^{2})^{-3/4}(2(x^{2}+y^{2})2y-2x2t)=\rho^{-3}((x^{2}+y^{2})y-xt),
\end{align*}
which entail
\begin{align}\label{grad-horiz-rho}
\nabla_{\mathbb{H}^{1}}\rho=(X\rho,Y\rho)=\rho^{-3}((x^{2}+y^{2})x+yt,(x^{2}+y^{2})y-xt).
\end{align}
Let us calculate now $\nabla_{\mathbb{H}^{1}}\theta.$ For this purpose, we have
\[
X\theta=X\left(\arctan\left(\frac{y}{x}\right)\right)=\frac{1}{1+\left(\frac{y}{x}\right)^{2}}X\left(\frac{y}{x}\right)=\frac{1}{1+\frac{y^{2}}{x^{2}}}\left(-\frac{y}{x^{2}}\right)=-\frac{x^{2}}{x^{2}+y^{2}}\frac{y}{x^{2}}=-\frac{y}{x^{2}+y^{2}}
\]
and
$$Y\theta=\frac{1}{1+\left(\frac{y}{x}\right)^{2}}Y\left(\frac{y}{x}\right)=\frac{x^{2}}{x^{2}+y^{2}}\left(\frac{1}{x}\right)=\frac{x}{x^{2}+y^{2}},$$
which imply
\begin{equation}\label{grad-horiz-theta}
\nabla_{\mathbb{H}^{1}}\theta=(X\theta,Y\theta)=\frac{1}{x^{2}+y^{2}}(-y,x).
\end{equation}
Using \eqref{grad-horiz-varphi}, \eqref{grad-horiz-rho} and \eqref{grad-horiz-theta}, we achieve
\begin{align}\label{grad-horiz-varphi-rho-theta}
&\nabla_{\mathbb{H}^{1}}\varphi=\frac{2}{\rho(x^{2}+y^{2})}\left(t\nabla_{\mathbb{H}^{1}}\rho+\rho(-y,x)\right)\nonumber,\\
&\nabla_{\mathbb{H}^{1}}\rho=\rho^{-3}((x^{2}+y^{2})x+ty,(x^{2}+y^{2})y-tx)\nonumber,\\
&\nabla_{\mathbb{H}^{1}}\theta=\frac{1}{x^{2}+y^{2}}(-y,x).
\end{align}
\end{proof}
\begin{lem}\label{norm-quad-nabla-horiz-theta-phi-rho}
The following relationships hold:
\begin{equation*}
\left|\nabla_{\mathbb{H}^{1}}\varphi\right|^{2}=\frac{4(x^{2}+y^{2})}{\rho^{4}},\quad
\left|\nabla_{\mathbb{H}^{1}}\rho\right|^{2}=\frac{x^{2}+y^{2}}{\rho^{2}},\quad
\left|\nabla_{\mathbb{H}^{1}}\theta\right|^{2}=\frac{1}{x^{2}+y^{2}}.
\end{equation*}
\end{lem}
\begin{proof}
From \eqref{grad-horiz-varphi-rho-theta}, we obtain:
\begin{align*}
\left|\nabla_{\mathbb{H}^{1}}\theta\right|^{2}&=\left|\frac{1}{x^{2}+y^{2}}(-y,x)\right|^{2}=\frac{1}{(x^{2}+y^{2})^{2}}(y^{2}+x^{2})=\frac{1}{x^{2}+y^{2}},\\
\left|\nabla_{\mathbb{H}^{1}}\rho\right|^{2}&=\left|\rho^{-3}((x^{2}+y^{2})x+ty,(x^{2}+y^{2})y-tx)\right|^{2}=\rho^{-6}(((x^{2}+y^{2})x+ty)^{2}+((x^{2}+y^{2})y-tx)^{2})\\
&=\rho^{-6}((x^{2}+y^{2})^{2}x^{2}+t^{2}y^{2}+2(x^{2}+y^{2})xty+(x^{2}+y^{2})y^{2}+t^{2}x^{2}-2(x^{2}+y^{2})ytx)\\
&=\rho^{-6}((x^{2}+y^{2})^{2}(x^{2}+y^{2})+t^{2}(x^{2}+y^{2}))=\rho^{-6}((x^{2}+y^{2})((x^{2}+y^{2})^{2}+t^{2})\\
&=\rho^{-6}(x^{2}+y^{2})\rho^{4}=\frac{x^{2}+y^{2}}{\rho^{2}},
\end{align*}
i.e.
\begin{align}\label{norm-quad-grad-horiz-theta-rho}
&\left|\nabla_{\mathbb{H}^{1}}\theta\right|^{2}=\frac{1}{x^{2}+y^{2}},
&\left|\nabla_{\mathbb{H}^{1}}\rho\right|^{2}=\frac{x^{2}+y^{2}}{\rho^{2}},
\end{align}
and, using \eqref{norm-quad-grad-horiz-theta-rho} and Lemma \ref{lemma-nabla-horiz-rho-theta-phi},
\begin{align*}
\left|\nabla_{\mathbb{H}^{1}}\varphi\right|^{2}&=\left|\frac{2}{\rho(x^{2}+y^{2})}(t\nabla_{\mathbb{H}^{1}}\rho+\rho(-y,x))\right|^{2}=\left(\frac{2}{\rho(x^{2}+y^{2})}\right)^{2}\langle t\nabla_{\mathbb{H}^{1}}\rho+\rho(-y,x),t\nabla_{\mathbb{H}^{1}}\rho+\rho(-y,x)\rangle\\
&=\frac{4}{\rho^{2}(x^{2}+y^{2})^{2}}(t^{2}\left|\nabla_{\mathbb{H}^{1}}\rho\right|^{2}+\rho^{2}(y^{2}+x^{2})+2t\rho\langle\nabla_{\mathbb{H}^{1}}\rho,(-y,x)\rangle)\\
&=\frac{4}{\rho^{2}(x^{2}+y^{2})^{2}}\bigg(t^{2}\frac{x^{2}+y^{2}}{\rho^{2}}+\rho^{2}(y^{2}+x^{2})+2t\rho\langle\rho^{-3}((x^{2}+y^{2})x+yt,(x^{2}+y^{2})y-xt),(-y,x)\rangle\bigg)\\
&=\frac{4}{\rho^{2}(x^{2}+y^{2})^{2}}\bigg(t^{2}\frac{x^{2}+y^{2}}{\rho^{2}}+\rho^{2}(y^{2}+x^{2})+2t\rho\rho^{-3}(-(x^{2}+y^{2})yx-y^{2}t+(x^{2}+y^{2})yx-x^{2}t)\bigg)\\
&=\frac{4}{\rho^{2}(x^{2}+y^{2})^{2}}\bigg(t^{2}\frac{x^{2}+y^{2}}{\rho^{2}}+\rho^{2}(y^{2}+x^{2})-2t^{2}\frac{x^{2}+y^{2}}{\rho^{2}}\bigg)\\
&=\frac{4}{\rho^{2}(x^{2}+y^{2})^{2}}\bigg(\rho^{2}(x^{2}+y^{2})-t^{2}\frac{x^{2}+y^{2}}{\rho^{2}}\bigg)=\frac{4}{\rho^{2}(x^{2}+y^{2})}\bigg(\frac{\rho^{4}-t^{2}}{\rho^{2}}\bigg)\\
&=\frac{4(x^{2}+y^{2})^2}{\rho^4(x^{2}+y^{2})}=\frac{4(x^{2}+y^{2})}{\rho^{4}},
\end{align*}
namely
\begin{equation}\label{norm-quad-grad-horiz-varphi}
\left|\nabla_{\mathbb{H}^{1}}\varphi\right|^{2}=\frac{4(x^{2}+y^{2})}{\rho^{4}}.
\end{equation}
Putting together \eqref{norm-quad-grad-horiz-theta-rho} and \eqref{norm-quad-grad-horiz-varphi}, we get the expected results:
\begin{equation}\label{norm-quad-grad-horiz-varphi-rho-theta}
\left|\nabla_{\mathbb{H}^{1}}\varphi\right|^{2}=\frac{4(x^{2}+y^{2})}{\rho^{4}},\quad
\left|\nabla_{\mathbb{H}^{1}}\rho\right|^{2}=\frac{x^{2}+y^{2}}{\rho^{2}},\quad
\left|\nabla_{\mathbb{H}^{1}}\theta\right|^{2}=\frac{1}{x^{2}+y^{2}}.
\end{equation}
\end{proof}
\begin{lem}\label{lemma-inner-product-nabla-phi-rho-theta}
The following relationships hold:
$$\langle\nabla_{\mathbb{H}^{1}}\varphi,\nabla_{\mathbb{H}^{1}}\rho\rangle=0,\quad \langle\nabla_{\mathbb{H}^{1}}\rho,\nabla_{\mathbb{H}^{1}}\theta\rangle=-\frac{\cos\varphi}{\rho},\quad \langle\nabla_{\mathbb{H}^{1}}\varphi,\nabla_{\mathbb{H}^{1}}\theta\rangle=\frac{2(x^{2}+y^{2})}{\rho^{4}}.$$
 \end{lem}
 \begin{proof}
Let us calculate now $\langle\nabla_{\mathbb{H}^{1}}\varphi, \nabla_{\mathbb{H}^{1}}\rho\rangle,$ $\langle\nabla_{\mathbb{H}^{1}}\varphi, \nabla_{\mathbb{H}^{1}}\theta\rangle$ and $\langle\nabla_{\mathbb{H}^{1}}\rho, \nabla_{\mathbb{H}^{1}}\theta\rangle,$ using \eqref{grad-horiz-varphi-rho-theta} and \eqref{norm-quad-grad-horiz-varphi-rho-theta}. Specifically, we have
\begin{equation*}
\begin{split}
&\langle\nabla_{\mathbb{H}^{1}}\varphi, \nabla_{\mathbb{H}^{1}}\rho\rangle=\langle\frac{2}{\rho(x^{2}+y^{2})}(t\nabla_{\mathbb{H}^{1}}\rho+\rho(-y,x)),\nabla_{\mathbb{H}^{1}}\rho\rangle=\frac{2}{\rho(x^{2}+y^{2})}(t\left|\nabla_{\mathbb{H}^{1}}\rho\right|^{2}+\rho\langle(-y,x),\nabla_{\mathbb{H}^{1}}\rho\rangle)\\
&=\frac{2}{\rho(x^{2}+y^{2})}\bigg(t\frac{x^{2}+y^{2}}{\rho^{2}}+\rho\langle(-y,x),\rho^{-3}((x^{2}+y^{2})x+yt,(x^{2}+y^{2})y-xt)\rangle\bigg)\\
&=\frac{2}{\rho(x^{2}+y^{2})}\left(t\frac{x^{2}+y^{2}}{\rho^{2}}+\rho^{-2}(-(x^{2}+y^{2})yx-y^{2}t+(x^{2}+y^{2})xy-x^{2}t)\right)\\
&=\frac{2}{\rho(x^{2}+y^{2})}\left(t\frac{x^{2}+y^{2}}{\rho^{2}}+\rho^{-2}(-t(x^{2}+y^{2}))\right)=0,
\end{split}
\end{equation*}
and
\begin{equation*}
\begin{split}
&\langle\nabla_{\mathbb{H}^{1}}\rho, \nabla_{\mathbb{H}^{1}}\theta\rangle=\langle\rho^{-3}((x^{2}+y^{2})x+yt,(x^{2}+y^{2})y-xt,\frac{1}{x^{2}+y^{2}}(-y,x)\rangle\\
&=\frac{\rho^{-3}}{x^{2}+y^{2}}(-(x^{2}+y^{2})xy-y^{2}t+(x^{2}+y^{2})xy-x^{2}t)=-\frac{t}{\rho^{3}}=-\frac{\rho^{2}\cos\varphi}{\rho^{3}}=-\frac{\cos\varphi}{\rho}.
\end{split}
\end{equation*}
Thus, it results
\begin{equation}\label{prod-scal-grad-varphi-grad-rho-grad-rho-grad-theta}
\langle\nabla_{\mathbb{H}^{1}}\varphi,\nabla_{\mathbb{H}^{1}}\rho\rangle=0,\quad \langle\nabla_{\mathbb{H}^{1}}\rho,\nabla_{\mathbb{H}^{1}}\theta\rangle=-\frac{\cos\varphi}{\rho},
\end{equation}
and, using \eqref{prod-scal-grad-varphi-grad-rho-grad-rho-grad-theta},
\begin{align*}
&\langle\nabla_{\mathbb{H}^{1}}\varphi, \nabla_{\mathbb{H}^{1}}\theta\rangle=\langle\frac{2}{\rho(x^{2}+y^{2})}(t\nabla_{\mathbb{H}^{1}}\rho+\rho(-y,x)),\frac{1}{x^{2}+y^{2}}(-y,x)\rangle\\
&=\frac{2}{\rho(x^{2}+y^{2})^{2}}(t\langle\rho^{-3}((x^{2}+y^{2})x+yt,(x^{2}+y^{2})y-xt),(-y,x)\rangle+\rho(y^{2}+x^{2}))\\
&=\frac{2}{\rho(x^{2}+y^{2})^{2}}(t\rho^{-3}(-(x^{2}+y^{2})xy-y^{2}t+(x^{2}+y^{2})xy-x^{2}t)+\rho(y^{2}+x^{2}))\\
&=\frac{2}{\rho(x^{2}+y^{2})^{2}}(-t^2\rho^{-3}(x^{2}+y^{2})+\rho(y^{2}+x^{2}))=\frac{2}{\rho(x^{2}+y^{2})}(-t^{2}\rho^{-3}+\rho)\\
&=\frac{2}{\rho(x^{2}+y^{2})}\left(\frac{-t^{2}+\rho^{4}}{\rho^{3}}\right)=\frac{2(x^{2}+y^{2})^{2}}{\rho^4(x^{2}+y^{2})}=\frac{2(x^{2}+y^{2})}{\rho^{4}},
\end{align*}
that is
\begin{equation}\label{prod-scal-grad-varphi-grad-theta}
\langle\nabla_{\mathbb{H}^{1}}\varphi,\nabla_{\mathbb{H}^{1}}\theta\rangle=\frac{2(x^{2}+y^{2})}{\rho^{4}}.
\end{equation}
Considering together \eqref{prod-scal-grad-varphi-grad-rho-grad-rho-grad-theta} and \eqref{prod-scal-grad-varphi-grad-theta}, we obtain
\begin{equation}\label{prod-scal-grad-varphi-grad-rho-grad-rho-grad-theta-grad-varphi-grad-theta}
\langle\nabla_{\mathbb{H}^{1}}\varphi,\nabla_{\mathbb{H}^{1}}\rho\rangle=0,\quad
\langle\nabla_{\mathbb{H}^{1}}\varphi, \nabla_{\mathbb{H}^{1}}\theta\rangle=\frac{2(x^{2}+y^{2})}{\rho^{4}},\quad
\langle\nabla_{\mathbb{H}^{1}}\rho, \nabla_{\mathbb{H}^{1}}\theta\rangle=-\frac{\cos\varphi}{\rho}.
\end{equation}
\end{proof}
At this point, we are in position to compute $\Delta_{\mathbb{H}^1}\varphi,$ $\Delta_{\mathbb{H}^1}\rho$ and $\Delta_{\mathbb{H}^1}\theta,$ using \eqref{grad-horiz-varphi-rho-theta}.
In particular, we have the following result.
\begin{lem}
The following relationships hold:
$$
\Delta_{\mathbb{H}^1}\theta=0, \quad \Delta_{\mathbb{H}^1}\rho=\frac{3(x^{2}+y^{2})}{\rho^{3}},\quad \Delta_{\mathbb{H}^1}\varphi=\frac{4\cos \varphi}{\rho^{2}}.
$$
\end{lem}
\begin{proof}
We remark that:
\begin{align*}
\Delta_{\mathbb{H}^1}\theta&=(X^{2}+Y^{2})\theta=X^{2}\theta+Y^{2}\theta=X(X\theta)+Y(Y\theta)=X\left(-\frac{y}{x^{2}+y^{2}}\right)+Y\left(\frac{x}{x^{2}+y^{2}}\right)\\
&=-y\left(-\frac{2x}{(x^{2}+y^{2})^{2}}\right)+x\left(-\frac{2y}{(x^{2}+y^{2})^{2}}\right)=\frac{2xy}{(x^{2}+y^{2})^{2}}-\frac{2xy}{(x^{2}+y^{2})^{2}}=0,
\end{align*}
which entails
\begin{equation}\label{lapl-horiz-theta}
\Delta_{\mathbb{H}^1}\theta=0.
\end{equation}
Concerning $\Delta_{\mathbb{H}^1}\rho,$ we can use the following formula, see \cite{BLU},
\begin{equation}\label{lapl-horiz-rho-functions}
\Delta_{\mathbb{H}^1}f(\rho)=\left|\nabla_{\mathbb{H}^{1}}\rho\right|^{2}\left(f''+\frac{Q-1}{\rho}f'\right),
\end{equation}
in the particular case of $f(\rho)=\rho$ and $Q=4,$ and we achieve, in view of \eqref{norm-quad-grad-horiz-varphi-rho-theta},
$$\Delta_{\mathbb{H}^1}\rho=\left|\nabla_{\mathbb{H}^{1}}\rho\right|^{2}\frac{3}{\rho}=\frac{x^{2}+y^{2}}{\rho^{2}}\frac{3}{\rho}=\frac{3(x^{2}+y^{2})}{\rho^{3}},$$
that is 
\begin{equation}\label{lapl-horiz-rho-1}
\Delta_{\mathbb{H}^1}\rho=\frac{3(x^{2}+y^{2})}{\rho^{3}}.
\end{equation}
As regards $\Delta_{\mathbb{H}^1}\varphi,$ we obtain, because from \eqref{grad-horiz-varphi-rho-theta} we have
$$\nabla_{\mathbb{H}^{1}}\varphi=\frac{2}{\rho(x^{2}+y^{2})}(t\nabla_{\mathbb{H}^{1}}\rho+\rho(-y,x))=\frac{2}{\rho(x^{2}+y^{2})}(tX\rho-\rho y,tY\rho+\rho x),$$
and by virtue of \eqref{lapl-horiz-rho-1} and \eqref{norm-quad-grad-horiz-varphi-rho-theta},
\begin{align*}
&\Delta_{\mathbb{H}^1}\varphi=X(X\varphi)+Y(Y\varphi)=X\bigg(\frac{2}{\rho(x^{2}+y^{2})}(tX\rho-\rho y)\bigg)+Y\bigg(\frac{2}{\rho(x^{2}+y^{2})}(tY\rho+\rho x)\bigg)\\
&=X\left(\frac{2}{\rho(x^{2}+y^{2})}\right)(tX\rho-\rho y)+\frac{2}{\rho(x^{2}+y^{2})}(XtX\rho+tX^{2}\rho-(X\rho)y-\rho Xy)\\
&+Y\left(\frac{2}{\rho(x^{2}+y^{2})}\right)(tY\rho+\rho x)+\frac{2}{\rho(x^{2}+y^{2})}(YtY\rho+tY^{2}\rho+(Y\rho) x+\rho Yx),
\end{align*}
so that
\begin{align*}
&\Delta_{\mathbb{H}^1}\varphi=(tX\rho-\rho y)X\bigg(\frac{2}{\rho(x^{2}+y^{2})}\bigg)+(tY\rho+\rho x)Y\bigg(\frac{2}{\rho(x^{2}+y^{2})}\bigg)\\
&+\frac{2t}{\rho(x^{2}+y^{2})}(X^{2}\rho+Y^{2}\rho)+\frac{2}{\rho(x^{2}+y^{2})}((Xt-y)X\rho-\rho Xy+(Yt+x)Y\rho+\rho Yx)\\
&=(tX\rho-\rho y)\bigg(-\frac{2}{\rho^{2}(x^{2}+y^{2})^{2}}\bigg)\bigg(\frac{\partial\rho}{\partial x}(x^{2}+y^{2})+\rho2x+2y\frac{\partial \rho}{\partial t}(x^{2}+y^{2})\bigg)\\
&+(tY\rho+\rho x)\bigg(-\frac{2}{\rho^{2}(x^{2}+y^{2})^{2}}\bigg)\bigg(\frac{\partial \rho}{\partial y}(x^{2}+y^{2})+\rho2y-2x\frac{\partial \rho}{\partial t}(x^{2}+y^{2})\bigg)\\
&+\frac{2t}{\rho(x^{2}+y^{2})}\Delta_{\mathbb{H}^1}\rho+\frac{2}{\rho(x^{2}+y^{2})}((2y-y)X\rho+(-2x+x)Y\rho)\\
&=-2\frac{(tX\rho-\rho y)((x^{2}+y^{2})X\rho+2\rho x)}{\rho^{2}(x^{2}+y^{2})^{2}}-2\frac{(tY\rho+\rho x)((x^{2}+y^{2})Y\rho+2\rho y)}{\rho^{2}(x^{2}+y^{2})^{2}}\\
&+\frac{2t}{\rho(x^{2}+y^{2})}\frac{3(x^{2}+y^{2})}{\rho^{3}}+\frac{2}{\rho(x^{2}+y^{2})}(yX\rho-xY\rho).\\
\end{align*}
Furthermore, continuing the computation, we actually get
\begin{align*}
&\Delta_{\mathbb{H}^1}\varphi=-\frac{2}{\rho^{2}(x^{2}+y^{2})^{2}}(t(x^{2}+y^{2})(X\rho)^{2}-\rho y(x^{2}+y^{2})X\rho+2\rho xtX\rho-2\rho^{2}xy+t(x^{2}+y^{2})(Y\rho)^{2}\\
&+\rho x(x^{2}+y^{2})Y\rho+2\rho ytY\rho+2\rho^{2}xy)+\frac{6t}{\rho^{4}}+\frac{2y}{\rho(x^{2}+y^{2})}X\rho-\frac{2x}{\rho(x^{2}+y^{2})}Y\rho\\
&=-\frac{2}{\rho^{2}(x^{2}+y^{2})^{2}}(t(x^{2}+y^{2})((X\rho)^{2}+(Y\rho)^{2})+\rho(x^{2}+y^{2})(xY\rho-yX\rho)+2t\rho(xX\rho+yY\rho))\\
&+\frac{6t}{\rho^{4}}+\frac{2y}{\rho(x^{2}+y^{2})}X\rho-\frac{2x}{\rho(x^{2}+y^{2})}Y\rho\\
&=-\frac{2t}{\rho^{2}(x^{2}+y^{2})}\left|\nabla_{\mathbb{H}^{1}} \rho\right|^{2}-\frac{2x}{\rho(x^{2}+y^{2})}Y\rho+\frac{2y}{\rho(x^{2}+y^{2})}X\rho-\frac{4t}{\rho(x^{2}+y^{2})^{2}}(xX\rho+yY\rho)+\frac{6t}{\rho^{4}}\\
&+\frac{2y}{\rho(x^{2}+y^{2})}X\rho-\frac{2x}{\rho(x^{2}+y^{2})}Y\rho\\
&=-\frac{2t}{\rho^{2}(x^{2}+y^{2})}\left|\nabla_{\mathbb{H}^{1}} \rho\right|^{2}+\frac{4}{\rho(x^{2}+y^{2})}(yX\rho-xY\rho)-\frac{4t}{\rho(x^{2}+y^{2})^{2}}(xX\rho+yY\rho)+\frac{6t}{\rho^{4}}\\
&=-\frac{2t}{\rho^{2}(x^{2}+y^{2})}\frac{x^{2}+y^{2}}{\rho^{2}}+\frac{4}{\rho(x^{2}+y^{2})}(y\rho^{-3}((x^{2}+y^{2})x+ty)-x\rho^{-3}((x^{2}+y^{2})y-xt))\\
&-\frac{4t}{\rho(x^{2}+y^{2})^{2}}(x\rho^{-3}((x^{2}+y^{2})x+yt)+y\rho^{-3}((x^{2}+y^{2})y-xt))+\frac{6t}{\rho^{4}}
\end{align*}
which finally implies
\begin{align*}
&\Delta_{\mathbb{H}^1}\varphi=-\frac{2t}{\rho^{4}}+\frac{4}{\rho^4(x^{2}+y^{2})}((x^{2}+y^{2})yx+ty^{2}-(x^{2}+y^{2})xy+x^{2}t)\\
&-\frac{4t}{\rho^4(x^{2}+y^{2})^{2}}((x^{2}+y^{2})x^{2}+xyt+(x^{2}+y^{2})y^{2}-yxt)+\frac{6t}{\rho^{4}}\\
&=\frac{4t}{\rho^{4}}+\frac{4}{\rho^{4}(x^{2}+y^{2})}t(x^{2}+y^{2})-\frac{4t}{\rho^{4}(x^{2}+y^{2})^{2}}(x^{2}+y^{2})^2=\frac{4t}{\rho^{4}}=\frac{4\rho^{2}\cos\varphi}{\rho^{4}}=\frac{4\cos\varphi}{\rho^{2}},
\end{align*}
i.e.
\begin{equation}\label{lapl-horiz-phi}
\Delta_{\mathbb{H}^1}\varphi=\frac{4\cos \varphi}{\rho^{2}}.
\end{equation}
Putting together \eqref{lapl-horiz-theta}, \eqref{lapl-horiz-rho-1} and \eqref{lapl-horiz-phi}, we have
\begin{equation}
\label{lapl-horiz-varphi-rho-theta}
\Delta_{\mathbb{H}^1}\theta=0,\quad \Delta_{\mathbb{H}^1}\rho=\frac{3(x^{2}+y^{2})}{\rho^{3}}\quad
\Delta_{\mathbb{H}^1}\varphi=\frac{4\cos\varphi}{\rho^{2}}.
\end{equation}
\end{proof}
At this point, assuming that $u=\rho^{\alpha}f(\theta,\varphi),$  we compute $\Delta_{\mathbb{H}^1}u$ obtaining the following result.
\begin{lem}\label{lapl-horiz-f-theta-phi}
Let $u=\rho^{\alpha}f(\theta,\varphi)$. Then:
\begin{equation*}
\begin{split}
&\Delta_{\mathbb{H}^1}u=\Delta_{\mathbb{H}^1}(\rho^{\alpha}f(\theta,\varphi))=\rho^{\alpha-2}\bigg(\alpha(\alpha+2)(\sin\varphi)f(\theta,\varphi)-2\alpha(\cos\varphi)\frac{\partial f}{\partial \theta}+\frac{1}{\sin\varphi}\frac{\partial^{2}f}{\partial \theta^{2}}\\
&+4\sin\varphi\frac{\partial^{2}f}{\partial \varphi \partial \theta}+4\sin\varphi\frac{\partial^{2}f}{\partial \varphi^{2}}+4\cos\varphi\frac{\partial f}{\partial \varphi}\bigg).
\end{split}
\end{equation*}
\end{lem}

\begin{proof}
Let us begin the proof by computing $Xu$ and $Yu.$
We point out that $\theta=\theta(x,y,t),$ $\varphi=\varphi(x,y,t)$ and $\rho=\rho(x,y,t).$
As a consequence, we have
\begin{align}\label{grad-horiz-u-components}
&Xu=\alpha \rho^{\alpha-1}(X\rho)f(\theta,\varphi)+\rho^{\alpha}\bigg(\frac{\partial f}{\partial \theta}X\theta+\frac{\partial f}{\partial \varphi}X\varphi\bigg)\nonumber,\\
&Yu=\alpha \rho^{\alpha-1}(Y\rho)f(\theta,\varphi)+\rho^{\alpha}\bigg(\frac{\partial f}{\partial \theta}Y\theta+\frac{\partial f}{\partial \varphi}Y\varphi\bigg).
\end{align}
We now calculate $\Delta_{\mathbb{H}^1}u,$ using \eqref{grad-horiz-u-components}.
Specifically, we achieve, using also the computation done to find \eqref{grad-horiz-u-components},
\begin{align*}
&\Delta_{\mathbb{H}^1}u=X(Xu)+Y(Yu)\\
&=X\bigg(\alpha \rho^{\alpha-1}(X\rho)f(\theta,\varphi)+\rho^{\alpha}\bigg(\frac{\partial f}{\partial \theta}X\theta+\frac{\partial f}{\partial \varphi}X\varphi\bigg)\bigg)+Y\bigg(\alpha\rho^{\alpha-1}(Y\rho)f(\theta,\varphi)+\rho^{\alpha}\bigg(\frac{\partial f}{\partial \theta}Y\theta+\frac{\partial f}{\partial \varphi}Y\varphi\bigg)\bigg)\\
&=\alpha(\alpha-1)\rho^{\alpha-2}((X\rho)^{2}+(Y\rho)^{2})f(\theta,\varphi)+\alpha\rho^{\alpha-1}(X^{2}\rho+Y^{2}\rho)f(\theta,\varphi)+2\alpha\rho^{\alpha-1}(X\rho)\bigg(\frac{\partial f}{\partial \theta}X\theta+\frac{\partial f}{\partial \varphi}X\varphi\bigg)\\
&+2\alpha\rho^{\alpha-1}(Y\rho)\bigg(\frac{\partial f}{\partial \theta} Y\rho Y\theta+\frac{\partial f}{\partial \varphi}Y\rho Y\varphi\bigg)+\rho^{\alpha}\bigg(X\bigg(\frac{\partial f}{\partial \theta}X\theta+\frac{\partial f}{\partial \varphi}X\varphi\bigg)+Y\bigg(\frac{\partial f}{\partial \theta}Y\theta+\frac{\partial f}{\partial \varphi}Y\varphi\bigg)\bigg),\\
\end{align*}
so that it results
\begin{align*}
&\Delta_{\mathbb{H}^1}u=\alpha(\alpha-1)\rho^{\alpha-2}\left|\nabla_{\mathbb{H}^{1}}\rho\right|^{2}f(\theta,\varphi)+\alpha\rho^{\alpha-1}(\Delta_{\mathbb{H}^1}\rho) f(\theta,\varphi)+2\alpha\rho^{\alpha-1}\bigg(\frac{\partial f}{\partial \theta}X\rho X\theta+\frac{\partial f}{\partial \varphi}X\rho X\varphi\\
&+\frac{\partial f}{\partial \theta}Y\rho Y\theta+\frac{\partial f}{\partial \varphi} Y\rho Y\varphi\bigg)+\rho^{\alpha}\bigg(X\bigg(\frac{\partial f}{\partial \theta}\bigg)X\theta+\frac{\partial f}{\partial \theta}X^{2}\theta+X\bigg(\frac{\partial f}{\partial \varphi}\bigg)X\varphi+\frac{\partial f}{\partial \varphi}X^{2}\varphi
+Y\bigg(\frac{\partial f}{\partial \theta}\bigg)Y\theta\\
&+\frac{\partial f}{\partial \theta}Y^{2}\theta+Y\left(\frac{\partial f}{\partial \varphi}\right)Y\varphi+\frac{\partial f}{\partial \varphi}Y^{2}\varphi\bigg).\\
\end{align*}
Therefore, by continuing the computation, we have
\begin{align*}
&\Delta_{\mathbb{H}^1}u=(\alpha(\alpha-1)\rho^{\alpha-2}\left|\nabla_{\mathbb{H}^{1}}\rho\right|^{2}+\alpha \rho^{\alpha-1}\Delta_{\mathbb{H}^1}\rho)f(\theta,\varphi)+2\alpha\rho^{\alpha-1}\bigg(\frac{\partial f}{\partial \theta}X\rho X\theta+\frac{\partial f}{\partial \varphi}X\rho X\varphi\\
&+\frac{\partial f}{\partial \theta}Y\rho Y\theta+\frac{\partial f}{\partial \varphi} Y\rho Y\varphi\bigg)+\rho^{\alpha}\bigg(\bigg(\frac{\partial^{2}f}{\partial \theta^{2}}X\theta+\frac{\partial^{2}f}{\partial \varphi \partial \theta}X\varphi\bigg)X\theta+\bigg(\frac{\partial^{2}f}{\partial \theta \partial \varphi}X\theta+\frac{\partial^{2}f}{\partial \varphi^{2}}X\varphi\bigg)X\varphi\\
&+\bigg(\frac{\partial^{2}f}{\partial \theta^{2}}Y\theta+\frac{\partial^{2}f}{\partial \varphi \partial \theta}Y\varphi\bigg)Y\theta+\bigg(\frac{\partial^{2}f}{\partial \theta \partial \varphi}Y\theta+\frac{\partial^{2}f}{\partial \varphi^{2}}Y\varphi\bigg)Y\varphi+\frac{\partial f}{\partial \theta}(X^{2}\theta+Y^{2}\theta)+\frac{\partial f}{\partial \varphi}(X^{2}\varphi+Y^{2}\varphi)\bigg)\\
&=(\alpha(\alpha-1)\rho^{\alpha-2}\left|\nabla_{\mathbb{H}^{1}}\rho\right|^{2}+\alpha\rho^{\alpha-1}\Delta_{\mathbb{H}^1}\rho)f(\theta,\varphi)+2\alpha\rho^{\alpha-1}\bigg(\frac{\partial f}{\partial \theta}\langle\nabla_{\mathbb{H}^{1}}\rho, \nabla_{\mathbb{H}^{1}}\theta\rangle+\frac{\partial f}{\partial \varphi}\langle\nabla_{\mathbb{H}^{1}}\rho,\nabla_{\mathbb{H}^{1}}\varphi\rangle\bigg)\\
&+\rho^{\alpha}\bigg(\frac{\partial^{2}f}{\partial \theta^{2}}\left|\nabla_{\mathbb{H}^{1}}\theta\right|^{2}+2\frac{\partial^{2}f}{\partial \varphi \partial \theta}\langle\nabla_{\mathbb{H}^{1}}\varphi,\nabla_{\mathbb{H}^{1}}\theta\rangle+\frac{\partial^{2}f}{\partial \varphi^{2}}\left|\nabla_{\mathbb{H}^{1}}\varphi\right|^{2}+\frac{\partial f}{\partial \theta}\Delta_{\mathbb{H}^1}\theta+\frac{\partial f}{\partial \varphi}\Delta_{\mathbb{H}^1}\varphi\bigg),
\end{align*}
which yields
\begin{align}\label{lapl-horiz-u-1}
\Delta_{\mathbb{H}^1}u&=\Delta_{\mathbb{H}^1}(\rho^{\alpha}f(\theta,\varphi))=(\alpha(\alpha-1)\rho^{\alpha-2}\left|\nabla_{\mathbb{H}^{1}}\rho\right|^{2}+\alpha\rho^{\alpha-1}\Delta_{\mathbb{H}^1}\rho)f(\theta,\varphi)+2\alpha\rho^{\alpha-1}\bigg(\frac{\partial f}{\partial \theta}\langle\nabla_{\mathbb{H}^{1}}\rho,\nabla_{\mathbb{H}^{1}}\theta\rangle\nonumber\\
&+\frac{\partial f}{\partial \varphi}\langle\nabla_{\mathbb{H}^{1}}\rho,\nabla_{\mathbb{H}^{1}}\varphi\rangle\bigg)+\rho^{\alpha}\bigg(\frac{\partial^{2}f}{\partial \theta^{2}}\left|\nabla_{\mathbb{H}^{1}}\theta\right|^{2}+2\frac{\partial^{2}f}{\partial \varphi \partial \theta}(\nabla_{\mathbb{H}^{1}}\varphi\cdot \nabla_{\mathbb{H}^{1}}\theta)+\frac{\partial^{2}f}{\partial \varphi^{2}}\left|\nabla_{\mathbb{H}^{1}}\varphi\right|^{2}+\frac{\partial f}{\partial \theta}\Delta_{\mathbb{H}^1}\theta\nonumber\\
&+\frac{\partial f}{\partial \varphi}\Delta_{\mathbb{H}^1}\varphi\bigg).
\end{align}
In particular, using \eqref{norm-quad-grad-horiz-varphi-rho-theta}, \eqref{prod-scal-grad-varphi-grad-rho-grad-rho-grad-theta-grad-varphi-grad-theta} and \eqref{lapl-horiz-varphi-rho-theta}, we get, in view of \eqref{lapl-horiz-u-1},
\begin{align*}
\Delta_{\mathbb{H}^1}u&=\Delta_{\mathbb{H}^1}(\rho^{\alpha}f(\theta,\varphi))=\bigg(\alpha(\alpha-1)\rho^{\alpha-2}\frac{x^{2}+y^{2}}{\rho^{2}}+\alpha\rho^{\alpha-1}\frac{3(x^{2}+y^{2})}{\rho^{3}}\bigg)f(\theta,\varphi)-2\alpha\rho^{\alpha-1}\frac{\partial f}{\partial \theta}\frac{\cos \varphi}{\rho}\\
&+\rho^{\alpha}\bigg(\frac{\partial^{2}f}{\partial \theta^{2}}\frac{1}{x^{2}+y^{2}}+4\frac{\partial^{2}f}{\partial \varphi \partial \theta}\frac{x^{2}+y^{2}}{\rho^{4}}+\frac{\partial^{2}f}{\partial \varphi^{2}}\frac{4(x^{2}+y^{2})}{\rho^{4}}+\frac{\partial f}{\partial \varphi}\frac{4\cos\varphi}{\rho^{2}}\bigg)\\
&=\alpha(\alpha+2)\bigg(\frac{x^{2}+y^{2}}{\rho^{2}}\bigg)\rho^{\alpha-2}f(\theta,\varphi)-2\alpha(\cos\varphi)\rho^{\alpha-2}\frac{\partial f}{\partial \theta}+\rho^{\alpha-2}\bigg(\frac{\rho^{2}}{x^{2}+y^{2}}\frac{\partial^{2}f}{\partial \theta^{2}}\\
&+4\frac{x^{2}+y^{2}}{\rho^{2}}\frac{\partial^{2}f}{\partial \varphi \partial \theta}+4\frac{x^{2}+y^{2}}{\rho^{2}}\frac{\partial^{2}f}{\partial \varphi^{2}}+4\cos\varphi\frac{\partial f}{\partial \varphi}\bigg).
\end{align*}
Thus, we finally get
\begin{align}\label{lapl-horiz-u-final-1}
&\Delta_{\mathbb{H}^1}u=\Delta_{\mathbb{H}^1}(\rho^{\alpha}f(\theta,\varphi))=\alpha(\alpha+2)(\sin\varphi)\rho^{\alpha-2}f(\theta,\varphi)-2\alpha(\cos\varphi)\rho^{\alpha-2}\frac{\partial f}{\partial \theta}+\rho^{\alpha-2}\bigg(\frac{1}{\sin\varphi}\frac{\partial^{2}f}{\partial \theta^{2}}\nonumber\\
&+4\sin\varphi\frac{\partial^{2}f}{\partial \varphi \partial \theta}+4\sin\varphi\frac{\partial^{2}f}{\partial \varphi^{2}}+4\cos\varphi\frac{\partial f}{\partial \varphi}\bigg)\nonumber\\
&=\rho^{\alpha-2}\bigg(\alpha(\alpha+2)(\sin\varphi)f(\theta,\varphi)-2\alpha(\cos\varphi)\frac{\partial f}{\partial \theta}+\frac{1}{\sin\varphi}\frac{\partial^{2}f}{\partial \theta^{2}}+4\sin\varphi\frac{\partial^{2}f}{\partial \varphi \partial \theta}+4\sin\varphi\frac{\partial^{2}f}{\partial \varphi^{2}}\nonumber\\
&+4\cos\varphi\frac{\partial f}{\partial \varphi}\bigg),
\end{align}
since $\frac{x^{2}+y^{2}}{\rho^{2}}=\sin\varphi$ from \eqref{pol-coord}.\newline
To recap, we obtain, by recalling \eqref{lapl-horiz-u-final-1},
\begin{align*}
&\Delta_{\mathbb{H}^1}u=\Delta_{\mathbb{H}^1}(\rho^{\alpha}f(\theta,\varphi))=\rho^{\alpha-2}\bigg(\alpha(\alpha+2)(\sin\varphi)f(\theta,\varphi)-2\alpha(\cos\varphi)\frac{\partial f}{\partial \theta}+\frac{1}{\sin\varphi}\frac{\partial^{2}f}{\partial \theta^{2}}\\
&+4\sin\varphi\frac{\partial^{2}f}{\partial \varphi \partial \theta}+4\sin\varphi\frac{\partial^{2}f}{\partial \varphi^{2}}+4\cos\varphi\frac{\partial f}{\partial \varphi}\bigg).
\end{align*}
\end{proof}

\section{Computation for $f$ independent of $\theta$}

In Lemma \ref{lapl-horiz-f-theta-phi} we proved that if $u=\rho^{\alpha}f(\theta,\varphi),$ then
\begin{align}\label{lapl-horiz-final-theta-varphi}
&\Delta_{\mathbb{H}^1}u=\Delta_{\mathbb{H}^1}(\rho^{\alpha}f(\theta,\varphi))=\rho^{\alpha-2}\bigg(\alpha(\alpha+2)(\sin\varphi)f(\theta,\varphi)-2\alpha(\cos\varphi)\frac{\partial f}{\partial \theta}+\frac{1}{\sin\varphi}\frac{\partial^{2}f}{\partial \theta^{2}}\nonumber\\
&+4\sin\varphi\frac{\partial^{2}f}{\partial \varphi \partial \theta}+4\sin\varphi\frac{\partial^{2}f}{\partial \varphi^{2}}+4\cos\varphi\frac{\partial f}{\partial \varphi}\bigg).
\end{align}
Now, if we evaluate the expression \eqref{lapl-horiz-final-theta-varphi} on $\partial B^{\mathbb{H}^1}_1(0),$ we get, because $\rho=\left|(x,y,t)\right|_{\mathbb{H}^1}=1$ on $\partial B^{\mathbb{H}^1}_1(0),$
\begin{align}\label{lapl-horiz-theta-varphi-boundary-boule-koranyi}
&\Delta_{\mathbb{H}^1}u \restrict{\partial B^{\mathbb{H}^1}_1(0)}=\Delta_{\mathbb{H}^1}(\rho^{\alpha}f(\theta,\varphi)) \restrict{\partial B^{\mathbb{H}^1}_1(0)}=\alpha(\alpha+2)(\sin \varphi)f(\theta,\varphi)-2\alpha(\cos\varphi)\frac{\partial f}{\partial \theta}+\frac{1}{\sin\varphi}\frac{\partial^{2}f}{\partial \theta^{2}}\nonumber\\
&+4\sin\varphi\frac{\partial^{2}f}{\partial \varphi \partial \theta}+4\sin\varphi\frac{\partial^{2}f}{\partial \varphi^{2}}+4\cos\varphi\frac{\partial f}{\partial \varphi}.
\end{align}
\begin{cor}\label{corollary-lapl-horiz-f-phi}
If $u=\rho^{\alpha}f(\varphi),$ then
\begin{equation} \label{lapl-horiz-f-phi-boundary-boule-koranyi}
\Delta_{\mathbb{H}^1}u\restrict{\partial B^{\mathbb{H}^1}_1(0)}=\Delta_{\mathbb{H}^1}(\rho^{\alpha}f(\varphi))\restrict{\partial B^{\mathbb{H}^1}_1(0)}=\alpha(\alpha+2)(\sin \varphi)f(\varphi)+4\frac{\partial}{\partial \varphi}\left(\sin\varphi\frac{\partial f}{\partial \varphi}\right).
\end{equation}
\end{cor}
\begin{proof}
For sake of simplicity, we will identify $\Delta_{\mathbb{H}^1}u$ with $\Delta_{\mathbb{H}^1}u\restrict{\partial B^{\mathbb{H}^1}_1(0)}$ in the following.

In particular, if $f(\theta,\varphi)$ does not depend on $\theta,$ i.e. $f=f(\varphi),$ we obtain, in view of \eqref{lapl-horiz-theta-varphi-boundary-boule-koranyi},
\begin{equation}\label{lapl-horiz-varphi-prov}
\Delta_{\mathbb{H}^1}u=\alpha(\alpha+2)(\sin\varphi)f(\varphi)+4\sin\varphi\frac{\partial^{2}f}{\partial \varphi^{2}}+4\cos\varphi\frac{\partial f}{\partial \varphi}.
\end{equation}

At this point, we note that
\[
4\sin\varphi\frac{\partial^{2}f}{\partial \varphi^{2}}+4\cos\varphi\frac{\partial f}{\partial \varphi}=4\left(\sin\varphi\frac{\partial}{\partial \varphi}\left(\frac{\partial f}{\partial \varphi}\right)+\left(\frac{\partial}{\partial \varphi}(\sin\varphi)\right)\frac{\partial f}{\partial \varphi}\right)=4\frac{\partial}{\partial \varphi}\left(\sin\varphi\frac{\partial f}{\partial \varphi}\right),
\]
which implies, from \eqref{lapl-horiz-varphi-prov}, our thesis:
\[
\Delta_{\mathbb{H}^1}u\restrict{\partial B^{\mathbb{H}^1}_1(0)}=\Delta_{\mathbb{H}^1}(\rho^{\alpha}f(\varphi))\restrict{\partial B^{\mathbb{H}^1}_1(0)}=\alpha(\alpha+2)(\sin \varphi)f(\varphi)+4\frac{\partial}{\partial \varphi}\left(\sin\varphi\frac{\partial f}{\partial \varphi}\right).
\]
\end{proof}
Corollary \ref{corollary-lapl-horiz-f-phi} yields the following lemma as well.
\begin{lem} If $\alpha=2$ and we take $u=\rho^{2}\cos\varphi,$ we have that $u$ is $\Delta_{\mathbb{H}^1}$-harmonic, that is $\Delta_{\mathbb{H}^1}u=0,$ and
\begin{equation}\label{weak-equality-cos-varphi}
8=\frac{4\displaystyle \int_{0}^{\frac{\pi}{2}}\sin\varphi((\cos\varphi)')^{2}\hspace{0.1cm}d\varphi}{\displaystyle \int_{0}^{\frac{\pi}{2}}\sin\varphi\cos^{2}(\varphi)\hspace{0.1cm}d\varphi}.
\end{equation}
\end{lem}
\begin{proof}
This result can be found in \cite{Greiner}. However, for helping the reader, we give a straightforward proof of this fact.

First of all, from \eqref{lapl-horiz-final-theta-varphi}, we get that if $u=\rho^{\alpha}f(\varphi),$
\begin{equation}\label{lapl-horiz-f-varphi}
\Delta_{\mathbb{H}^1}u=\Delta_{\mathbb{H}^1}(\rho^{\alpha}f(\varphi))=\rho^{\alpha-2}\bigg(\alpha(\alpha+2)(\sin\varphi)f(\varphi)+4\sin\varphi\frac{\partial^{2}f}{\partial \varphi^{2}}+4\cos\varphi\frac{\partial f}{\partial \varphi}\bigg).
\end{equation}
As a consequence, if $u=\rho^{2}\cos\varphi,$ we have, in view of \eqref{lapl-horiz-f-varphi}, being $\alpha=2$ and $f(\varphi)=\cos\varphi,$
\begin{align*}
\Delta_{\mathbb{H}^1}u&=\Delta_{\mathbb{H}^1}(\rho^{2}\cos\varphi)=8(\sin\varphi)\cos\varphi+4\sin\varphi\frac{\partial^{2}}{\partial \varphi^{2}}(\cos\varphi)+4\cos\varphi\frac{\partial}{\partial \varphi}(\cos\varphi)\\
&=8\sin\varphi\cos\varphi-4\sin\varphi\cos\varphi-4\cos\varphi\sin\varphi=0,
\end{align*}
which gives
\begin{equation}\label{rho-2-cos-varphi-H-harmonic}
\Delta_{\mathbb{H}^1}u=\Delta_{\mathbb{H}^1}(\rho^{2}\cos\varphi)=0.
\end{equation}
Now, if $u=\rho^{\alpha}f(\varphi)$ satisfies $\Delta_{\mathbb{H}^1}u\restrict{\partial B^{\mathbb{H}^1}_1(0)}=0,$ we have, according to \eqref{lapl-horiz-f-phi-boundary-boule-koranyi},
$$\alpha(\alpha+2)(\sin\varphi)f(\varphi)+4\frac{\partial}{\partial \varphi}\left(\sin\varphi\frac{\partial f}{\partial \varphi}\right)=0,$$
which implies, writing $f=f(\varphi)$ and $\frac{\partial}{\partial\varphi}\left(\sin\varphi\frac{\partial f}{\partial \varphi}\right)=(\sin\varphi f')',$ since $f$ is a function depending only on $\varphi,$
$$-4(\sin\varphi f')'=\alpha(\alpha+2)(\sin\varphi)f,$$
and multiplying both the terms of the equality by $\eta$ sufficiently smooth with $\eta\left(\frac{\pi}{2}\right)=0,$
\begin{equation}\label{weak-equality-lapl-horiz-1}
\alpha(\alpha+2)(\sin\varphi)f\eta=-4(\sin\varphi f')'\eta.
\end{equation}
Integrating over $\left[0,\dfrac{\pi}{2}\right]$ the equality in \eqref{weak-equality-lapl-horiz-1}, we then obtain
\begin{align*}
&\int_{0}^{\frac{\pi}{2}}\alpha(\alpha+2)(\sin\varphi)f\eta\hspace{0.1cm}d\varphi=\alpha(\alpha+2)\int_{0}^{\frac{\pi}{2}}(\sin\varphi)f\eta\hspace{0.1cm}d\varphi\\
&=\int_{0}^{\frac{\pi}{2}}-4(\sin\varphi f')'\eta\hspace{0.1cm}d\varphi=-4\int_{0}^{\frac{\pi}{2}}(\sin\varphi f')'\eta\hspace{0.1cm}d\varphi,
\end{align*}
in other words
\begin{equation}\label{weak-equality-lapl-horiz-2}
\alpha(\alpha+2)\int_{0}^{\frac{\pi}{2}}(\sin\varphi)f\eta\hspace{0.1cm}d\varphi=-4\int_{0}^{\frac{\pi}{2}}(\sin\varphi f')'\eta\hspace{0.1cm}d\varphi.
\end{equation}
In particular, if we choose $\eta=f,$ we get, from \eqref{weak-equality-lapl-horiz-2}, by the Theorem of Integration by Parts:
\begin{align*}
&\alpha(\alpha+2)\int_{0}^{\frac{\pi}{2}}(\sin\varphi)f^{2}\hspace{0.1cm}d\varphi=-4\int_{0}^{\frac{\pi}{2}}(\sin\varphi f')'f\hspace{0.1cm}d\varphi=-4\bigg(\bigg[(\sin\varphi f')f\bigg]^{\varphi=\frac{\pi}{2}}_{\varphi=0}-\int_{0}^{\frac{\pi}{2}}\sin\varphi f'f'\hspace{0.1cm}d\varphi\bigg)\\
&=-4\bigg(\sin\left(\frac{\pi}{2}\right)f'\left(\frac{\pi}{2}\right)f\left(\frac{\pi}{2}\right)-\sin(0)f'(0)f(0)-\int_{0}^{\frac{\pi}{2}}\sin\varphi (f')^{2}\hspace{0.1cm}d\varphi\bigg).
\end{align*}
This implies, because $\sin(0)=0$ and $f\big(\frac{\pi}{2}\big)=0$ by virtue of the choice of $f$, that
\begin{equation}\label{weak-equality-lapl-horiz-3}
\alpha(\alpha+2)\int_{0}^{\frac{\pi}{2}}(\sin\varphi)f^{2}\hspace{0.1cm}d\varphi=4\int_{0}^{\frac{\pi}{2}}\sin\varphi(f')^{2}\hspace{0.1cm}d\varphi.
\end{equation}
In addition, in view of \eqref{weak-equality-lapl-horiz-3}, we also have
\begin{equation}\label{weak-equality-lapl-horiz-4}
\alpha(\alpha+2)=\frac{4\displaystyle\int_{0}^{\frac{\pi}{2}}\sin\varphi (f')^{2}\hspace{0.1cm}d\varphi}{\displaystyle\int_{0}^{\frac{\pi}{2}}(\sin\varphi)f^{2}\hspace{0.1cm}d\varphi}.
\end{equation}
At this point, we recall that, from \eqref{rho-2-cos-varphi-H-harmonic}, $\rho^{2}\cos\varphi$ is $\mathbb{H}^1$-harmonic, where $\alpha=2$ and $f(\varphi)=\cos\varphi,$ with $\cos\left(\dfrac{\pi}{2}\right)=0,$ hence, repeating the same argument used to achieve \eqref{weak-equality-lapl-horiz-4}, we have
\[
8=\frac{4\displaystyle \int_{0}^{\frac{\pi}{2}}\sin\varphi((\cos\varphi)')^{2}\hspace{0.1cm}d\varphi}{\displaystyle \int_{0}^{\frac{\pi}{2}}\sin\varphi\cos^{2}(\varphi)\hspace{0.1cm}d\varphi}.
\]
\end{proof}
 
 \section{Estimates in $\mathbb{H}^1$ and characteristic points}

 Let $\nabla_{\mathbb{H}^1}u(p)\in H\mathbb{H}^1_p,$ where $H\mathbb{H}^1_p$ denotes the horizontal vector space at $p\in \mathbb{H}^1,$ see Section \ref{notation_Heisenberg}. Let us define
 $$
 e_\rho\coloneqq \frac{\nabla_{\mathbb{H}^1}\rho}{|\nabla_{\mathbb{H}^1}\rho|},\quad  e_\varphi\coloneqq \frac{\nabla_{\mathbb{H}^1}\varphi}{|\nabla_{\mathbb{H}^1}\varphi|}.
 $$
 We recall, according to Lemma \ref{lemma-inner-product-nabla-phi-rho-theta}, that $\langle e_\rho,e_\varphi\rangle_{\mathbb{R}^2}=0.$  Then, whenever $e_\rho,e_\varphi$ exist, we have:
 $$
 \mbox{span}\{e_\rho(p),e_\varphi(p)\}=H\mathbb{H}^1_p.
 $$ 
 As a consequence, in these cases, since $\{e_\rho,e_\varphi\}$ is an orthonormal basis,
 $$
 \nabla_{\mathbb{H}^1}u(p)=\langle \nabla_{\mathbb{H}^1}u(p),e_\rho(p)\rangle e_\rho(p)+\langle \nabla_{\mathbb{H}^1}u(p),e_\varphi(p)\rangle e_\varphi(p)
 $$
 and denoting 
 \begin{equation}\label{two-components-nabla-horiz}
 \nabla^\rho_{\mathbb{H}^1}u(p)=\langle\nabla_{\mathbb{H}^1}u(p),e_\rho(p)\rangle e_\rho(p),\quad
 \nabla^\varphi_{\mathbb{H}^1}u(p)=\langle \nabla_{\mathbb{H}^1}u(p),e_\varphi(p)\rangle e_\varphi(p), 
 \end{equation}
 we have
 $$
 |\nabla_{\mathbb{H}^1}u(p)|^2=\langle \nabla_{\mathbb{H}^1}u(p),e_\rho(p)\rangle^2+\langle \nabla_{\mathbb{H}^1}u(p),e_\varphi(p)\rangle^2,
 $$
 and
\begin{equation}\label{norm-quad-lapl-horiz-u-two-components_f}
 |\nabla_{\mathbb{H}^1}u(p)|^2= |\nabla^\rho_{\mathbb{H}^1}u(p)|^2+ |\nabla^\varphi_{\mathbb{H}^1}u(p)|^2.
\end{equation}

\begin{lem}
The couple
 $(\nabla_{\mathbb{H}^1}\rho)(p)$ , $(\nabla_{\mathbb{H}^1}\varphi)(p)$ determines a basis of $H\mathbb{H}^1_p,$ for every $p=(x,y,t),$ such that $x^{2}+y^{2}\neq 0.$
\end{lem}
\begin{proof}
We look for the points where $\nabla_{\mathbb{H}^1}\rho$ and $\nabla_{\mathbb{H}^1}\varphi$  vanish.
We have that $\nabla_{\mathbb{H}^1}\rho=0$ if
\begin{equation}\label{nabla-horiz-rho-equal-to-0-1}
\begin{cases}
\rho^{-3}((x^{2}+y^{2})x+yt)=0\\
\rho^{-3}((x^{2}+y^{2})y-xt)=0
\end{cases}
\Longleftrightarrow\quad
\begin{cases}
(x^{2}+y^{2})x+yt=0\\
(x^{2}+y^{2})y-xt=0,
\end{cases}
\end{equation}
which gives, multiplying the first row by $y\neq 0$ and the second one by $x\neq 0,$
\begin{equation}\label{nabla-horiz-rho-equal-to-0-2}
\begin{cases}
(x^{2}+y^{2})yx+y^{2}t=0\\
(x^{2}+y^{2})yx-x^{2}t=0.
\end{cases}
\end{equation}
Subtracting the second row to the first one in \eqref{nabla-horiz-rho-equal-to-0-2}, we get
$$0=y^{2}t+x^{2}t=(x^{2}+y^{2})t,$$
which implies $t=0,$ because $x\neq 0$ and $y\neq 0.$
Now, if $t=0,$ we obtain, from the first row in \eqref{nabla-horiz-rho-equal-to-0-1}, $(x^{2}+y^{2})x=0,$ which is a contradiction, recalling that we have supposed that $x\neq 0.$

Therefore, suppose that $y=0,$ and in view of the first row in \eqref{nabla-horiz-rho-equal-to-0-1}, we have $x=0.$ Analogously, if we assume $x=0,$ we achieve, by the second row in \eqref{nabla-horiz-rho-equal-to-0-1}, $y=0.$
To sum up, we have $\nabla_{\mathbb{H}^1}\rho=0$ in points $p=(x,y,t),$ with $x=0$ and $y=0.$

Concerning $\nabla_{\mathbb{H}^1}\varphi,$ we have $\nabla_{\mathbb{H}^1}\varphi=0$ if
$$\frac{2}{\rho(x^{2}+y^{2})}\left(\rho(-y,x)+t\nabla_{\mathbb{H}^1}\rho\right)=0,$$
which immediately yields that $x$ and $y$ can not be equal to $0$ at the same time, so it is equivalent to 
$$\rho(-y,x)+t\nabla_{\mathbb{H}^1}\rho=0,$$
that is
\begin{align*}
&\begin{cases}
-\rho y+tX\rho=0\\
\rho x+tY\rho=0
\end{cases}
\Longleftrightarrow \quad
\begin{cases}
-\rho y+t\rho^{-3}((x^{2}+y^{2})x+yt)=0\\
\rho x+t\rho^{-3}((x^{2}+y^{2})y-xt)=0
\end{cases}\\
&\Longleftrightarrow \quad
\begin{cases}
t\rho^{-3}((x^{2}+y^{2})x+yt)=\rho y\\
t\rho^{-3}((x^{2}+y^{2})y-xt)=-\rho x
\end{cases}
\Longleftrightarrow \quad
\begin{cases}
t((x^{2}+y^{2})x+yt)=\rho^{4}y\\
t((x^{2}+y^{2})y-xt)=-\rho^{4}x,
\end{cases}
\end{align*}
thus, we have to solve
\begin{equation}\label{nabla-horiz-varphi-equal-to-0-1}
\begin{cases}
t((x^{2}+y^{2})x+yt)=\rho^{4}y\\
t((x^{2}+y^{2})y-xt)=-\rho^{4}x.
\end{cases}
\end{equation}
Specifically, multiplying the first row in \eqref{nabla-horiz-varphi-equal-to-0-1} by $y\neq 0$ and the second one by $x\neq 0,$ we get
\begin{equation}\label{nabla-horiz-varphi-equal-to-0-2}
\begin{cases}
t(x^{2}+y^{2})xy+t^{2}y^{2}=\rho^{4}y^{2}\\
t(x^{2}+y^{2})yx-t^{2}x^{2}=-\rho^{4}x^{2}.
\end{cases}
\end{equation}
Subtracting the second row in \eqref{nabla-horiz-varphi-equal-to-0-2} to the first one,
$$(x^{2}+y^{2})t^{2}=(x^{2}+y^{2})\rho^{4},$$
and dividing by $(x^{2}+y^{2})\neq 0,$ recalling that $x$ and $y$ can not be equal to $0$ at the same time,
$$t^{2}=\rho^{4},$$
which implies
$$\left|t\right|=\rho^{2},$$
and hence $t=\pm\rho^{2}.$ Substituting $t=\rho^{2}$ in the first row of \eqref{nabla-horiz-varphi-equal-to-0-2}, we achieve
$$\rho^{2}(x^{2}+y^{2})yx+\rho^{4}y^{2}=\rho^{4}y^{2},$$
which gives
$$\rho^{2}(x^{2}+y^{2})yx=0,$$
which is a contradiction, since $x\neq 0,$ $y\neq 0$ and $\rho\neq 0.$ Analogously, if we take $t=-\rho^{2},$ we have, always from the first row in \eqref{nabla-horiz-varphi-equal-to-0-2},
$$-\rho^{2}(x^{2}+y^{2})yx+\rho^{4}y^{2}=\rho^{4}y^{2},$$
i.e.
$$-\rho^{2}(x^{2}+y^{2})yx=0,$$
which is a contradiction, always because $x\neq 0,$ $y\neq 0$ and $\rho\neq 0.$\newline
Suppose now that $y=0,$ and we have, according the first row in \eqref{nabla-horiz-varphi-equal-to-0-1},
$tx^{3}=0,$ which entails $t=0,$ inasmuch $x$ and $y$ can not be equal to $0$ at the same time. At this point, if $y=t=0,$ we have, by the second row in \eqref{nabla-horiz-varphi-equal-to-0-1}, $\rho^4x=0,$ in other words $x=0,$ 
recalling that $\rho\neq 0,$ which is impossible, since $y=0.$ Analogously, if we assume $x=0,$ we have from the second row in \eqref{nabla-horiz-varphi-equal-to-0-1} that the only possibility is $t=0,$ but this condition yields, by virtue of the first row in \eqref{nabla-horiz-varphi-equal-to-0-1}, $y=0,$ which is impossible, because $x=0.$ To recap, $\nabla_{\mathbb{H}^1}\varphi\neq 0,$ $\forall p \in \mathbb{H}^{1},$ where it is well defined, i.e. in points $p=(x,y,t)$ such that $x^{2}+y^{2}\neq 0.$

This fact, together with $\nabla_{\mathbb{H}^1}\rho=0$ if $x=y=0,$ gives that $(\nabla_{\mathbb{H}^1}\rho)(p)$ and $(\nabla_{\mathbb{H}^1}\varphi)(p)$ are a basis of $H\mathbb{H}^1_p$ in points $p=(x,y,t)$ with $x^{2}+y^{2}\neq 0.$
\end{proof}

 \section{Some general estimates in $\mathbb{H}^1$}
In this section we only work in $\mathbb{H}^1.$ Nevertheless, for improving, hopefully, the presentation of our result we repeat some computations already done in $\mathbb{H}^{n}.$

Let us introduce the following notation:
\begin{equation}\label{def-A-rho-A-varphi-A-u_x}
\begin{split}
A_{\rho}\coloneqq \int_{\partial B^{\mathbb{H}^1}_1(0)}\frac{\left|\nabla_{\mathbb{H}^1}^{\rho}u\right|^{2}}{\sqrt{x^{2}+y^{2}}}&\hspace{0.1cm}d\sigma_{\mathbb{H}^1}(\xi),\quad
A_{\varphi}\coloneqq \int_{\partial B^{\mathbb{H}^1}_1(0)}\frac{\left|\nabla_{\mathbb{H}^1}^{\varphi}u\right|^{2}}{\sqrt{x^{2}+y^{2}}}\hspace{0.1cm}d\sigma_{\mathbb{H}^1}(\xi),\\
&A_{u}\coloneqq \int_{\partial B^{\mathbb{H}^1}_1(0)}u^{2}\sqrt{x^{2}+y^{2}}\hspace{0.1cm}d\sigma_{\mathbb{H}^1}(\xi).
\end{split}
\end{equation}

\begin{lem} Let $u$ be one of the two functions of the Theorem \ref{lowerbound_f}. Then
the following lower bound holds:
\[ \frac{\displaystyle \int_{\partial B^{\mathbb{H}^1}_1(0)}\frac{\left|\nabla_{\mathbb{H}^1}u\right|^{2}}{\sqrt{x^{2}+y^{2}}}\hspace{0.1cm}d\sigma_{\mathbb{H}^1}(\xi)}{\displaystyle \int_{B^{\mathbb{H}^1}_1(0)}\frac{\left|\nabla_{\mathbb{H}^1}u\right|^{2}}{\left|\xi\right|_{\mathbb{H}^1}^{2}}\hspace{0.1cm}d\xi}\geq \frac{A_{\rho}+A_{\varphi}}{A_{u}+A_{u}^{1/2}A_{\rho}^{1/2}}.\]
\end{lem}
\begin{proof}
Using \eqref{norm-quad-lapl-horiz-u-two-components_f}, we get
\begin{align}\label{general-term-condition-monotonicity}
&\frac{\displaystyle \int_{\partial B^{\mathbb{H}^1}_1(0)}\frac{\left|\nabla_{\mathbb{H}^1}u\right|^{2}}{\sqrt{x^{2}+y^{2}}}\hspace{0.1cm}d\sigma_{\mathbb{H}^1}(\xi)}{\displaystyle \int_{B^{\mathbb{H}^1}_1(0)}\frac{\left|\nabla_{\mathbb{H}^1}u\right|^{2}}{\left|\xi\right|_{\mathbb{H}^1}^{2}}\hspace{0.1cm}d\xi}=\frac{\displaystyle\int_{\partial B^{\mathbb{H}^1}_1(0)}\frac{\left|\nabla_{\mathbb{H}^1}^{\rho}u\right|^{2}+\left|\nabla_{\mathbb{H}^1}^{\varphi}u\right|^{2}}{\sqrt{x^{2}+y^{2}}}\hspace{0.1cm}d\sigma_{\mathbb{H}^1}(\xi)}{\displaystyle\int_{B^{\mathbb{H}^1}_1(0)}\frac{\left|\nabla_{\mathbb{H}^1}u\right|^{2}}{\left|\xi\right|_{\mathbb{H}^1}^{2}}\hspace{0.1cm}d\xi}\nonumber\\
&=\frac{\displaystyle \int_{\partial B^{\mathbb{H}^1}_1(0)}\frac{\left|\nabla_{\mathbb{H}^1}^{\rho}u\right|^{2}}{\sqrt{x^{2}+y^{2}}}\hspace{0.1cm}d\sigma_{\mathbb{H}^1}(\xi)+\displaystyle \int_{\partial B^{\mathbb{H}^1}_1(0)}\frac{\left|\nabla_{\mathbb{H}^1}^{\varphi}u\right|^{2}}{\sqrt{x^{2}+y^{2}}}\hspace{0.1cm}d\sigma_{\mathbb{H}^1}(\xi)}{\displaystyle \int_{B^{\mathbb{H}^1}_1(0)}\frac{\left|\nabla_{\mathbb{H}^1}u\right|^{2}}{\left|\xi\right|_{\mathbb{H}^1}^{2}}\hspace{0.1cm}d\xi}.
\end{align}
We now look for an upper bound for
\begin{equation*}
\int_{B^{\mathbb{H}^1}_1(0)}\frac{\left|\nabla_{\mathbb{H}^1}u\right|^{2}}{\left|\xi\right|_{\mathbb{H}^1}^{2}}\hspace{0.1cm}d\xi.
\end{equation*}
Specifically, we have
\begin{align*}
&\Delta_{\mathbb{H}^1}u^{2}=X(X(u^{2}))+Y(Y(u^{2}))=X(2uXu)+Y(2uYu)=2((Xu)^2+uX^{2}u)+2((Yu)^2+uY^{2}u)\\
&=2((Xu)^{2}+(Yu)^{2})+2u(X^{2}u+Y^{2}u)=2\left|\nabla_{\mathbb{H}^1}u\right|^{2}+2u\Delta_{\mathbb{H}^1}u,
\end{align*}
i.e.
\[ \Delta_{\mathbb{H}^1}u^{2}=2\left|\nabla_{\mathbb{H}^1}u\right|^{2}+2u\Delta_{\mathbb{H}^1}u, \]
which implies, if $u$ satisfies $u\Delta_{\mathbb{H}^1}u\geq 0,$
\[2\left|\nabla_{\mathbb{H}^1}u\right|^{2}\leq \Delta_{\mathbb{H}^1} u^{2}, \]
and thus
\begin{equation}\label{ineq-norm-quad-nabla-horiz-u}
\left|\nabla_{\mathbb{H}^1}u\right|^{2}\leq \frac{1}{2}\Delta_{\mathbb{H}^1}u^{2}.
\end{equation}
In view of \eqref{ineq-norm-quad-nabla-horiz-u}, we then achieve
\begin{equation}\label{denominator-general-term-condition-monotonicity-1}
\int_{B^{\mathbb{H}^1}_1(0)}\frac{\left|\nabla_{\mathbb{H}^1}u\right|^{2}}{\left|\xi\right|_{\mathbb{H}^1}^{2}}\hspace{0.1cm}d\xi\leq \frac{1}{2}\int_{B^{\mathbb{H}^1}_1(0)}\frac{\Delta_{\mathbb{H}^1}u^{2}}{\left|\xi\right|_{\mathbb{H}^1}^{2}}\hspace{0.1cm}d\xi.
\end{equation}
At this point, we denote with $\diver_{\mathbb{H}^1}$ the following operator
\begin{equation}\label{def-diver-horiz}
\diver_{\mathbb{H}^1}b=\diver_{\mathbb{H}^1}\left(b_{1},b_{2}\right)=Xb_{1}+Yb_{2},
\end{equation} 
where $b:\mathbb{H}^1\to \mathbb{R}^2.$
Therefore, we have, by \eqref{def-diver-horiz},
\begin{align*}
&\diver_{\mathbb{H}^1}\left(\left|\xi\right|_{\mathbb{H}^1}^{-2}\nabla_{\mathbb{H}^1}u^{2}\right)=X(\left|\xi\right|_{\mathbb{H}^1}^{-2}Xu^{2})+Y(\left|\xi\right|_{\mathbb{H}^1}^{-2}Yu^{2})=X\left|\xi\right|_{\mathbb{H}^1}^{-2}Xu^{2}+\left|\xi\right|_{\mathbb{H}^1}^{-2}X^{2}u^{2}+Y\left|\xi\right|_{\mathbb{H}^1}^{-2}Yu^{2}\\
&+\left|\xi\right|_{\mathbb{H}^1}^{-2}Y^{2}u^{2}=\langle\nabla_{\mathbb{H}^1}\left|\xi\right|_{\mathbb{H}^1}^{-2}, \nabla_{\mathbb{H}^1}u^{2}\rangle+\left|\xi\right|_{\mathbb{H}^1}^{-2}\Delta_{\mathbb{H}^1}u^{2},
\end{align*}
which entails
\begin{equation}\label{laplacian-horiz-u-^-2-equality}
\left|\xi\right|_{\mathbb{H}^1}^{-2}\Delta_{\mathbb{H}^1}u^{2}=\diver_{\mathbb{H}^1}\left(\left|\xi\right|_{\mathbb{H}^1}^{-2}\nabla_{\mathbb{H}^1}u^{2}\right)-\langle\nabla_{\mathbb{H}^1}\left|\xi\right|_{\mathbb{H}^1}^{-2},\nabla_{\mathbb{H}^1}u^{2}\rangle.
\end{equation}
Using \eqref{laplacian-horiz-u-^-2-equality}, we obtain, by virtue of \eqref{denominator-general-term-condition-monotonicity-1},
\begin{align}\label{denominator-general-term-condition-monotonicity-2}
&\int_{B^{\mathbb{H}^1}_1(0)}\frac{\left|\nabla_{\mathbb{H}^1}u\right|^{2}}{\left|\xi\right|_{\mathbb{H}^1}^{2}}\hspace{0.1cm}d\xi\leq\frac{1}{2}\int_{B^{\mathbb{H}^1}_1(0)}\diver_{\mathbb{H}^1}\left(\left|\xi\right|_{\mathbb{H}^1}^{-2}\nabla_{\mathbb{H}^1}u^{2}\right)\hspace{0.1cm}d\xi-\frac{1}{2}\int_{B^{\mathbb{H}^1}_1(0)}\langle\nabla_{\mathbb{H}^1}\left|\xi\right|_{\mathbb{H}^1}^{-2},\nabla_{\mathbb{H}^1}u^{2}\rangle\hspace{0.1cm}d\xi.
\end{align}
Now, we have
\begin{equation}\label{nabla-horiz-u-^-2}
\nabla_{\mathbb{H}^1}u^{2}=(Xu^{2},Yu^{2})=(2uXu,2uYu)=2u\nabla_{\mathbb{H}^1}u,
\nabla_{\mathbb{H}^1}u^{2}=2u\nabla_{\mathbb{H}^1}u.
\end{equation}
Consequently, by the analogous of the Divergence Theorem in $\mathbb{H}^{1}$ and from \eqref{nabla-horiz-u-^-2}, we get
\begin{align*}
&\int_{B^{\mathbb{H}^1}_1(0)}\diver_{\mathbb{H}^1}\left(\left|\xi\right|_{\mathbb{H}^1}^{-2}\nabla_{\mathbb{H}^1}u^{2}\right)\hspace{0.1cm}d\xi=\int_{\partial B^{\mathbb{H}^1}_1(0)}\langle\left|\xi\right|_{\mathbb{H}^1}^{-2}\nabla_{\mathbb{H}^1}u^{2},\nu_{\mathbb{H}^1}\rangle\hspace{0.1cm}d\sigma_{\mathbb{H}^1}(\xi)\\
&=\int_{\partial B^{\mathbb{H}^1}_1(0)}\langle \left|\xi\right|_{\mathbb{H}^1}^{-2}2u\nabla_{\mathbb{H}^1}u,\nu_{\mathbb{H}^1}\rangle\hspace{0.1cm}d\sigma_{\mathbb{H}^1}(\xi)=\int_{\partial B^{\mathbb{H}^1}_1(0)}2\left|\xi\right|_{\mathbb{H}^1}^{-2}u\langle\nabla_{\mathbb{H}^1}u,\nu_{\mathbb{H}^1}\rangle\hspace{0.1cm}d\sigma_{\mathbb{H}^1}(\xi),
\end{align*} 
namely
\begin{equation}\label{divergence-theorem-heisenberg-1}
\int_{B^{\mathbb{H}^1}_1(0)}\diver_{\mathbb{H}^1}\left(\left|\xi\right|_{\mathbb{H}^1}^{-2}\nabla_{\mathbb{H}^1}u^{2}\right)\hspace{0.1cm}d\xi=\int_{\partial B^{\mathbb{H}^1}_1(0)}2\left|\xi\right|_{\mathbb{H}^1}^{-2}u\langle\nabla_{\mathbb{H}^1}u,\nu_{\mathbb{H}^1}\rangle\hspace{0.1cm}d\sigma_{\mathbb{H}^1}(\xi),
\end{equation}
where 
\begin{equation}\label{def-nu-H}
\nu_{\mathbb{H}^1}=\frac{\nabla_{\mathbb{H}^1}\rho}{\left|\nabla_{\mathbb{H}^1}\rho\right|},
\end{equation}
with
\begin{equation}\label{def-rho}
\rho=\left|\xi\right|_{\mathbb{H}^1}.
\end{equation}
In particular, on $\partial B^{\mathbb{H}^1}_1(0),$ we have $\left|\xi\right|_{\mathbb{H}^1}=1,$ therefore, in view of \eqref{divergence-theorem-heisenberg-1}, we achieve
\begin{equation}\label{divergence-theorem-heisenberg-2}
\int_{B^{\mathbb{H}^1}_1(0)}\diver_{\mathbb{H}^1}\left(\left|\xi\right|_{\mathbb{H}^1}^{-2}\nabla_{\mathbb{H}^1}u^{2}\right)\hspace{0.1cm}d\xi=\int_{\partial B^{\mathbb{H}^1}_1(0)}2u\langle\nabla_{\mathbb{H}^1}u, \nu_{\mathbb{H}^1}\rangle\hspace{0.1cm}d\sigma_{\mathbb{H}^1}(\xi).
\end{equation}
In addition, repeating the same calculation done to find \eqref{laplacian-horiz-u-^-2-equality}, we obtain
\[\diver_{\mathbb{H}^1}\left(u^{2}\nabla_{\mathbb{H}^1}\left|\xi\right|_{\mathbb{H}^1}^{-2}\right)=\langle\nabla_{\mathbb{H}^1}u^{2},\nabla_{\mathbb{H}^1}\left|\xi\right|_{\mathbb{H}^1}^{-2}\rangle+u^{2}\diver_{\mathbb{H}^1}\left(\nabla_{\mathbb{H}^1}\left|\xi\right|_{\mathbb{H}^1}^{-2}\right), \]
which implies
\begin{equation}\label{nabla-horiz-u-^-2-cdot-nabla-horiz-norm-horiz-nu}
\langle\nabla_{\mathbb{H}^1}u^{2},\nabla_{\mathbb{H}^1}\left|\xi\right|_{\mathbb{H}^1}^{-2}\rangle=\diver_{\mathbb{H}^1}\left(u^{2}\nabla_{\mathbb{H}^1}\left|\xi\right|_{\mathbb{H}^1}^{-2}\right)-u^{2}\diver_{\mathbb{H}^1}\left(\nabla_{\mathbb{H}^1}\left|\xi\right|_{\mathbb{H}^1}^{-2}\right).
\end{equation}
Hence, from \eqref{nabla-horiz-u-^-2-cdot-nabla-horiz-norm-horiz-nu} and by the analogous of the Divergence Theorem in $\mathbb{H}^{1},$ we have
\begin{align*}
&\int_{B^{\mathbb{H}^1}_1(0)}\langle\nabla_{\mathbb{H}^1}u^{2},\nabla_{\mathbb{H}^1}\left|\xi\right|_{\mathbb{H}^1}^{-2}\rangle\hspace{0.1cm}d\xi=\int_{B^{\mathbb{H}^1}_1(0)}\left(\diver_{\mathbb{H}^1}\left(u^{2}\nabla_{\mathbb{H}^1}\left|\xi\right|_{\mathbb{H}^1}^{-2}\right)-u^{2}\diver_{\mathbb{H}^1}\left(\nabla_{\mathbb{H}^1}\left|\xi\right|_{\mathbb{H}^1}^{-2}\right)\right)\hspace{0.1cm}d\xi\\
&=\int_{B^{\mathbb{H}^1}_1(0)}\diver_{\mathbb{H}^1}\left(u^{2}\nabla_{\mathbb{H}^1}\left|\xi\right|_{\mathbb{H}^1}^{-2}\right)\hspace{0.1cm}d\xi-\int_{B^{\mathbb{H}^1}_1(0)}u^{2}\diver_{\mathbb{H}^1}\left(\nabla_{\mathbb{H}^1}\left|\xi\right|_{\mathbb{H}^1}^{-2}\right)\hspace{0.1cm}d\xi\\
&=\int_{\partial B^{\mathbb{H}^1}_1(0)}u^{2}\langle\nabla_{\mathbb{H}^1}\left|\xi\right|_{\mathbb{H}^1}^{-2}, \nu_{\mathbb{H}^1}\rangle\hspace{0.1cm}d\sigma_{\mathbb{H}^1}(\xi)-\int_{B^{\mathbb{H}^1}_1(0)}u^{2}\diver_{\mathbb{H}^1}\left(\nabla_{\mathbb{H}^1}\left|\xi\right|_{\mathbb{H}^1}^{-2}\right)\hspace{0.1cm}d\xi
\end{align*}
which entails
\begin{align}\label{second-term-denominator-general-term-condition-monotonicity-2}
&\int_{B^{\mathbb{H}^1}_1(0)}\nabla_{\mathbb{H}^1}u^{2}\cdot \nabla_{\mathbb{H}^1}\left|\xi\right|_{\mathbb{H}^1}^{-2}\hspace{0.1cm}d\xi=\int_{\partial B^{\mathbb{H}^1}_1(0)}u^{2}\langle\nabla_{\mathbb{H}^1}\left|\xi\right|_{\mathbb{H}^1}^{-2}, \nu_{\mathbb{H}^1}\rangle\hspace{0.1cm}d\sigma_{\mathbb{H}^1}(\xi)\notag\\
&-\int_{B^{\mathbb{H}^1}_1(0)}u^{2}\diver_{\mathbb{H}^1}\left(\nabla_{\mathbb{H}^1}\left|\xi\right|_{\mathbb{H}^1}^{-2}\right)\hspace{0.1cm}d\xi.
\end{align}
As a consequence, by virtue of \eqref{divergence-theorem-heisenberg-2} and \eqref{second-term-denominator-general-term-condition-monotonicity-2}, we achieve from \eqref{denominator-general-term-condition-monotonicity-2}
\begin{align*}
&\int_{B^{\mathbb{H}^1}_1(0)}\frac{\left|\nabla_{\mathbb{H}^1}u\right|^{2}}{\left|\xi\right|_{\mathbb{H}^1}^{2}}\hspace{0.1cm}d\xi\leq\frac{1}{2}\int_{\partial B^{\mathbb{H}^1}_1(0)}2u\langle\nabla_{\mathbb{H}^1}u, \nu_{\mathbb{H}^1}\rangle\hspace{0.1cm}d\sigma_{\mathbb{H}^1}(\xi)\\
&-\frac{1}{2}\bigg(\int_{\partial B^{\mathbb{H}^1}_1(0)}u^{2}\langle\nabla_{\mathbb{H}^1}\left|\xi\right|_{\mathbb{H}^1}^{-2}, \nu_{\mathbb{H}^1}\rangle\hspace{0.1cm}d\sigma_{\mathbb{H}^1}(\xi)-\int_{B^{\mathbb{H}^1}_1(0)}u^{2}\diver_{\mathbb{H}^1}\left(\nabla_{\mathbb{H}^1}\left|\xi\right|_{\mathbb{H}^1}^{-2}\right)\hspace{0.1cm}d\xi\bigg)\\
&=\int_{\partial B^{\mathbb{H}^1}_1(0)}u\langle\nabla_{\mathbb{H}^1}u, \nu_{\mathbb{H}^1}\rangle\hspace{0.1cm}d\sigma_{\mathbb{H}^1}(\xi)+\frac{1}{2}\int_{B^{\mathbb{H}^1}_1(0)}u^{2}\diver_{\mathbb{H}^1}\left(\nabla_{\mathbb{H}^1}\left|\xi\right|_{\mathbb{H}^1}^{-2}\right)\hspace{0.1cm}d\xi\\
&-\frac{1}{2}\int_{\partial B^{\mathbb{H}^1}_1(0)}u^{2}\langle\nabla_{\mathbb{H}^1}\left|\xi\right|_{\mathbb{H}^1}^{-2},\nu_{\mathbb{H}^1}\rangle\hspace{0.1cm}d\sigma_{\mathbb{H}^1}(\xi),
\end{align*}
that is
\begin{align}\label{denominator-general-term-condition-monotonicity-3}
&\int_{B^{\mathbb{H}^1}_1(0)}\frac{\left|\nabla_{\mathbb{H}^1}u\right|^{2}}{\left|\xi\right|_{\mathbb{H}^1}^{2}}\hspace{0.1cm}d\xi\leq\int_{\partial B^{\mathbb{H}^1}_1(0)}u\langle\nabla_{\mathbb{H}^1}u, \nu_{\mathbb{H}^1}\rangle\hspace{0.1cm}d\sigma_{\mathbb{H}^1}(\xi)\notag\\
&+\frac{1}{2}\int_{B^{\mathbb{H}^1}_1(0)}u^{2}\diver_{\mathbb{H}^1}\left(\nabla_{\mathbb{H}^1}\left|\xi\right|_{\mathbb{H}^1}^{-2}\right)\hspace{0.1cm}d\xi-\frac{1}{2}\int_{\partial B^{\mathbb{H}^1}_1(0)}u^{2}\langle\nabla_{\mathbb{H}^1}\left|\xi\right|_{\mathbb{H}^1}^{-2},\nu_{\mathbb{H}^1}\rangle\hspace{0.1cm}d\sigma_{\mathbb{H}^1}(\xi).
\end{align}
At this point, we have
\begin{equation}\label{nabla-horiz-norm-horiz-xi}
\nabla_{\mathbb{H}^1}\left|\xi\right|_{\mathbb{H}^1}^{-2}=\left(X\left|\xi\right|_{\mathbb{H}^1}^{-2},Y\left|\xi\right|_{\mathbb{H}^1}^{-2}\right)=\left(-2\left|\xi\right|_{\mathbb{H}^1}^{-3}X\left|\xi\right|_{\mathbb{H}^1},-2\left|\xi\right|_{\mathbb{H}^1}^{-3}Y\left|\xi\right|_{\mathbb{H}^1}\right)=-2\left|\xi\right|_{\mathbb{H}^1}^{-3}\nabla_{\mathbb{H}^1}\left|\xi\right|_{\mathbb{H}^1}.
\end{equation}
As a result, using \eqref{nabla-horiz-norm-horiz-xi}, \eqref{def-nu-H} and \eqref{def-rho}, we get
\[
\langle\nabla_{\mathbb{H}^1}\left|\xi\right|_{\mathbb{H}^1}^{-2}, \nu_{\mathbb{H}^1}\rangle=\langle-2\left|\xi\right|_{\mathbb{H}^1}^{-3}\nabla_{\mathbb{H}^1}\left|\xi\right|_{\mathbb{H}^1}, \frac{\nabla_{\mathbb{H}^1}\rho}{\left|\nabla_{\mathbb{H}^1}\rho\right|}\rangle=-2\rho^{-3}\langle\nabla_{\mathbb{H}^1}\rho,\frac{\nabla_{\mathbb{H}^1}\rho}{\left|\nabla_{\mathbb{H}^1}\rho\right|}\rangle=-2\rho^{-3}\left|\nabla_{\mathbb{H}^1}\rho\right|,
\]
which yields
\begin{equation}\label{scal-prod-nabla-horiz-norm-horiz-xi-nu-H}
\langle\nabla_{\mathbb{H}^1}\left|\xi\right|_{\mathbb{H}^1}^{-2}, \nu_{\mathbb{H}^1}\rangle=-2\rho^{-3}\left|\nabla_{\mathbb{H}^1}\rho\right|.
\end{equation}
Notice, in particular, that by \eqref{def-nu-H}, we have 
\begin{equation}\label{nu-H-equal-to-e-rho}
\nu_{\mathbb{H}^1}=e_{\rho}.
\end{equation}
Moreover, it holds 
\begin{equation}\label{rho-on-partial-B-H-0-1}
\rho=1\quad\mbox{on }\partial B^{\mathbb{H}^1}_1(0).
\end{equation}
Thus, by virtue of \eqref{scal-prod-nabla-horiz-norm-horiz-xi-nu-H},\eqref{nu-H-equal-to-e-rho} and \eqref{rho-on-partial-B-H-0-1}, we obtain from \eqref{denominator-general-term-condition-monotonicity-3}, using also \eqref{two-components-nabla-horiz},
\begin{align*}
&\int_{B^{\mathbb{H}^1}_1(0)}\frac{\left|\nabla_{\mathbb{H}^1}u\right|^{2}}{\left|\xi\right|_{\mathbb{H}^1}^{2}}\hspace{0.1cm}d\xi\leq\int_{\partial B^{\mathbb{H}^1}_1(0)}u\hspace{0.05cm}\langle\nabla_{\mathbb{H}^1}u,e_\rho\rangle\hspace{0.1cm}d\sigma_{\mathbb{H}^1}(\xi)+\frac{1}{2}\int_{B^{\mathbb{H}^1}_1(0)}u^{2}\diver_{\mathbb{H}^1}\left(\nabla_{\mathbb{H}^1}\left|\xi\right|_{\mathbb{H}^1}^{-2}\right)\hspace{0.1cm}d\xi\\
&-\frac{1}{2}\int_{\partial B^{\mathbb{H}^1}_1(0)}u^{2}\left(-2\rho^{-3}\left|\nabla_{\mathbb{H}^1}\rho\right|\right)\hspace{0.1cm}d\sigma_{\mathbb{H}^1}(\xi)=\int_{\partial B^{\mathbb{H}^1}_1(0)}u\langle\nabla_{\mathbb{H}^1}u,e_{\rho}\rangle\hspace{0.1cm}d\sigma_{\mathbb{H}^1}(\xi)\\
&+\frac{1}{2}\int_{B^{\mathbb{H}^1}_1(0)}u^{2}\diver_{\mathbb{H}^1}\left(\nabla_{\mathbb{H}^1}\left|\xi\right|_{\mathbb{H}^1}^{-2}\right)\hspace{0.1cm}d\xi+\int_{\partial B^{\mathbb{H}^1}_1(0)}u^{2}\left|\nabla_{\mathbb{H}^1}\rho\right|\hspace{0.1cm}d\sigma_{\mathbb{H}^1}(\xi)\\
&=\int_{\partial B^{\mathbb{H}^1}_1(0)}\left(u\langle\nabla_{\mathbb{H}^1}u,e_{\rho}\rangle+u^{2}\left|\nabla_{\mathbb{H}^1}\rho\right|\right)\hspace{0.1cm}d\sigma_{\mathbb{H}^1}(\xi)+\frac{1}{2}\int_{B^{\mathbb{H}^1}_1(0)}u^{2}\diver_{\mathbb{H}^1}\left(\nabla_{\mathbb{H}^1}\left|\xi\right|_{\mathbb{H}^1}^{-2}\right)\hspace{0.1cm}d\xi,
\end{align*}
which implies
\begin{align}\label{denominator-general-term-condition-monotonicity-4}
&\int_{B^{\mathbb{H}^1}_1(0)}\frac{\left|\nabla_{\mathbb{H}^1}u\right|^{2}}{\left|\xi\right|_{\mathbb{H}^1}^{2}}\hspace{0.1cm}d\xi\leq\int_{\partial B^{\mathbb{H}^1}_1(0)}\left(u\langle\nabla_{\mathbb{H}^1}u,e_{\rho}\rangle+u^{2}\left|\nabla_{\mathbb{H}^1}\rho\right|\right)\hspace{0.1cm}d\sigma_{\mathbb{H}^1}(\xi)+\frac{1}{2}\int_{B^{\mathbb{H}^1}_1(0)}u^{2}\diver_{\mathbb{H}^1}\left(\nabla_{\mathbb{H}^1}\left|\xi\right|_{\mathbb{H}^1}^{-2}\right)\hspace{0.1cm}d\xi.
\end{align}
At this point, we know that $\left|\xi\right|_{\mathbb{H}^1}^{-2}$ is, up to a multiplicative constant, the fundamental solution of $\Delta_{\mathbb{H}^1},$ and in addition
\[\Delta_{\mathbb{H}^1}=\diver_{\mathbb{H}^1}\nabla_{\mathbb{H}^1}, \]
thus
\begin{equation}\label{fundamental-solution-norm-horiz-xi-^--2}
\diver_{\mathbb{H}^1}\left(\nabla_{\mathbb{H}^1}\left|\xi\right|_{\mathbb{H}^1}^{-2}\right)=\Delta_{\mathbb{H}^1}\left|\xi\right|_{\mathbb{H}^1}^{-2}=\delta_{0},
\end{equation}
with $\delta_{0}$ the Dirac delta centered at $0.$\newline
Consequently, recalling that $u(0)=0,$ and therefore also $u^{2}(0)=0,$ we achieve
\[u^{2}\diver_{\mathbb{H}^1}\left(\nabla_{\mathbb{H}^1}\left|\xi\right|_{\mathbb{H}^1}^{-2}\right)=u^{2}\delta_{0}=0 \quad \mbox{in }B^{\mathbb{H}^1}_1(0),\]
which entails, in view of \eqref{denominator-general-term-condition-monotonicity-4},
\begin{equation}\label{denominator-general-term-condition-monotonicity-5}
\int_{B^{\mathbb{H}^1}_1(0)}\frac{\left|\nabla_{\mathbb{H}^1}u\right|^2}{\left|\xi\right|_{\mathbb{H}^1}^{2}}\hspace{0.1cm}d\xi\leq\int_{\partial B^{\mathbb{H}^1}_1(0)}\left(u\langle\nabla_{\mathbb{H}^1}u,e_{\rho}\rangle+u^{2}\left|\nabla_{\mathbb{H}^1}\rho\right|\right)\hspace{0.1cm}d\sigma_{\mathbb{H}^1}(\xi).
\end{equation}
In particular, we have obtained in Lemma \ref{norm-quad-nabla-horiz-theta-phi-rho}
\[\left|\nabla_{\mathbb{H}^1}\rho\right|^{2}=\frac{x^{2}+y^{2}}{\rho^{2}},\]
and so
\[\left|\nabla_{\mathbb{H}^1}\rho\right|=\frac{\sqrt{x^{2}+y^{2}}}{\rho},\]
which gives, from \eqref{rho-on-partial-B-H-0-1},
\begin{equation}\label{nabla-horiz-rho-on-partial-B-H-0-1}
\left|\nabla_{\mathbb{H}^1}\rho\right|=\sqrt{x^{2}+y^{2}}\quad \mbox{on }\partial B^{\mathbb{H}^1}_1(0).
\end{equation}
Substituting \eqref{nabla-horiz-rho-on-partial-B-H-0-1} in \eqref{denominator-general-term-condition-monotonicity-5}, we then get
\begin{equation}\label{denominator-general-term-condition-monotonicity-6}
\int_{B^{\mathbb{H}^1}_1(0)}\frac{\left|\nabla_{\mathbb{H}^1}u\right|^{2}}{\left|\xi\right|_{\mathbb{H}^1}^{2}}\hspace{0.1cm}d\xi\leq\int_{\partial B^{\mathbb{H}^1}_1(0)}\left(u\langle\nabla_{\mathbb{H}^1}u,e_{\rho}\rangle+u^{2}\sqrt{x^{2}+y^{2}}\right)\hspace{0.1cm}d\sigma_{\mathbb{H}^1}(\xi).
\end{equation}
Specifically, we can rewrite the right term in \eqref{denominator-general-term-condition-monotonicity-6} as
\begin{align*}
&\int_{\partial B^{\mathbb{H}^1}_1(0)}\left(u\langle\nabla_{\mathbb{H}^1}u,e_{\rho}\rangle+u^{2}\sqrt{x^{2}+y^{2}}\right)\hspace{0.1cm}d\sigma_{\mathbb{H}^1}(\xi)=\int_{\partial B^{\mathbb{H}^1}_1(0)}u\langle\nabla_{\mathbb{H}^1}u,e_{\rho}\rangle\hspace{0.05cm}\frac{\sqrt[4]{x^{2}+y^{2}}}{\sqrt[4]{x^{2}+y^{2}}}\hspace{0.1cm}d\sigma_{\mathbb{H}^1}(\xi)\\
&+\int_{\partial B^{\mathbb{H}^1}_1(0)}u^{2}\sqrt{x^{2}+y^{2}}\hspace{0.1cm}d\sigma_{\mathbb{H}^1}(\xi),
\end{align*}
which yields, by H\"older inequality,
\begin{align}\label{holder-ineq-result}
&\int_{\partial B^{\mathbb{H}^1}_1(0)}\left(u\langle\nabla_{\mathbb{H}^1}u,e_{\rho}\rangle +u^{2}\sqrt{x^{2}+y^{2}}\right)\hspace{0.1cm}d\sigma_{\mathbb{H}^1}(\xi)\leq \int_{\partial B^{\mathbb{H}^1}_1(0)}u^{2}\sqrt{x^{2}+y^{2}}\hspace{0.1cm}d\sigma_{\mathbb{H}^1}(\xi)\notag\\
&+\bigg(\int_{\partial B^{\mathbb{H}^1}_1(0)}u^{2}\sqrt{x^{2}+y^{2}}\hspace{0.1cm}d\sigma_{\mathbb{H}^1}(\xi)\bigg)^{1/2}\bigg(\int_{\partial B^{\mathbb{H}^1}_1(0)}\frac{\left|\nabla^{\rho}_{\mathbb{H}^1}u\right|^{2}}{\sqrt{x^{2}+y^{2}}}\hspace{0.1cm}d\sigma_{\mathbb{H}^1}(\xi)\bigg)^{1/2}.
\end{align}
As a consequence, by virtue of \eqref{holder-ineq-result}, we have, from \eqref{denominator-general-term-condition-monotonicity-6},
\begin{align}\label{denominator-general-term-condition-monotonicity-final}
&\int_{B^{\mathbb{H}^1}_1(0)}\frac{\left|\nabla_{\mathbb{H}^1}u\right|^{2}}{\left|\xi\right|_{\mathbb{H}^1}^{2}}\hspace{0.1cm}d\xi\leq \int_{\partial B^{\mathbb{H}^1}_1(0)}u^{2}\sqrt{x^{2}+y^{2}}\hspace{0.1cm}d\sigma_{\mathbb{H}^1}(\xi)\notag\\
&+\bigg(\int_{\partial B^{\mathbb{H}^1}_1(0)}u^{2}\sqrt{x^{2}+y^{2}}\hspace{0.1cm}d\sigma_{\mathbb{H}^1}(\xi)\bigg)^{1/2}\bigg(\int_{\partial B^{\mathbb{H}^1}_1(0)}\frac{\left|\nabla^{\rho}_{\mathbb{H}^1}u\right|^{2}}{\sqrt{x^{2}+y^{2}}}\hspace{0.1cm}d\sigma_{\mathbb{H}^1}(\xi)\bigg)^{1/2}.
\end{align}
To recap, we have increased the denominator of \eqref{general-term-condition-monotonicity} in \eqref{denominator-general-term-condition-monotonicity-final} and using this, we then obtain from \eqref{general-term-condition-monotonicity}
\begin{equation}\label{general-term-condition-monotonicity-lower-ineq-1}
\frac{\displaystyle \int_{\partial B^{\mathbb{H}^1}_1(0)}\frac{\left|\nabla_{\mathbb{H}^1}u\right|^{2}}{\sqrt{x^{2}+y^{2}}}\hspace{0.1cm}d\sigma_{\mathbb{H}^1}(\xi)}{\displaystyle\int_{B^{\mathbb{H}^1}_1(0)}\frac{\left|\nabla_{\mathbb{H}^1}u\right|^{2}}{\left|\xi\right|_{\mathbb{H}^1}^{2}}\hspace{0.1cm}d\xi}\geq\frac{A_{\rho}+A_{\varphi}}{A_{u}+A_{u}^{1/2}A_{\rho}^{1/2}}.
\end{equation}

\end{proof}

Let us introduce now the following notation:
 
\begin{equation}\label{def-lambda-varphi}
\lambda_{\varphi_{\left(\Sigma\right)}}\coloneqq \inf_{v\hspace{0.05cm}\in\hspace{0.05cm}H^{1}_{0}(\Sigma)}\frac{\displaystyle \int_{\Sigma}\frac{\left|\nabla^{\varphi}_{\mathbb{H}^1}v\right|^{2}}{\sqrt{x^{2}+y^{2}}}\hspace{0.1cm}d\sigma_{\mathbb{H}^1}(\xi)}{\displaystyle \int_{\Sigma}v^{2}\sqrt{x^{2}+y^{2}}\hspace{0.1cm}d\sigma_{\mathbb{H}^1}(\xi)},
\end{equation}
where
\begin{equation}\label{def-sigma}
\Sigma\coloneqq \partial B^{\mathbb{H}^1}_1(0)\cap\left\{u\neq 0\right\}.
\end{equation}
\begin{teo}\label{teo_teo_lb}
Let $u$ be one of the two functions of the Theorem \ref{lowerbound_f}. Then
\begin{equation}\label{general-term-condition-monotonicity-lower-bound}
\frac{\displaystyle \int_{\partial B^{\mathbb{H}^1}_1(0)}\frac{\left|\nabla_{\mathbb{H}^1}u\right|^{2}}{\sqrt{x^{2}+y^{2}}}\hspace{0.1cm}d\sigma_{\mathbb{H}^1}(\xi)}{\displaystyle \int_{B^{\mathbb{H}^1}_1(0)}\frac{\left|\nabla_{\mathbb{H}^1}u\right|^{2}}{\left|\xi\right|_{\mathbb{H}^1}^{2}}\hspace{0.1cm}d\xi}\geq 2\left(\sqrt{1+\lambda_{\varphi_{\left(\Sigma\right)}}}-1\right).
\end{equation}
\end{teo}
\begin{proof}
First of all, we remark that $A_{u}\neq 0,$ hence the right term in \eqref{general-term-condition-monotonicity-lower-ineq-1} becomes
\begin{equation}\label{equality-frac-A-rho-A-varphi-A-u}
\frac{A_{\rho}+A_{\varphi}}{A_{u}+A_{u}^{1/2}A_{\rho}^{1/2}}=\frac{\displaystyle\frac{A_{\rho}}{A_{u}}+\displaystyle\frac{A_{\varphi}}{A_{u}}}{1+\left(\displaystyle\frac{A_{\rho}}{A_{u}}\right)^{1/2}}.
\end{equation}
Substituting \eqref{equality-frac-A-rho-A-varphi-A-u} in \eqref{general-term-condition-monotonicity-lower-ineq-1} we then achieve
\begin{equation}\label{general-term-condition-monotonicity-lower-ineq-2}
\frac{\displaystyle \int_{\partial B^{\mathbb{H}^1}_1(0)}\frac{\left|\nabla_{\mathbb{H}^1}u\right|^{2}}{\sqrt{x^{2}+y^{2}}}\hspace{0.1cm}d\sigma_{\mathbb{H}^1}(\xi)}{\displaystyle \int_{B^{\mathbb{H}^1}_1(0)}\frac{\left|\nabla_{\mathbb{H}^1}u\right|^{2}}{\left|\xi\right|_{\mathbb{H}^1}^{2}}\hspace{0.1cm}d\xi}\geq \frac{\displaystyle\frac{A_{\rho}}{A_{u}}+\displaystyle\frac{A_{\varphi}}{A_{u}}}{1+\left(\displaystyle\frac{A_{\rho}}{A_{u}}\right)^{1/2}}.
\end{equation}

Furthermore, recalling that $u\in H^{1}_{0}\left(\Sigma\right)$ and \eqref{def-sigma}, we have
\[\frac{A_{\varphi}}{A_{u}}\geq \lambda_{\varphi_{\left(\Sigma\right)}},\]
which entails, in view of \eqref{general-term-condition-monotonicity-lower-ineq-2},
\begin{equation}\label{general-term-condition-monotonicity-lower-ineq-3}
\frac{\displaystyle \int_{\partial B^{\mathbb{H}^1}_1(0)}\frac{\left|\nabla_{\mathbb{H}^1}u\right|^{2}}{\sqrt{x^{2}+y^{2}}}\hspace{0.1cm}d\sigma_{\mathbb{H}^1}(\xi)}{\displaystyle\int_{B^{\mathbb{H}^1}_1(0)}\frac{\left|\nabla_{\mathbb{H}^1}u\right|^{2}}{\left|\xi\right|_{\mathbb{H}^1}^{2}}\hspace{0.1cm}d\xi}\geq \frac{\displaystyle\frac{A_{\rho}}{A_{u}}+\lambda_{\varphi_{\left(\Sigma\right)}}}{1+\left(\displaystyle\frac{A_{\rho}}{A_{u}}\right)^{1/2}}.
\end{equation}
At this point, if we call 
\begin{equation}\label{def-s}
s=\frac{A_{\rho}}{A_{u}},
\end{equation}
we can rewrite the right term in \eqref{general-term-condition-monotonicity-lower-ineq-3} as a function depending on $s,$ precisely:
\begin{equation}\label{def-F}
F(s)=\frac{s+\lambda_{\varphi_{\left(\Sigma\right)}}}{1+\sqrt{s}},\quad s\in\mathbb{R},\hspace{0.1cm} s>0.
\end{equation}
Our idea is to find the minimum of $F$ to get a lower bound of
\[ \frac{\displaystyle\int_{\partial B^{\mathbb{H}^1}_1(0)}\frac{\left|\nabla_{\mathbb{H}^1}u\right|^{2}}{\sqrt{x^{2}+y^{2}}}\hspace{0.1cm}d\sigma_{\mathbb{H}^1}(\xi)}{\displaystyle\int_{B^{\mathbb{H}^1}_1(0)}\frac{\left|\nabla_{\mathbb{H}^1}u\right|^{2}}{\left|\xi\right|_{\mathbb{H}^1}^{2}}\hspace{0.1cm}d\xi}.\]
Specifically, by virtue of \eqref{def-F}, we have
\begin{align*}
&F'(s)=\frac{1+\sqrt{s}-(s+\lambda_{\varphi_{\left(\Sigma\right)}})\frac{1}{2\sqrt{s}}}{(1+\sqrt{s})^{2}}=\frac{(1+\sqrt{s})2\sqrt{s}-s-\lambda_{\varphi_{\left(\Sigma\right)}}}{2\sqrt{s}(1+\sqrt{s})^{2}}=\frac{2\sqrt{s}+2s-s-\lambda_{\varphi_{\left(\Sigma\right)}}}{2\sqrt{s}(1+\sqrt{s})^{2}}\\
&=\frac{2\sqrt{s}+s-\lambda_{\varphi_{\left(\Sigma\right)}}}{2\sqrt{s}(1+\sqrt{s})^{2}},
\end{align*}
which gives
\begin{equation}\label{expression-F'}
F'(s)=\frac{2\sqrt{s}+s-\lambda_{\varphi_{\left(\Sigma\right)}}}{2\sqrt{s}(1+\sqrt{s})^{2}}.
\end{equation}
At this point, we notice that the denominator of \eqref{expression-F'} is always positive, so we have to study the numerator to find the minimum.\newline
In particular, it results
\begin{equation}\label{equivalence-condition-minimum}
2\sqrt{s}+s-\lambda_{\varphi_{\left(\Sigma\right)}}\geq 0 \stackrel{z=\sqrt{s}}{\Longleftrightarrow}2z+z^{2}-\lambda_{\varphi_{\left(\Sigma\right)}}\geq 0.
\end{equation}
Now, the roots of $z^{2}+2z-\lambda_{\varphi_{\left(\Sigma\right)}}$ are
\[z_{\pm}=-1\pm\sqrt{1+\lambda_{\varphi_{\left(\Sigma\right)}}},\]
but inasmuch $z>0$ from \eqref{def-F} and \eqref{equivalence-condition-minimum}, we obtain 
\[z^{2}+2z-\lambda_{\varphi_{\left(\Sigma\right)}}\geq 0\Longleftrightarrow z\geq -1+\sqrt{1+\lambda_{\varphi_{\left(\Sigma\right)}}},\]
which implies, using \eqref{equivalence-condition-minimum}, that
\[s=z^{2}=\left(-1+\sqrt{1+\lambda_{\varphi_{\left(\Sigma\right)}}}\right)^{2},\]
is the minimum point of $F.$\newline
Consequently, from \eqref{def-F}, we achieve
\begin{align*}
F(s)&=\frac{s+\lambda_{\varphi_{\left(\Sigma\right)}}}{1+\sqrt{s}}\geq F\left(\left(-1+\sqrt{1+\lambda_{\varphi_{\left(\Sigma\right)}}}\right)^{2}\right)=\frac{\left(-1+\sqrt{1+\lambda_{\varphi_{\left(\Sigma\right)}}}\right)^{2}+\lambda_{\varphi_{\left(\Sigma\right)}}}{1+\sqrt{\left(-1+\sqrt{1+\lambda_{\varphi_{\left(\Sigma\right)}}}\right)^{2}}}\\
&=\frac{1+1+\lambda_{\varphi_{\left(\Sigma\right)}}-2\sqrt{1+\lambda_{\varphi_{\left(\Sigma\right)}}}+\lambda_{\varphi_{\left(\Sigma\right)}}}{1-1+\sqrt{1+\lambda_{\varphi_{\left(\Sigma\right)}}}}=2\frac{1+\lambda_{\varphi_{\left(\Sigma\right)}}-\sqrt{1+\lambda_{\varphi_{\left(\Sigma\right)}}}}{\sqrt{1+\lambda_{\varphi_{\left(\Sigma\right)}}}}\\
&=2\frac{\sqrt{1+\lambda_{\varphi_{\left(\Sigma\right)}}}\left(\sqrt{1+\lambda_{\varphi_{\left(\Sigma\right)}}}-1\right)}{{\sqrt{1+\lambda_{\varphi_{\left(\Sigma\right)}}}}}=2\left(\sqrt{1+\lambda_{\varphi_{\left(\Sigma\right)}}}-1\right),
\end{align*}
which gives
\[\frac{s+\lambda_{\varphi_{\left(\Sigma\right)}}}{1+\sqrt{s}}\geq 2\left(\sqrt{1+\lambda_{\varphi_{\left(\Sigma\right)}}}-1\right),\]
and thus, using \eqref{def-s},
\begin{equation*}
\frac{\displaystyle\frac{A_{\rho}}{A_{u}}+\lambda_{\varphi_{\left(\Sigma\right)}}}{1+\left(\displaystyle\frac{A_{\rho}}{A_{u}}\right)^{1/2}}\geq 2\left(\sqrt{1+\lambda_{\varphi_{\left(\Sigma\right)}}}-1\right).
\end{equation*}

This fact, together with \eqref{general-term-condition-monotonicity-lower-ineq-3}, entails
\begin{equation*}\label{general-term-condition-monotonicity-lower-bound_x}
\frac{\displaystyle \int_{\partial B^{\mathbb{H}^1}_1(0)}\frac{\left|\nabla_{\mathbb{H}^1}u\right|^{2}}{\sqrt{x^{2}+y^{2}}}\hspace{0.1cm}d\sigma_{\mathbb{H}^1}(\xi)}{\displaystyle \int_{B^{\mathbb{H}^1}_1(0)}\frac{\left|\nabla_{\mathbb{H}^1}u\right|^{2}}{\left|\xi\right|_{\mathbb{H}^1}^{2}}\hspace{0.1cm}d\xi}\geq 2\left(\sqrt{1+\lambda_{\varphi_{\left(\Sigma\right)}}}-1\right).
\end{equation*}

 \end{proof}

 We show now an alternative proof of Theorem \ref{lowerbound_f}, following the idea contained in \cite{CS}.
 
 We first consider for every $\beta\in (0,1)$ the following lower bound, recalling the definition of $\lambda_{\varphi_{\left(\Sigma\right)}}$ in \eqref{def-lambda-varphi}:
 \begin{equation}\label{ottantatre}
 \begin{split}
& \int_{\partial B^{\mathbb{H}^1}_1(0)}\frac{\left|\nabla_{\mathbb{H}^1}u\right|^{2}}{\sqrt{x^{2}+y^{2}}}\hspace{0.1cm}d\sigma_{\mathbb{H}^1}(\xi)=\int_{\partial B^{\mathbb{H}^1}_1(0)}\frac{\left|\nabla^{\rho}_{\mathbb{H}^1}u\right|^{2}}{\sqrt{x^{2}+y^{2}}}\hspace{0.1cm}d\sigma_{\mathbb{H}^1}(\xi)+\displaystyle\int_{\partial B^{\mathbb{H}^1}_1(0)}\frac{\left|\nabla^{\varphi}_{\mathbb{H}^1}u\right|^{2}}{\sqrt{x^{2}+y^{2}}}\hspace{0.1cm}d\sigma_{\mathbb{H}^1}(\xi)\\
 &\geq \int_{\partial B^{\mathbb{H}^1}_1(0)}\frac{\left|\nabla^{\rho}_{\mathbb{H}^1}u\right|^{2}}{\sqrt{x^{2}+y^{2}}}\hspace{0.1cm}d\sigma_{\mathbb{H}^1}(\xi)+\lambda_{\varphi_{\left(\Sigma\right)}}\int_{\partial B^{\mathbb{H}^1}_1(0)}u^2\sqrt{x^{2}+y^{2}}\hspace{0.1cm}d\sigma_{\mathbb{H}^1}(\xi)\\
 &=\int_{\partial B^{\mathbb{H}^1}_1(0)}\frac{\left|\nabla^{\rho}_{\mathbb{H}^1}u\right|^{2}}{\sqrt{x^{2}+y^{2}}}\hspace{0.1cm}d\sigma_{\mathbb{H}^1}(\xi)+\lambda_{\varphi_{\left(\Sigma\right)}}\beta \int_{\partial B^{\mathbb{H}^1}_1(0)}u^2\sqrt{x^{2}+y^{2}}\hspace{0.1cm}d\sigma_{\mathbb{H}^1}(\xi)\\
 &+\lambda_{\varphi_{\left(\Sigma\right)}}(1-\beta) \int_{\partial B^{\mathbb{H}^1}_1(0)}u^2\sqrt{x^{2}+y^{2}}\hspace{0.1cm}d\sigma_{\mathbb{H}^1}(\xi)\\
 &\geq 2\left(\int_{\partial B^{\mathbb{H}^1}_1(0)}\frac{\left|\nabla^{\rho}_{\mathbb{H}^1}u\right|^{2}}{\sqrt{x^{2}+y^{2}}}\hspace{0.1cm}d\sigma_{\mathbb{H}^1}(\xi)\right)^{\frac{1}{2}}\left(\lambda_{\varphi_{\left(\Sigma\right)}}\beta \int_{\partial B^{\mathbb{H}^1}_1(0)}u^2\sqrt{x^{2}+y^{2}}\hspace{0.1cm}d\sigma_{\mathbb{H}^1}(\xi)\right)^{\frac{1}{2}}\\
 &+\lambda_{\varphi_{\left(\Sigma\right)}}(1-\beta) \int_{\partial B^{\mathbb{H}^1}_1(0)}u^2\sqrt{x^{2}+y^{2}}\hspace{0.1cm}d\sigma_{\mathbb{H}^1}(\xi).
 \end{split}
 \end{equation}
In view of (\ref{ottantatre}) and (\ref{denominator-general-term-condition-monotonicity-final}), it then follows, since $A_u\neq 0:$
 \begin{align*}
&\frac{\displaystyle\int_{\partial B^{\mathbb{H}^1}_1(0)}\frac{\left|\nabla_{\mathbb{H}^1}u\right|^{2}}{\sqrt{x^{2}+y^{2}}}\hspace{0.1cm}d\sigma_{\mathbb{H}^1}(\xi)}{\displaystyle \int_{B^{\mathbb{H}^1}_1(0)}\frac{\left|\nabla_{\mathbb{H}^1}u\right|^{2}}{\left|\xi\right|_{\mathbb{H}^1}^{2}}\hspace{0.1cm}d\xi}\geq \frac{2(\lambda_{\varphi_{\left(\Sigma\right)}} \beta)^{1/2}A_{\rho}^{1/2}A_{u}^{1/2}+(1-\beta)\lambda_{\varphi_{\left(\Sigma\right)}} A_{u}}{A_{u}+A_{u}^{1/2}A_{\rho}^{1/2}}\\
&=\frac{2(\lambda_{\varphi_{\left(\Sigma\right)}}\beta)^{1/2}+(1-\beta)\lambda_{\varphi_{\left(\Sigma\right)}} \frac{A_{\rho}^{1/2}}{A_{u}^{1/2}}}{1+\frac{A_{\rho}^{1/2}}{A_{u}^{1/2}}}\geq \min\{(1-\beta)\lambda_{\varphi_{\left(\Sigma\right)}}, 2(\lambda_{\varphi_{\left(\Sigma\right)}}\beta)^{1/2}\}.
\end{align*}
At this point, let $\beta\in (0,1)$  be such that  
\begin{equation}\label{beta_f}
(1-\beta)\lambda_{\varphi_{\left(\Sigma\right)}}=2(\lambda_{\varphi_{\left(\Sigma\right)}}\beta)^{1/2}.
\end{equation}
Then, denoting $\alpha\coloneqq (\lambda_{\varphi_{\left(\Sigma\right)}}\beta)^{1/2},$ we obtain that the previous relationship is satisfied whenever
\begin{equation}\label{fundamental_relation}
\alpha^2+2\alpha-\lambda_{\varphi_{\left(\Sigma\right)}}=0.
\end{equation}
We also point out that, from (\ref{beta_f}), it follows 
$$
(1-\beta)\sqrt{\lambda_{\varphi_{\left(\Sigma\right)}}}=2\sqrt{\beta},
$$
so that, denoting $\gamma=\sqrt{\beta},$ we get
$$
\sqrt{\lambda_{\varphi_{\left(\Sigma\right)}}} \gamma^2+2\gamma-\sqrt{\lambda_{\varphi_{\left(\Sigma\right)}}}=0
$$
which yields
$$
\gamma=\frac{-1\pm\sqrt{1+\lambda_{\varphi_{\left(\Sigma\right)}}}}{\sqrt{\lambda_{\varphi_{\left(\Sigma\right)}}}},
$$
but, being $\beta>0,$ it results
$$
\gamma=\frac{-1+\sqrt{1+\lambda_{\varphi_{\left(\Sigma\right)}}}}{\sqrt{\lambda_{\varphi_{\left(\Sigma\right)}}}}=-\frac{1}{\sqrt{\lambda_{\varphi_{\left(\Sigma\right)}}}}+\sqrt{1+\frac{1}{\lambda_{\varphi_{\left(\Sigma\right)}}}}.
$$
Now, the function $r\to -r+\sqrt{1+r^2}$ is positive in $[0,+\infty)$ and since $( -r+\sqrt{1+r^2})'=-1+\frac{r}{1+r^2}<0$ for every $r>0,$ this function is monotone decreasing, so that $0< -r+\sqrt{1+r^2}\leq 1 $.

As a consequence, 
there exists $\beta\in (0,1)$ given by 
$$
\beta=\left(\frac{-1+\sqrt{1+\lambda_{\varphi_{\left(\Sigma\right)}}}}{\sqrt{\lambda_{\varphi_{\left(\Sigma\right)}}}}\right)^2
$$ 
such that, when (\ref{beta_f}) is realized, it holds
$$
\min\{(1-\beta)\lambda_{\varphi_{\left(\Sigma\right)}}, 2(\lambda_{\varphi_{\left(\Sigma\right)}} \beta)^{1/2}\}=2\left(\sqrt{1+\lambda_{\varphi_{\left(\Sigma\right)}}}-1\right),
$$
as stated in Theorem \ref{lowerbound_f}.
%


\section{Further remarks about the dependence on $\varphi$}\label{phidependence}
We remarked in \eqref{lapl-horiz-f-varphi} that,
if $u=\rho^{\alpha}f(\varphi),$ then
\begin{equation}\label{lapl-horiz-f-varphi-2}
\Delta_{\mathbb{H}^1}u=\Delta_{\mathbb{H}^1}(\rho^{\alpha}f(\varphi))=\rho^{\alpha-2}\bigg(\alpha(\alpha+2)(\sin\varphi)f(\varphi)+4\sin\varphi\frac{\partial^{2}f}{\partial \varphi^{2}}\nonumber+4\cos\varphi\frac{\partial f}{\partial \varphi}\bigg).
\end{equation}
As a consequence, whenever $u=\rho^{\alpha}f(\varphi)$ and 
$$
\alpha(\alpha+2)(\sin\varphi)f(\varphi)+4\sin\varphi\frac{\partial^{2}f}{\partial \varphi^{2}}+4\cos\varphi\frac{\partial f}{\partial \varphi}=0,$$
then $\Delta_{\mathbb{H}^1}u=0.$
This equation, in particular, may be reduced to the following one
$$
(\sin(\varphi)f'(\varphi))'+\alpha(\alpha+2)\sin(\varphi)f(\varphi)=0.
$$
Thus, denoting by $\mathcal{L}f\coloneqq (\sin(\varphi)f'(\varphi))',$ and considering
the following eigenvalues problem
\begin{equation}\label{eigenvp}
\left\{ 
\begin{array}{l}
\mathcal{L}f+\lambda\sin (\varphi)f=0,\quad\varphi_1<\varphi<\varphi_2\\
f(\varphi_1)=0\\
f(\varphi_2)=0,
\end{array}
\right.
\end{equation}
it results that $\alpha$ has to satisfy the following relationship
$$\alpha(\alpha+2)=\lambda,$$
which is exactly the same one obtained in (\ref{fundamental_relation}).

On the other hand, by writing $f=f(\varphi)$ and $\frac{\partial}{\partial\varphi}\left(\sin\varphi\frac{\partial f}{\partial \varphi}\right)=(\sin\varphi f')',$ since $f$ is a function depending only on $\varphi,$ we have
$$-4(\sin\varphi f')'=\alpha(\alpha+2)(\sin\varphi)f,$$
and multiplying both the terms of the equality by $\eta$ sufficiently smooth with $\eta\left(\varphi_0\right)=0,$ $\varphi_0\in (0,\frac{\pi}{2}],$
\begin{equation}\label{weak-equality-lapl-horiz-1x}
\alpha(\alpha+2)(\sin\varphi)f\eta=-4(\sin\varphi f')'\eta.
\end{equation}
Integrating over $\left[0,\varphi_0\right]$ the equality in \eqref{weak-equality-lapl-horiz-1x}, we then obtain
\begin{equation}\label{weak-equality-lapl-horiz-2x}
\alpha(\alpha+2)\int_{0}^{\varphi_0}(\sin\varphi)f\eta\hspace{0.1cm}d\varphi=-4\int_{0}^{\varphi_0}(\sin\varphi f')'\eta\hspace{0.1cm}d\varphi.
\end{equation}
In particular, if we choose $\eta=f,$ we get, from \eqref{weak-equality-lapl-horiz-2x}, using the Theorem of Integration by Parts:
\begin{align*}
&\alpha(\alpha+2)\int_{0}^{\varphi_0}(\sin\varphi)f^{2}\hspace{0.1cm}d\varphi=-4\int_{0}^{\varphi_0}(\sin\varphi f')'f\hspace{0.1cm}d\varphi=-4\bigg(\bigg[(\sin\varphi f')f\bigg]^{\varphi=\varphi_0}_{\varphi=0}-\int_{0}^{\varphi_0}\sin\varphi f'f'\hspace{0.1cm}d\varphi\bigg)\\
&=-4\bigg(\sin\left(\varphi_0\right)f'\left(\varphi_0\right)f\left(\varphi_0\right)-\sin(0)f'(0)f(0)-\int_{0}^{\varphi_0}\sin\varphi (f')^{2}\hspace{0.1cm}d\varphi\bigg).
\end{align*}
This implies that, because $\sin(0)=0$ and recalling that $f\left(\varphi_0\right)=0,$ by virtue of the choice of $f$,
\begin{equation}\label{weak-equality-lapl-horiz-3x}
\alpha(\alpha+2)\int_{0}^{\varphi_0}(\sin\varphi)f^{2}\hspace{0.1cm}d\varphi=4\int_{0}^{\varphi_0}\sin\varphi(f')^{2}\hspace{0.1cm}d\varphi.
\end{equation}
In addition, in view of \eqref{weak-equality-lapl-horiz-3x}, we also have
\begin{equation}\label{weak-equality-lapl-horiz-4x}
\alpha(\alpha+2)=\frac{4\displaystyle\int_{0}^{\varphi_0}\sin\varphi (f')^{2}\hspace{0.1cm}d\varphi}{\displaystyle\int_{0}^{\varphi_0}(\sin\varphi)f^{2}\hspace{0.1cm}d\varphi}.
\end{equation}
On the other hand, performing a change of variable $\tau= \pi-\varphi,$ we get
\begin{equation}\label{weak-equality-lapl-horiz-5x}
\begin{split}
\alpha(\alpha+2)&=\frac{-4\displaystyle\int_{\pi}^{\pi-\varphi_0}\sin(\pi-\tau) (f')^{2}(\pi-\tau)\hspace{0.1cm}d\tau}{-\displaystyle\int_{\pi}^{\pi-\varphi_0}\sin(\pi-\tau)f^{2}(\pi-\tau)\hspace{0.1cm}d\tau}=\frac{4\displaystyle\int^{\pi}_{\pi-\varphi_0}\sin(\tau) (f')^{2}(\pi-\tau)\hspace{0.1cm}d\tau}{\displaystyle\int^{\pi}_{\pi-\varphi_0}\sin(\tau)f^{2}(\pi-\tau)\hspace{0.1cm}d\tau}.
\end{split}
\end{equation}
In this way, we achieve that 
\begin{equation}\label{weak-equality-lapl-horiz-4xx}
\alpha_1(\varphi_0)(\alpha_1(\varphi_0)+2)=\frac{4\displaystyle\int_{0}^{\varphi_0}\sin\varphi (f')^{2}\hspace{0.1cm}d\varphi}{\displaystyle\int_{0}^{\varphi_0}(\sin\varphi)f^{2}\hspace{0.1cm}d\varphi}
\end{equation}
and
\begin{equation}\label{weak-equality-lapl-horiz-4xxx}
\begin{split}
\alpha_1(\eta_0)(\alpha_1(\eta_0)+2)
&=\frac{4\displaystyle\int^{\pi}_{\eta_0}\sin(\tau) (f')^{2}(\pi-\tau)\hspace{0.1cm}d\tau}{\displaystyle\int^{\pi}_{\eta_0}\sin(\tau)f^{2}(\pi-\tau)\hspace{0.1cm}d\tau}
\end{split}
\end{equation}
where $\varphi_0+\eta_0=\pi.$ 
\begin{lem}\label{go_go_1}
The function
\[
G(\varphi)=\alpha_1(\varphi)(\alpha_1(\varphi)+2)+\alpha_1(\pi-\varphi)(\alpha_1(\pi-\varphi)+2),\quad \varphi\in [0,\pi]
\]
is symmetric with respect to $\frac{\pi}{2}.$
\end{lem}
\begin{proof}
For every $\varphi_0\in [0,\frac{\pi}{2}],$ we have
\begin{align*}
&G(\varphi_0)=\alpha_1(\varphi_0)(\alpha_1(\varphi_0)+2)+\alpha_1(\pi-\varphi_0)(\alpha_1(\pi-\varphi_0)+2)\\
&=\alpha_1(\pi-(\pi-\varphi_0))(\alpha_1(\pi-(\pi-\varphi_0))+2)+\alpha_1(\pi-\varphi_0)(\alpha_1(\pi-\varphi_0)+2)\\
&=\alpha_1(\pi-\varphi_0)(\alpha_1(\pi-\varphi_0)+2)+\alpha_1(\pi-(\pi-\varphi_0))(\alpha_1(\pi-(\pi-\varphi_0))+2)=G(\pi-\varphi_0).
\end{align*}
\end{proof}

We note that determining $\alpha$ and $f$ so that
$$
\left\{
\begin{array}{l}
4\sin\varphi\frac{\partial^{2}f}{\partial \varphi^{2}}+4\cos\varphi\frac{\partial f}{\partial \varphi}+\alpha(\alpha+2)(\sin\varphi)f(\varphi)=0\\
f'(0)=0\\
f(\varphi_1)=0,
\end{array}
\right.$$
we find $\mathbb{H}^1-$harmonic functions on sets radially symmetric with respect to the $t$-axis. For instance, if $\varphi_1=\frac{\pi}{2},$ then $\alpha=2$  and $f(\varphi)=\cos(\varphi)$ is the solution. In particular $u=\rho^2\cos(\varphi)$ is $\mathbb{H}^1-$harmonic in $\{(x,y,t)\in \mathbb{H}^1:\hspace{0.2cm} t\geq 0\}.$

Let us denote, at this point, by $\lambda_{\varphi_{1}}(\varphi_2)$ the eigenvalue of the problem (\ref{eigenvp}).  Moreover, let
\begin{equation}\label{h_func}
\begin{split}
&h(\varphi)\coloneqq 2(\sqrt{1+\lambda_0(\varphi)}-1)+2(\sqrt{1+\lambda_0(\pi-\varphi)}-1).
\end{split}
\end{equation}
Then, we get the following result.
\begin{lem}\label{go_go_2}
In case the minimum value corresponds to the configuration in which the Koranyi ball is split in two parts by the plane $t=0,$ then
\begin{equation}\label{half}\begin{split}
&
\lambda_0(\varphi)+ \lambda_{\varphi} (\pi)
=\lambda_0(\varphi)+ \lambda_{0} (\pi-\varphi)\geq 2\lambda_0(\frac{\pi}{2})=16
\end{split}
\end{equation}
and
 \begin{equation*}
\begin{split}
&h(\varphi)\geq 2(\sqrt{18}-2)> 2(4-2)=4.
\end{split}
\end{equation*}
In general supposing only that 
\begin{equation}\label{half3}\begin{split}
&
\lambda_0(\varphi)+ \lambda_0 (\pi-\varphi)\geq q>0,
\end{split}
\end{equation}
for a value $\varphi\not= \frac{\pi}{2}$ we obtain
 \begin{equation*}
\begin{split}
&h(\varphi)\geq 2(\sqrt{2+q}-2).
\end{split}
\end{equation*}
\end{lem}
\begin{proof}
We know that $\sqrt{a+b}\leq \sqrt{a}+\sqrt{b}\leq \sqrt{2}\sqrt{a+b}$ so that from \eqref{h_func} it follows
\begin{equation}\label{mah2}
\begin{split}
&h(\varphi)\geq 2(\sqrt{2+\lambda_0(\varphi)+\lambda_0(\pi-\varphi)}-2)\geq 2(\sqrt{2+q}-2).
\end{split}
\end{equation}
Thus, if $q$ were $16,$ that is if we suppose that the minimum of (\ref{half}) is realized for the half ball split by the plane $t=0,$ then
\begin{equation}\label{mah23}
\begin{split}
&h(\varphi)> 2(\sqrt{18}-2)\geq 2(4-2)=4.
\end{split}
\end{equation}
\end{proof}
%
Thus, we conclude that if the minimum for $h$ is realized for $t=\frac{\pi}{2},$ then, in order to obtain a monotone increasing formula, we have to assume that $\beta\leq 8.$  
In fact, as we have seen in \eqref{rho-2-cos-varphi-H-harmonic}, since $\rho^{2}\cos\varphi$ is $\mathbb{H}^1$-harmonic, where $\alpha=2$ and $f(\varphi)=\cos\varphi,$ with $\cos\left(\dfrac{\pi}{2}\right)=0,$ we obtain repeating the same argument used to achieve \eqref{weak-equality-lapl-horiz-4x}, that
$$8=\frac{4\displaystyle \int_{0}^{\frac{\pi}{2}}\sin\varphi ((\cos\varphi)')^{2}\hspace{0.1cm}d\varphi}{\displaystyle \int_{0}^{\frac{\pi}{2}}\sin\varphi \cos^{2}(\varphi)\hspace{0.1cm}d\varphi}.$$
On the other hand, by fixing $\beta =4$ our previous computations hold and if the minimum of $h$ were realized for the value $16,$ the monotonicity formula will hold for $\beta=4.$
\section{Computation having the dependence on $\varphi$ and $\theta$}
We achieved in Lemma \ref{lapl-horiz-f-theta-phi} that if $u=\rho^{\alpha}f(\theta,\varphi),$ 
\begin{align*}
&\Delta_{\mathbb{H}^1}u=\Delta_{\mathbb{H}^1}\left(\rho^{\alpha}f(\theta,\varphi)\right)=\rho^{\alpha-2}\bigg(\alpha(\alpha+2)(\sin\varphi)f(\theta,\varphi)-2\alpha(\cos\varphi)\frac{\partial f}{\partial \theta}+\frac{1}{\sin\varphi}\frac{\partial f^{2}}{\partial \theta^{2}}\\
&+4\sin\varphi\frac{\partial ^{2}f}{\partial \varphi \partial \theta}+4\sin\varphi\frac{\partial^{2}f}{\partial \varphi^{2}}+4\cos\varphi\frac{\partial f}{\partial \varphi}\bigg),
\end{align*}
which also implied
\begin{align}\label{laplacian-horiz-rho-^-alpha-f-theta-phi-partial-B-H-0-1-bis}
&\Delta_{\mathbb{H}^1}u\restrict{\partial B^{\mathbb{H}^1}_1(0)}=\Delta_{\mathbb{H}^1}\left(\rho^{\alpha}f(\theta,\varphi)\right)\restrict{\partial B^{\mathbb{H}^1}_1(0)}=\alpha(\alpha+2)(\sin\varphi)f(\theta,\varphi)-2\alpha(\cos\varphi)\frac{\partial f}{\partial \theta}+\frac{1}{\sin\varphi}\frac{\partial f^{2}}{\partial \theta^{2}}\nonumber\\
&+4\sin\varphi\frac{\partial ^{2}f}{\partial \varphi \partial \theta}+4\sin\varphi\frac{\partial^{2}f}{\partial \varphi^{2}}+4\cos\varphi\frac{\partial f}{\partial \varphi},
\end{align}
since $\rho=1$ on $\partial B^{\mathbb{H}^1}_1(0).$\newline

Let now
\begin{equation}\label{matrix-divergence-form_x}
A(\theta,\varphi)=
\begin{bmatrix}
\displaystyle\frac{1}{\sin\varphi}&(4+2\alpha)\sin\varphi\\
\\
-2\alpha\sin\varphi& 4\sin\varphi
\end{bmatrix}.
\end{equation}
and we define
\begin{align}\label{laplacian-horiz-rho-^-alpha-f-theta-phi-partial-B-H-0-1-equal-to-0-bis-3} 
&\mathcal{L}_{\theta,\varphi}\coloneqq \diver_{\theta,\varphi}\left(A(\theta,\varphi)\nabla_{\theta,\varphi}\right)=\diver_{\theta,\varphi}\left(
\begin{bmatrix}
\displaystyle\frac{1}{\sin\varphi}&(4+2\alpha)\sin\varphi\\
\\
-2\alpha\sin\varphi& 4\sin\varphi
\end{bmatrix}
\begin{bmatrix}
\displaystyle\frac{\partial}{\partial \theta}\\
\\
\displaystyle\frac{\partial}{\partial \varphi}
\end{bmatrix}\right).
\end{align}
\begin{lem}
Let  $\Omega_{\theta,\varphi}\subseteq [0,2\pi]\times[0,\pi]$ and $T(\Omega_{\theta,\varphi})\subset \partial B^{\mathbb{H}^1}_1(0).$ If $u=\rho^\alpha f(\theta, \varphi)$ is solution  of $\Delta_{\mathbb{H}^1}u=0$ in $\delta_{R}(T(\Omega_{\theta,\varphi})),$ $R>0,$ then
\[ \mathcal{L}_{\theta,\varphi}f=-\alpha(\alpha+2)(\sin\varphi)f\quad\mbox{in }\Omega_{\theta,\varphi}.\]
Furthermore, it results
\begin{equation}\label{weak-equation-theta-phi-6_x}
\alpha(\alpha+2)=\frac{\displaystyle\int_{\Omega_{\theta,\varphi}}\bigg(\frac{1}{\sin\varphi}\bigg(\frac{\partial f}{\partial \theta}\bigg)^{2}+4\sin\varphi\frac{\partial f}{\partial \theta}\frac{\partial f}{\partial \varphi}+4\sin\varphi\bigg(\frac{\partial f}{\partial \varphi}\bigg)^{2}\bigg)\hspace{0.1cm}d\theta d\varphi}{\displaystyle\int_{\Omega_{\theta,\varphi}}(\sin\varphi)f^{2}\hspace{0.1cm}d\theta d\varphi}.
\end{equation}
\end{lem}

\begin{proof}
First of all, we point out that \eqref{laplacian-horiz-rho-^-alpha-f-theta-phi-partial-B-H-0-1-bis} yields that if $\Delta_{\mathbb{H}^1}u\restrict{\partial B^{\mathbb{H}^1}_1(0)}=0,$ with $u=\rho^{\alpha}f(\theta,\varphi),$ then
\begin{align}\label{laplacian-horiz-rho-^-alpha-f-theta-phi-partial-B-H-0-1-equal-to-0-bis-1}
&\alpha(\alpha+2)(\sin\varphi)f(\theta,\varphi)-2\alpha(\cos\varphi)\frac{\partial f}{\partial \theta}+\frac{1}{\sin\varphi}\frac{\partial^{2} f}{\partial \theta^{2}}+4\sin\varphi\frac{\partial^{2}f}{\partial\varphi\partial\theta}\nonumber\\
&+4\sin\varphi\frac{\partial^{2}f}{\partial \varphi^{2}}+4\cos\varphi\frac{\partial f}{\partial \varphi}=0.
\end{align}
So, if we denote
\begin{align}\label{definition-operator-theta-phi}
\mathcal{L}_{\theta,\varphi}\coloneqq -2\alpha(\cos\varphi)\frac{\partial }{\partial \theta}+\frac{1}{\sin\varphi}\frac{\partial^{2}}{\partial \theta^{2}}+4\sin\varphi\frac{\partial^{2}}{\partial\varphi\partial\theta}+4\sin\varphi\frac{\partial^{2}}{\partial \varphi^{2}}+4\cos\varphi\frac{\partial}{\partial \varphi},
\end{align}
\eqref{laplacian-horiz-rho-^-alpha-f-theta-phi-partial-B-H-0-1-equal-to-0-bis-1} implies, writing $f(\theta,\varphi)=f,$
\begin{equation}\label{laplacian-horiz-rho-^-alpha-f-theta-phi-partial-B-H-0-1-equal-to-0-bis-2}
\alpha(\alpha+2)(\sin\varphi)f+\mathcal{L}_{\theta,\varphi}f=0.
\end{equation}
We would like to see, at this point, if $\mathcal{L}_{\theta,\varphi}$ can be written in divergence form, i.e.
\[ \mathcal{L}_{\theta,\varphi}=\diver_{\theta,\varphi}\left(A(\theta,\varphi)\nabla_{\theta,\varphi}\right),\]
where $A(\theta,\varphi)$ is a matrix-valued function.\newline
In particular, if we take
\begin{equation}\label{matrix-divergence-form}
A(\theta,\varphi)=
\begin{bmatrix}
\displaystyle\frac{1}{\sin\varphi}&(4+2\alpha)\sin\varphi\\
\\
-2\alpha\sin\varphi& 4\sin\varphi
\end{bmatrix},
\end{equation}
we have
\begin{align*}
&\diver_{\theta,\varphi}\left(A(\theta,\varphi)\nabla_{\theta,\varphi}\right)=\diver_{\theta,\varphi}\left(
\begin{bmatrix}
\displaystyle\frac{1}{\sin\varphi}&(4+2\alpha)\sin\varphi\\
\\
-2\alpha\sin\varphi& 4\sin\varphi
\end{bmatrix}
\begin{bmatrix}
\displaystyle\frac{\partial}{\partial \theta}\\
\\
\displaystyle\frac{\partial}{\partial \varphi}
\end{bmatrix}\right)\\
&=\diver_{\theta,\varphi}\left(\frac{1}{\sin\varphi}\frac{\partial}{\partial \theta}+(4+2\alpha)\sin\varphi\frac{\partial}{\partial \varphi},-2\alpha\sin\varphi\frac{\partial}{\partial \theta}+4\sin\varphi\frac{\partial}{\partial \varphi}\right)\\
&=\frac{\partial}{\partial \theta}\left(\frac{1}{\sin\varphi}\frac{\partial}{\partial \theta}+(4+2\alpha)\sin\varphi\frac{\partial}{\partial \varphi}\right)+\frac{\partial}{\partial \varphi}\left(-2\alpha\sin\varphi\frac{\partial}{\partial \theta}+4\sin\varphi\frac{\partial}{\partial \varphi}\right)\\
&=\frac{1}{\sin\varphi}\frac{\partial^{2}}{\partial \theta^{2}}+(4+2\alpha)\sin\varphi\frac{\partial^{2}}{\partial \theta \partial \varphi}-2\alpha\cos\varphi\frac{\partial}{\partial \theta}-2\alpha\sin\varphi\frac{\partial^{2}}{\partial \varphi \partial \theta}+4\cos\varphi\frac{\partial}{\partial \varphi}\\
&+4\sin\varphi\frac{\partial^{2}}{\partial \varphi^{2}}=\frac{1}{\sin\varphi}\frac{\partial^{2}}{\partial \theta^{2}}+4\sin\varphi\frac{\partial^{2}}{\partial \theta \partial \varphi}-2\alpha\cos\varphi\frac{\partial}{\partial \theta}+4\cos\varphi\frac{\partial}{\partial \varphi}+4\sin\varphi\frac{\partial^{2}}{\partial \varphi^{2}}=\mathcal{L}_{\theta,\varphi},
\end{align*}
in other words
\[\mathcal{L}_{\theta,\varphi}=\diver_{\theta,\varphi}\left(A_{\theta,\varphi}\nabla_{\theta,\varphi}\right),\]
with $A_{\theta,\varphi}$ defined in \eqref{matrix-divergence-form}.\newline
From this fact, we obtain, by virtue of \eqref{laplacian-horiz-rho-^-alpha-f-theta-phi-partial-B-H-0-1-equal-to-0-bis-2},
\begin{equation}\label{laplacian-horiz-rho-^-alpha-f-theta-phi-partial-B-H-0-1-equal-to-0-bis-3_x}\alpha(\alpha+2)(\sin\varphi)f+\diver_{\theta,\varphi}\left(A_{\theta,\varphi}\nabla_{\theta,\varphi}f\right)=0.
\end{equation}
Therefore, 
if we consider $T(\Omega_{\theta,\varphi})\subseteq \partial B^{\mathbb{H}^1}_1(0),$ \eqref{laplacian-horiz-rho-^-alpha-f-theta-phi-partial-B-H-0-1-equal-to-0-bis-3_x} entails
\[ \diver_{\theta,\varphi}\left(A_{\theta,\varphi}\nabla_{\theta,\varphi}f\right)=-\alpha(\alpha+2)(\sin\varphi)f\quad\mbox{in }\Omega_{\theta,\varphi},\]
and, multiplying both the terms of the equality by $\eta$ sufficiently smooth, with compact support in $\Omega_{\theta,\varphi},$ we get
\begin{equation}\label{weak-equation-theta-phi-1}
\diver_{\theta,\varphi}\left(A_{\theta,\varphi}\nabla_{\theta,\varphi}f\right)\eta=-\alpha(\alpha+2)(\sin\varphi)f\eta\quad\mbox{in }\Omega_{\theta,\varphi}.
\end{equation}
Integrating the equality in \eqref{weak-equation-theta-phi-1} over $\Omega_{\theta,\varphi},$ we then achieve
\begin{equation}\label{weak-equation-theta-phi-2}
\int_{\Omega_{\theta,\varphi}}\diver_{\theta,\varphi}\left(A_{\theta,\varphi}\nabla_{\theta,\varphi}f\right)\eta\hspace{0.1cm}d\theta d\varphi=-\alpha(\alpha+2)\int_{\Omega_{\theta,\varphi}}(\sin\varphi)f\eta\hspace{0.1cm}d\theta d\varphi.
\end{equation}
In particular, if we choose $\eta=f,$ we have, in view of \eqref{weak-equation-theta-phi-2},
\begin{equation}\label{weak-equation-theta-phi-3}
\int_{\Omega_{\theta,\varphi}}\diver_{\theta,\varphi}\left(A_{\theta,\varphi}\nabla_{\theta,\varphi}f\right)f\hspace{0.1cm}d\theta d\varphi=-\alpha(\alpha+2)\int_{\Omega_{\theta,\varphi}}(\sin\varphi)f^{2}\hspace{0.1cm}d\theta d\varphi.
\end{equation}
In addition, by the Divergence Theorem, we obtain
\begin{align*}
&\int_{\Omega_{\theta,\varphi}}\diver_{\theta,\varphi}\left(A_{\theta,\varphi}\nabla_{\theta,\varphi}f\right)f\hspace{0.1cm}d\theta d\varphi=\int_{\partial \Omega_{\theta,\varphi}}\langle fA_{\theta,\varphi}\nabla_{\theta,\varphi}f, \nu\rangle\hspace{0.1cm}d\sigma(\theta,\varphi)-\int_{\Omega_{\theta,\varphi}}\langle A_{\theta,\varphi}\nabla_{\theta,\varphi}f,\nabla_{\theta,\varphi}f\rangle \hspace{0.1cm}d\theta d\varphi\\
&=\int_{\partial \Omega_{\theta,\varphi}}f\langle A_{\theta,\varphi}\nabla_{\theta,\varphi}f, \nu\rangle\hspace{0.1cm}d\sigma(\theta,\varphi)-\int_{\Omega_{\theta,\varphi}}\langle A_{\theta,\varphi}\nabla_{\theta,\varphi}f,\nabla_{\theta,\varphi}f\rangle\hspace{0.1cm}d\theta d\varphi=-\int_{\Omega_{\theta,\varphi}}\langle A_{\theta,\varphi}\nabla_{\theta,\varphi}f,\nabla_{\theta,\varphi}f\rangle\hspace{0.1cm}d\theta d\varphi,
\end{align*}
which implies
\begin{equation}\label{result-divergence-theorem-bis}
\int_{\Omega_{\theta,\varphi}}\diver_{\theta,\varphi}\left(A_{\theta,\varphi}\nabla_{\theta,\varphi}f\right)f\hspace{0.1cm}d\theta d\varphi=-\int_{\Omega_{\theta,\varphi}}\langle A_{\theta,\varphi}\nabla_{\theta,\varphi}f,\nabla_{\theta,\varphi}f\rangle\hspace{0.1cm}d\theta d\varphi,
\end{equation}
because $f$ has compact support in $\Omega_{\theta,\varphi},$ by the choice of $f.$\newline
As a consequence, substituting \eqref{result-divergence-theorem-bis} in \eqref{weak-equation-theta-phi-3}, we then have
\begin{equation}\label{weak-equation-theta-phi-4}
\int_{\Omega_{\theta,\varphi}}A_{\theta,\varphi}\nabla_{\theta,\varphi}f\cdot \nabla_{\theta,\varphi}f\hspace{0.1cm}d\theta d\varphi=\alpha(\alpha+2)\int_{\Omega_{\theta,\varphi}}(\sin\varphi)f^{2}\hspace{0.1cm}d\theta d\varphi.
\end{equation}
Now, let us note that with $4+2\alpha\neq-2\alpha,$ namely $\alpha\neq-1,$ we have, from \eqref{matrix-divergence-form}, that $A_{\theta,\varphi}$ is not symmetric.\newline
Hence, we can consider the symmetrized form of $A_{\theta,\varphi},$
\begin{equation*}
A^{S}_{\theta,\varphi}\coloneqq \frac{A_{\theta,\varphi}+A_{\theta,\varphi}^{T}}{2},
\end{equation*}
and we observe that
\begin{align*}
&\langle A^{S}_{\theta,\varphi}v, v\rangle=\langle\bigg(\frac{A_{\theta,\varphi}+A_{\theta,\varphi}^{T}}{2}\bigg)v, v\rangle=\frac{1}{2}\langle A_{\theta,\varphi}v+A_{\theta,\varphi}^{T}v, v\rangle=\frac{1}{2}(\langle A_{\theta,\varphi}v,v\rangle+\langle A_{\theta,\varphi}^{T}v,v\rangle)\\
&=\frac{1}{2}(\langle A_{\theta,\varphi}v,v\rangle+\langle v, A_{\theta,\varphi}v\rangle)=\frac{1}{2}(\langle A_{\theta,\varphi}v,v\rangle+\langle A_{\theta,\varphi}v,v\rangle)=\langle A_{\theta,\varphi}v,v\rangle,\quad v\in \mathbb{R}^{2},
\end{align*}
i.e.
\begin{equation}\label{quadratic-form-symmetrized-A-theta-phi}
\langle A^{S}_{\theta,\varphi}v, v\rangle=\langle A_{\theta,\varphi}v,v\rangle,\quad v\in \mathbb{R}^{2}.
\end{equation}
Using \eqref{quadratic-form-symmetrized-A-theta-phi}, we then achieve from \eqref{weak-equation-theta-phi-4}
\begin{equation}\label{weak-equation-theta-phi-5}
\int_{\Omega_{\theta,\varphi}}\langle A^{S}_{\theta,\varphi}\nabla_{\theta,\varphi}f,\nabla_{\theta,\varphi}f\rangle\hspace{0.1cm}d\theta d\varphi=\alpha(\alpha+2)\int_{\Omega_{\theta,\varphi}}(\sin\varphi)f^{2}\hspace{0.1cm}d\theta d\varphi.
\end{equation}
At this point, we look for an explicit expression for $A^{S}_{\theta,\varphi}.$\newline
Specifically, we have
\begin{align*}
&A^{S}_{\theta,\varphi}=\frac{A_{\theta,\varphi}+A_{\theta,\varphi}^{T}}{2}=\frac{1}{2}\left(
\begin{bmatrix}
\displaystyle \frac{1}{\sin\varphi}&(4+2\alpha)\sin\varphi\\
\\
-2\alpha\sin\varphi &4\sin\varphi\\
\end{bmatrix}+
\begin{bmatrix}
\displaystyle \frac{1}{\sin\varphi}&-2\alpha\sin\varphi\\
\\
(4+2\alpha)\sin\varphi & 4\sin\varphi
\end{bmatrix}\right)\\
&=\frac{1}{2}\begin{bmatrix}
\displaystyle\frac{2}{\sin\varphi}&4\sin\varphi\\
\\
4\sin\varphi & 8\sin\varphi
\end{bmatrix}=
\begin{bmatrix}
\displaystyle \frac{1}{\sin\varphi}&2\sin\varphi\\
\\
2\sin\varphi & 4\sin\varphi
\end{bmatrix},
\end{align*}
which yields
\begin{equation}\label{symmetrized-form-A-theta-phi}
A^{S}_{\theta,\varphi}=
\begin{bmatrix}
\displaystyle \frac{1}{\sin\varphi}&2\sin\varphi\\
\\
2\sin\varphi & 4\sin\varphi
\end{bmatrix}.
\end{equation}
Consequently, according to \eqref{symmetrized-form-A-theta-phi}, we get
\begin{align*}
&\langle A^{S}_{\theta,\varphi}\nabla_{\theta,\varphi}f,\nabla_{\theta,\varphi}f\rangle=\langle
\begin{bmatrix}
\displaystyle\frac{1}{\sin\varphi}&2\sin\varphi\\
\\
2\sin\varphi & 4\sin\varphi
\end{bmatrix}
\begin{bmatrix}
\displaystyle \frac{\partial f}{\partial \theta}\\
\\
\displaystyle \frac{\partial f}{\partial \varphi}
\end{bmatrix},
\begin{bmatrix}
\displaystyle \frac{\partial f}{\partial \theta}\\
\\
\displaystyle \frac{\partial f}{\partial \varphi}
\end{bmatrix}\rangle
=\langle\begin{bmatrix}
\displaystyle\frac{1}{\sin\varphi}\frac{\partial f}{\partial \theta}+2\sin\varphi\frac{\partial f}{\partial \varphi}\\
\\
\displaystyle 2\sin\varphi\frac{\partial f}{\partial \theta}+4\sin\varphi\frac{\partial f}{\partial \varphi}
\end{bmatrix},
\begin{bmatrix}
\displaystyle \frac{\partial f}{\partial \theta}\\
\\
\displaystyle \frac{\partial f}{\partial \varphi}
\end{bmatrix}\rangle\\
&=\frac{1}{\sin\varphi}\left(\frac{\partial f}{\partial \theta}\right)^{2}+2\sin\varphi\frac{\partial f}{\partial \varphi}\frac{\partial f}{\partial\theta}+2\sin\varphi\frac{\partial f}{\partial \theta}\frac{\partial f}{\partial \varphi}+4\sin\varphi\left(\frac{\partial f}{\partial \varphi}\right)^{2}\\
&=\frac{1}{\sin\varphi}\left(\frac{\partial f}{\partial \theta}\right)^{2}+4\sin\varphi\frac{\partial f}{\partial \theta}\frac{\partial f}{\partial \varphi}+4\sin\varphi\left(\frac{\partial f}{\partial \varphi}\right)^{2},
\end{align*}
i.e.
\begin{equation}\label{explicit-expression-quadratic-form-symmetrized-A-theta-phi}
\langle A^{S}_{\theta,\varphi}\nabla_{\theta,\varphi}f,\nabla_{\theta,\varphi}f\rangle =\frac{1}{\sin\varphi}\left(\frac{\partial f}{\partial \theta}\right)^{2}+4\sin\varphi\frac{\partial f}{\partial \theta}\frac{\partial f}{\partial \varphi}+4\sin\varphi\left(\frac{\partial f}{\partial \varphi}\right)^{2}.
\end{equation}
Substituting \eqref{explicit-expression-quadratic-form-symmetrized-A-theta-phi} in \eqref{weak-equation-theta-phi-5}, we then have
\begin{align*}
&\int_{\Omega_{\theta,\varphi}}\bigg(\frac{1}{\sin\varphi}\bigg(\frac{\partial f}{\partial \theta}\bigg)^{2}+4\sin\varphi\frac{\partial f}{\partial \theta}\frac{\partial f}{\partial \varphi}+4\sin\varphi\bigg(\frac{\partial f}{\partial \varphi}\bigg)^{2}\bigg)\hspace{0.1cm}d\theta d\varphi=\alpha(\alpha+2)\int_{\Omega_{\theta,\varphi}}(\sin\varphi)f^{2}\hspace{0.1cm}d\theta d\varphi,
\end{align*}
which gives
\begin{equation}\label{weak-equation-theta-phi-6}
\alpha(\alpha+2)=\frac{\displaystyle\int_{\Omega_{\theta,\varphi}}\bigg(\frac{1}{\sin\varphi}\bigg(\frac{\partial f}{\partial \theta}\bigg)^{2}+4\sin\varphi\frac{\partial f}{\partial \theta}\frac{\partial f}{\partial \varphi}+4\sin\varphi\bigg(\frac{\partial f}{\partial \varphi}\bigg)^{2}\bigg)\hspace{0.1cm}d\theta d\varphi}{\displaystyle\int_{\Omega_{\theta,\varphi}}(\sin\varphi)f^{2}\hspace{0.1cm}d\theta d\varphi}.
\end{equation}
\end{proof}
At this point, we remark that if $\alpha=1$ and $f=\sqrt{\sin\varphi}\cos\theta,$ we get, in view of \eqref{laplacian-horiz-rho-^-alpha-f-theta-phi-partial-B-H-0-1-bis},
\begin{align*}
&\Delta_{\mathbb{H}^1}\left(\rho \sqrt{\sin\varphi}\cos\theta\right)\restrict{\partial B^{\mathbb{H}^1}_1(0)}=3\sin\varphi\sqrt{\sin\varphi}\cos\theta-2\cos\varphi\frac{\partial}{\partial \theta}\left(\sqrt{\sin\varphi}\cos\theta\right)\\
&+\frac{1}{\sin\varphi}\frac{\partial^{2}}{\partial \theta^{2}}\left(\sqrt{\sin\varphi}\cos\theta\right)+4\sin\varphi\frac{\partial^{2}}{\partial \varphi \partial \theta}\left(\sqrt{\sin\varphi}\cos\theta\right)+4\sin\varphi\frac{\partial^{2}}{\partial \varphi^{2}}\left(\sqrt{\sin\varphi}\cos\theta\right)\\
&+4\cos\varphi\frac{\partial}{\partial \varphi}\left(\sqrt{\sin\varphi}\cos\theta\right)=3(\sin\varphi)^{3/2}\cos\theta+2\cos\varphi\sqrt{\sin\varphi}\sin\theta-\frac{\sqrt{\sin\varphi}\cos\theta}{\sin\varphi}\\
&+4\sin\varphi\frac{\partial}{\partial \varphi}\left(-\sqrt{\sin\varphi}\sin\theta\right)+4\sin\varphi\frac{\partial}{\partial \varphi}\left(\frac{\cos\varphi}{2\sqrt{\sin\varphi}}\cos\theta\right)+4\cos\varphi\frac{\cos\varphi}{2\sqrt{\sin\varphi}}\cos\theta\\
&=3\left(\sin\varphi\right)^{3/2}\cos\theta+2\sqrt{\sin\varphi}\cos\varphi\sin\theta-\frac{\sqrt{\sin\varphi}\cos\theta}{\sin\varphi}+4\sin\varphi\bigg(-\frac{\cos\varphi}{2\sqrt{\sin\varphi}}\sin\theta\bigg)\\
&+4\sin\varphi\frac{\left(-\sin\varphi\right)2\sqrt{\sin\varphi}-\cos\varphi\left(\displaystyle\frac{\cos\varphi}{\sqrt{\sin\varphi}}\right)}{4\sin\varphi}\cos\theta+2\hspace{0.05cm}\frac{\cos^{2}\varphi\cos\theta}{\sqrt{\sin\varphi}}\\
&=3\left(\sin\varphi\right)^{3/2}\cos\theta+2\sqrt{\sin\varphi}\cos\varphi\sin\theta-\frac{\cos\theta}{\sqrt{\sin\varphi}}-2\hspace{0.05cm}\frac{\sin\varphi\cos\varphi\sin\theta}{\sqrt{\sin\varphi}}\\
&-2(\sin\varphi)^{3/2}\cos\theta-\frac{\cos^{2}\varphi\cos\theta}{\sqrt{\sin\varphi}}+2\hspace{0.05cm}\frac{\cos^{2}\varphi\cos\theta}{\sqrt{\sin\varphi}}\\
&=\frac{\sin^{2}\varphi\cos\theta+2\sin\varphi\cos\varphi\sin\theta-\cos\theta-2\sin\varphi\cos\varphi\sin\theta+\cos^{2}\varphi\cos\theta}{\sqrt{\sin\varphi}}\\
&=\frac{\sin^{2}\varphi\cos\theta-\cos\theta+\cos^{2}\varphi\cos\theta}{\sqrt{\sin\varphi}}=\frac{\cos\theta-\cos\theta}{\sqrt{\sin\varphi}}=0,
\end{align*}
namely
\begin{equation}\label{rho-sqrt-sin-phi-cos-theta-harmonic-horiz}
\Delta_{\mathbb{H}^1}\left(\rho\sqrt{\sin\varphi}\cos\theta\right)\restrict{\partial B^{\mathbb{H}^1}_1(0)}=0.
\end{equation}
Hence, in view of \eqref{weak-equation-theta-phi-6}, with $\Omega_{\theta,\varphi}=(0,\pi)\times(0,\pi),$ we should have
\[
3=\frac{\displaystyle\int_{\Omega_{\theta,\varphi}}\bigg(\frac{1}{\sin\varphi}\bigg(\frac{\partial f}{\partial \theta}\bigg)^{2}+4\sin\varphi\frac{\partial f}{\partial \theta}\frac{\partial f}{\partial \varphi}+4\sin\varphi\bigg(\frac{\partial f}{\partial \varphi}\bigg)^{2}\bigg)\hspace{0.1cm}d\theta d\varphi}{\displaystyle\int_{\Omega_{\theta,\varphi}}(\sin\varphi)f^{2}\hspace{0.1cm}d\theta d\varphi},
\]
with $f=\sqrt{\sin\varphi}\cos\theta.$

Let us check now that this equality holds.\newline
Precisely, we have:
\begin{align*}
&\int_{(0,\pi)\times(0,\pi)}(\sin\varphi)(\sqrt{\sin\varphi}\cos\theta)^{2}\hspace{0.1cm}d\theta d\varphi=\int_{(0,\pi)\times(0,\pi)}\sin^{2}\varphi\cos^{2}\theta\hspace{0.1cm}d\theta d\varphi=\int_{0}^{\pi}\int_{0}^{\pi}\sin^{2}\varphi\cos^{2}\theta\hspace{0.1cm}d\theta d\varphi\\
&=\int_{0}^{\pi}\sin^{2}\varphi\hspace{0.1cm}d\varphi\int_{0}^{\pi}\cos^{2}\theta\hspace{0.1cm}d\theta,
\end{align*}
i.e.
\[\int_{(0,\pi)\times(0,\pi)}(\sin\varphi)(\sqrt{\sin\varphi}\cos\theta)^{2}\hspace{0.1cm}d\theta d\varphi=\int_{0}^{\pi}\sin^{2}\varphi\hspace{0.1cm}d\varphi\int_{0}^{\pi}\cos^{2}\theta\hspace{0.1cm}d\theta,\]
which implies, since, in general,
\begin{align}\label{formula-sin-^-2-cos-^-2}
\cos^{2}\tau=\frac{1+\cos(2\tau)}{2}\nonumber,\\
\sin^{2}\tau=\frac{1-\cos(2\tau)}{2},
\end{align}
\begin{align*}
&\int_{(0,\pi)\times(0,\pi)}(\sin\varphi)(\sqrt{\sin\varphi}\cos\theta)^{2}\hspace{0.1cm}d\theta d\varphi=\int_{0}^{\pi}\left(\frac{1-\cos(2\varphi)}{2}\right)\hspace{0.1cm}d\varphi\int_{0}^{\pi}\left(\frac{1+\cos(2\theta)}{2}\right)\hspace{0.1cm}d\theta\\
&=\left(\frac{\pi}{2}-\left[\frac{\sin(2\varphi)}{4}\right]^{\varphi=\pi}_{\varphi=0}\right)\left(\frac{\pi}{2}+\left[\frac{\sin(2\theta)}{4}\right]^{\theta=\pi}_{\theta=0}\right)=\frac{\pi}{2}\frac{\pi}{2}=\frac{\pi^{2}}{4},
\end{align*}
that is
\begin{equation}\label{denominator-weak-equation-theta-phi-6}
\int_{(0,\pi)\times(0,\pi)}(\sin\varphi)(\sqrt{\sin\varphi}\cos\theta)^{2}\hspace{0.1cm}d\theta d\varphi=\frac{\pi^{2}}{4}.
\end{equation}
In parallel, we also have
\begin{align*}
&\int_{(0,\pi)\times(0,\pi)}\bigg(\frac{1}{\sin\varphi}\bigg(\frac{\partial}{\partial \theta}(\sqrt{\sin\varphi}\cos\theta)\bigg)^{2}+4\sin\varphi\frac{\partial}{\partial \theta}(\sqrt{\sin\varphi}\cos\theta)\frac{\partial}{\partial \varphi}(\sqrt{\sin\varphi}\cos\theta)\hspace{0.1cm}d\theta d\varphi\nonumber\\
&+\int_{(0,\pi)\times(0,\pi)}4\sin\varphi\bigg(\frac{\partial}{\partial \varphi}(\sqrt{\sin\varphi}\cos\theta)\bigg)^{2}\bigg)\hspace{0.1cm}d\theta d\varphi=\int_{(0,\pi)\times(0,\pi)}\bigg(\frac{1}{\sin\varphi}\sin\varphi\sin^{2}\theta\\
&+4\sin\varphi(-\sqrt{\sin\varphi}\cos\theta)(\frac{\cos\varphi}{2\sqrt{\sin\varphi}}\cos\theta)+4\sin\varphi\bigg(\frac{\cos\varphi}{2\sqrt{\sin\varphi}}\cos\theta\bigg)^{2}\bigg)\hspace{0.1cm}d\theta d\varphi\\
&=\int_{(0,\pi)\times(0,\pi)}\bigg(\sin^{2}\theta-2\sin\varphi\cos\varphi\cos^{2}\theta+\cos^{2}\varphi\cos^{2}\theta\bigg)\hspace{0.1cm}d\theta d\varphi\\
&=\int_{(0,\pi)\times(0,\pi)}\bigg(\sin^{2}\theta-\sin(2\varphi)\cos^{2}\theta+\cos^{2}\varphi\cos^{2}\theta\bigg)\hspace{0.1cm}d\theta d\varphi
=\int_{(0,\pi)\times(0,\pi)}\sin^{2}\theta\hspace{0.1cm}d\theta d\varphi\\
&-\int_{(0,\pi)\times(0,\pi)}\sin(2\varphi)\cos^{2}\theta\hspace{0.1cm}d\theta d\varphi+\int_{(0,\pi)\times(0,\pi)}\cos^{2}\varphi\cos^{2}\theta\hspace{0.1cm}d\theta d\varphi,
\end{align*}
and thus, from \eqref{formula-sin-^-2-cos-^-2},
\begin{align*}
&\int_{(0,\pi)\times(0,\pi)}\bigg(\frac{1}{\sin\varphi}\bigg(\frac{\partial}{\partial \theta}(\sqrt{\sin\varphi}\cos\theta)\bigg)^{2}+4\sin\varphi\frac{\partial}{\partial \theta}(\sqrt{\sin\varphi}\cos\theta)\frac{\partial}{\partial \varphi}(\sqrt{\sin\varphi}\cos\theta)\hspace{0.1cm}d\theta d\varphi\nonumber\\
&+\int_{(0,\pi)\times(0,\pi)}4\sin\varphi\bigg(\frac{\partial}{\partial \varphi}(\sqrt{\sin\varphi}\cos\theta)\bigg)^{2}\bigg)\hspace{0.1cm}d\theta d\varphi=\int_{0}^{\pi}\int_{0}^{\pi}\sin^{2}\theta\hspace{0.1cm}d\theta d\varphi-\int_{0}^{\pi}\int_{0}^{\pi}\sin(2\varphi)\cos^{2}\theta\hspace{0.1cm}d\theta d\varphi\\
&+\int_{0}^{\pi}\int_{0}^{\pi}\cos^{2}\varphi\cos^{2}\theta\hspace{0.1cm}d\theta d\varphi=\pi\int_{0}^{\pi}\sin^{2}\theta\hspace{0.1cm}d\theta-\int_{0}^{\pi}\sin(2\varphi)\hspace{0.1cm}d\varphi\int_{0}^{\pi}\cos^{2}\theta\hspace{0.1cm}d\theta+\int_{0}^{\pi}\cos^{2}\varphi\hspace{0.1cm}d\varphi\int_{0}^{\pi}\cos^{2}\theta\hspace{0.1cm}d\theta,\\
&=\pi\int_{0}^{\pi}\bigg(\frac{1-\cos(2\theta)}{2}\bigg)\hspace{0.1cm}d\theta-\left[-\frac{\cos(2\varphi)}{2}\right]^{\varphi=\pi}_{\varphi=0}\int_{0}^{\pi}\bigg(\frac{1+\cos(2\theta)}{2}\bigg)\hspace{0.1cm}d\theta\\
&+\int_{0}^{\pi}\bigg(\frac{1+\cos(2\varphi)}{2}\bigg)\hspace{0.1cm}d\varphi\int_{0}^{\pi}\bigg(\frac{1+\cos(2\theta)}{2}\bigg)\hspace{0.1cm}d\theta=\pi\bigg(\frac{\pi}{2}-\left[\frac{\sin(2\theta)}{4}\right]^{\theta=\pi}_{\theta=0}\bigg)\\
&+\bigg(\frac{\pi}{2}+\left[\frac{\sin(2\varphi)}{4}\right]^{\varphi=\pi}_{\varphi=0}\bigg)\bigg(\frac{\pi}{2}+\left[\frac{\sin(2\theta)}{4}\right]^{\theta=\pi}_{\theta=0}\bigg)=\frac{\pi^{2}}{2}+\frac{\pi^{2}}{4}=\frac{3}{4}\pi^{2},
\end{align*}
namely
\begin{align}\label{numerator-weak-equation-theta-phi-6}
&\int_{(0,\pi)\times(0,\pi)}\bigg(\frac{1}{\sin\varphi}\bigg(\frac{\partial}{\partial \theta}(\sqrt{\sin\varphi}\cos\theta)\bigg)^{2}+4\sin\varphi\frac{\partial}{\partial \theta}(\sqrt{\sin\varphi}\cos\theta)\frac{\partial}{\partial \varphi}(\sqrt{\sin\varphi}\cos\theta)\hspace{0.1cm}d\theta d\varphi\nonumber\\
&+\int_{(0,\pi)\times(0,\pi)}4\sin\varphi\bigg(\frac{\partial}{\partial \varphi}(\sqrt{\sin\varphi}\cos\theta)\bigg)^{2}\bigg)\hspace{0.1cm}d\theta d\varphi=\frac{3}{4}\pi^{2}.
\end{align}
As a consequence, putting together \eqref{denominator-weak-equation-theta-phi-6} and \eqref{numerator-weak-equation-theta-phi-6}, we get
\begin{align*}
&\frac{\displaystyle \int_{(0,\pi)\times(0,\pi)}\bigg(\frac{1}{\sin\varphi}\bigg(\frac{\partial}{\partial \theta}(\sqrt{\sin\varphi}\cos\theta)\bigg)^{2}+4\sin\varphi\frac{\partial}{\partial \theta}(\sqrt{\sin\varphi}\cos\theta)\frac{\partial}{\partial \varphi}(\sqrt{\sin\varphi}\cos\theta)\hspace{0.1cm}d\theta d\varphi}{\displaystyle \int_{(0,\pi)\times(0,\pi)}(\sin\varphi)(\sqrt{\sin\varphi}\cos\theta)^{2}\hspace{0.1cm}d\theta d\varphi}\\
&+\frac{\displaystyle\int_{(0,\pi)\times(0,\pi)}4\sin\varphi\bigg(\frac{\partial}{\partial \varphi}(\sqrt{\sin\varphi}\cos\theta)\bigg)^{2}\bigg)\hspace{0.1cm}d\theta d\varphi}{\displaystyle\int_{(0,\pi)\times(0,\pi)}(\sin\varphi)(\sqrt{\sin\varphi}\cos\theta)^{2}\hspace{0.1cm}d\theta d\varphi}=3.
\end{align*}

\section{Straightforward computation of the two basic cases}

We want to state and prove two further lemmas.
\begin{lem}\label{lemma-integral-quotient-gradient-x-+}
	If $u=x^+,$ then
	\[\frac{\displaystyle \int_{\partial B^{\mathbb{H}^1}_1(0)}\frac{\left|\nabla_{\mathbb{H}^1}u\right|^2}{\sqrt{x^2+y^2}}\hspace{0.1cm}d\sigma_{\mathbb{H}^1}(\xi)}{\displaystyle \int_{B^{\mathbb{H}^1}_1(0)}\frac{\left|\nabla_{\mathbb{H}^1}u\right|^2}{\left|\xi\right|_{\mathbb{H}^1}^2}\hspace{0.1cm}d\xi}=2.\]
\end{lem}
\begin{proof}
	First of all, we note that
	\[\nabla_{\mathbb{H}^1}x^+\equiv\chi_{\left\{x>0\right\}}(1,0),\]
	hence
	\begin{equation}\label{integral-quotient-gradient-x-1}
	\frac{\displaystyle \int_{\partial B^{\mathbb{H}^1}_1(0)}\frac{\left|\nabla_{\mathbb{H}^1}u\right|^2}{\sqrt{x^2+y^2}}\hspace{0.1cm}d\sigma_{\mathbb{H}^1}(\xi)}{\displaystyle \int_{B^{\mathbb{H}^1}_1(0)}\frac{\left|\nabla_{\mathbb{H}^1}u\right|^2}{\left|\xi\right|_{\mathbb{H}^1}^2}\hspace{0.1cm}d\xi}=\frac{\displaystyle \int_{\partial B^{\mathbb{H}^1}_1(0)\cap\left\{x>0\right\}}\frac{1}{\sqrt{x^2+y^2}}\hspace{0.1cm}d\sigma_{\mathbb{H}^1}(\xi)}{\displaystyle \int_{B^{\mathbb{H}^1}_1(0)\cap\left\{x>0\right\}}\frac{1}{\left|\xi\right|_{\mathbb{H}^1}^2}\hspace{0.1cm}d\xi}.
	\end{equation}
	Let us calculate now
	\[\int_{B^{\mathbb{H}^1}_1(0)\cap\left\{x>0\right\}}\frac{1}{\left|\xi\right|_{\mathbb{H}^1}^2}\hspace{0.1cm}d\xi.\]
	To this end, we apply the change of variables in spherical coordinates, that is, denoting $\xi=(x,y,t),$
	\begin{equation}\label{polar-coordinates}
	T(\rho,\varphi,\theta)\coloneqq \begin{cases}
	x=\rho\sqrt{\sin\varphi}\cos\theta\\
	y=\rho\sqrt{\sin\varphi}\sin\theta\\
	t=\rho^2\cos\varphi,
	\end{cases}
	\end{equation}
	and we get, since $\left|\det J_T\right|=\rho^3$ and \begin{equation}\label{condition-x->-0}
	x=\rho\sqrt{\sin\varphi}\cos\theta>0\Longleftrightarrow-\frac{\pi}{2}<\theta<\frac{\pi}{2},
	\end{equation}
	\begin{equation}\label{integral-quotient-gradient-x-denominator-1}
    \int_{B^{\mathbb{H}^1}_1(0)\cap\left\{x>0\right\}}\frac{1}{\left|\xi\right|_{\mathbb{H}^1}^2}\hspace{0.1cm}d\xi=\int_{0}^{\pi}\int_{-\frac{\pi}{2}}^{\frac{\pi}{2}}\int_{0}^{1}\frac{\rho^3}{\rho^2}\hspace{0.1cm}d\rho\hspace{0.05cm}d\varphi\hspace{0.05cm}d\theta=\pi^2\left[\frac{\rho^2}{2}\right]^{\rho=1}_{\rho=0}=\frac{\pi^2}{2}.
	\end{equation}
	Regarding
	\[\int_{\partial B^{\mathbb{H}^1}_1(0)\cap\left\{x>0\right\}}\frac{1}{\sqrt{x^2+y^2}}\hspace{0.1cm}d\sigma_{\mathbb{H}^1}(\xi),\]
	instead, we recall first that, because $\rho=1$ on $\partial B^{\mathbb{H}^1}_1(0),$
	\begin{equation}\label{heisenberg-area-element}
	d\sigma_{\mathbb{H}^1}(\xi)=\frac{\left|\nabla_{\mathbb{H}^1}\rho\right|}{\left|\nabla\rho\right|}\hspace{0.1cm}d\sigma(\xi)=\frac{\sqrt{x^2+y^2}}{\left|\nabla\rho\right|}\hspace{0.1cm}d\sigma(\xi),
	\end{equation}
	hence we achieve
    \begin{align*}
    &\int_{\partial B^{\mathbb{H}^1}_1(0)\cap\left\{x>0\right\}}\frac{1}{\sqrt{x^2+y^2}}\hspace{0.1cm}d\sigma_{\mathbb{H}^1}(\xi)=\int_{\partial B^{\mathbb{H}^1}_1(0)\cap\left\{x>0\right\}}\frac{1}{\sqrt{x^2+y^2}}\frac{\sqrt{x^2+y^2}}{\left|\nabla \rho\right|}\hspace{0.1cm}d\sigma(\xi)\\
    &=\int_{\partial B^{\mathbb{H}^1}_1(0)\cap\left\{x>0\right\}}\frac{1}{\left|\nabla \rho\right|}\hspace{0.1cm}d\sigma(\xi),
    \end{align*}
    namely
    \begin{equation}\label{integral-quotient-gradient-x-numerator-1}
    \int_{\partial B^{\mathbb{H}^1}_1(0)\cap\left\{x>0\right\}}\frac{1}{\sqrt{x^2+y^2}}\hspace{0.1cm}d\sigma_{\mathbb{H}^1}(\xi)=\int_{\partial B^{\mathbb{H}^1}_1(0)\cap\left\{x>0\right\}}\frac{1}{\left|\nabla \rho\right|}\hspace{0.1cm}d\sigma(\xi).
    \end{equation}
	At this point, we consider the following parametrization of $\partial B^{\mathbb{H}^1}_1(0):$
	\begin{equation}\label{parametrization-of-the-boundary-of-Koranyi-ball}
	K(\theta,\varphi)\coloneqq 
	\begin{cases}
	x=\sqrt{\sin\varphi}\cos\theta\\
	y=\sqrt{\sin\varphi}\sin\theta\\
	t=\cos\varphi.
	\end{cases}
	\end{equation}
	Then, we obtain
	\[d\sigma(\xi)=\left|\frac{\partial K}{\partial \theta}\wedge\frac{\partial K}{\partial\varphi}\right|\hspace{0.1cm}d\theta\hspace{0.05cm}d\varphi,\]
	which yields
	\begin{equation}\label{integral-quotient-gradient-x-numerator-2}
    \int_{\partial B^{\mathbb{H}^1}_1(0)\cap\left\{x>0\right\}}\frac{1}{\left|\nabla \rho\right|}\hspace{0.1cm}d\sigma(\xi)=\int_{0}^{\pi}\int_{-\frac{\pi}{2}}^{\frac{\pi}{2}}\frac{1}{\left|\nabla\rho\right|}\left|\frac{\partial K}{\partial \theta}\wedge\frac{\partial K}{\partial\varphi}\right|\hspace{0.1cm}d\theta\hspace{0.05cm}d\varphi,
	\end{equation}
	using \eqref{condition-x->-0}.
	
	Specifically, we have
	\begin{align*}
	&\left|\nabla \rho\right|\restrict{\partial B^{\mathbb{H}^1}_1(0)}=\left|\nabla(((x^2+y^2)^2+t^2)^{1/4})\right|\\
	&=\left|\frac{1}{4}((x^2+y^2)^2+t^2)^{-3/4}(2(x^2+y^2)2x,2(x^2+y^2)2y,2t)\right|\\
	&=\frac{1}{2}\rho^{-3}\sqrt{4x^2(x^2+y^2)^2+4y^2(x^2+y^2)^2+t^2}=\frac{1}{2}\sqrt{4(x^2+y^2)^3+t^2}=\frac{1}{2}\sqrt{4\sin^3\varphi+\cos^2\varphi},
	\end{align*}
	that is
	\begin{equation}\label{gradient-rho-on-the-boundary-of-koranyi-ball}
	\left|\nabla\rho\right|\restrict{\partial B^{\mathbb{H}^1}_1(0)}=\frac{1}{2}\sqrt{4\sin^3\varphi+\cos^2\varphi}.
	\end{equation}
	On the other hand, in view of \eqref{parametrization-of-the-boundary-of-Koranyi-ball},
	\begin{align*}
	&\left|\frac{\partial K}{\partial \theta}\wedge\frac{\partial K}{\partial \varphi}\right|=\left|\det
	\begin{bmatrix}
	\vec{i} & \vec{j} &\vec{k}\\
	\\
	-\sqrt{\sin\varphi}\sin\theta & \sqrt{\sin\varphi}\cos\theta & 0\\
	\\
	\dfrac{\cos\varphi\cos\theta}{2\sqrt{\sin\varphi}}&\dfrac{\cos\varphi\sin\theta}{2\sqrt{\sin\varphi}} &-\sin\varphi
	\end{bmatrix}\right|\\
	&=\left|\left(-\sin^{3/2}(\varphi)\cos\theta,-\sin^{3/2}(\varphi)\sin\theta,-\frac{\cos\varphi}{2}\right)\right|=\sqrt{\sin^3\varphi+\frac{\cos^2\varphi}{4}}=\frac{1}{2}\sqrt{4\sin^3\varphi+\cos^2\varphi},
	\end{align*}
	which implies, according to \eqref{gradient-rho-on-the-boundary-of-koranyi-ball},
	\begin{equation}\label{quotient-vector-product-of-area-element-gradient-rho-on-the-boundary-of-koranyi-ball}
	\frac{\left|\frac{\partial K}{\partial \theta}\wedge\frac{\partial K}{\partial\varphi}\right|}{\left|\nabla\rho\right|}=1\quad\mbox{on }\partial B^{\mathbb{H}^1}_1(0).
	\end{equation}
	Consequently, by virtue of \eqref{quotient-vector-product-of-area-element-gradient-rho-on-the-boundary-of-koranyi-ball}, we have
	\begin{align*}
	\int_0^\pi\int_{-\frac{\pi}{2}}^{\frac{\pi}{2}}\frac{1}{\left|\nabla\rho\right|}\left|\frac{\partial K}{\partial \theta}\wedge\frac{\partial K}{\partial\varphi}\right|\hspace{0.1cm}d\theta\hspace{0.05cm}d\varphi=\int_0^\pi\int_{-\frac{\pi}{2}}^{\frac{\pi}{2}}\hspace{0.1cm}d\theta\hspace{0.05cm}d\varphi=\pi^2,
	\end{align*}
	hence, from \eqref{integral-quotient-gradient-x-numerator-1} and \eqref{integral-quotient-gradient-x-numerator-2}, we get
	\begin{equation}\label{integral-quotient-gradient-x-numerator-3}
	\int_{\partial B^{\mathbb{H}^1}_1(0)\cap\left\{x>0\right\}}\frac{1}{\sqrt{x^2+y^2}}\hspace{0.1cm}d\sigma_{\mathbb{H}^1}(\xi)=\pi^2.
	\end{equation}
	Putting together \eqref{integral-quotient-gradient-x-denominator-1} and \eqref{integral-quotient-gradient-x-numerator-3}, we finally achieve, by \eqref{integral-quotient-gradient-x-1},
	\[\frac{\displaystyle \int_{\partial B^{\mathbb{H}^1}_1(0)}\frac{\left|\nabla_{\mathbb{H}^1}u\right|^2}{\sqrt{x^2+y^2}}\hspace{0.1cm}d\sigma_{\mathbb{H}^1}(\xi)}{\displaystyle \int_{B^{\mathbb{H}^1}_1(0)}\frac{\left|\nabla_{\mathbb{H}^1}u\right|^2}{\left|\xi\right|_{\mathbb{H}^1}^2}\hspace{0.1cm}d\xi}=\frac{\pi^2}{\displaystyle\frac{\pi^2}{2}}=2,\]
	i.e.
	\[\frac{\displaystyle \int_{\partial B^{\mathbb{H}^1}_1(0)}\frac{\left|\nabla_{\mathbb{H}^1}u\right|^2}{\sqrt{x^2+y^2}}\hspace{0.1cm}d\sigma_{\mathbb{H}^1}(\xi)}{\displaystyle \int_{B^{\mathbb{H}^1}_1(0)}\frac{\left|\nabla_{\mathbb{H}^1}u\right|^2}{\left|\xi\right|_{\mathbb{H}^1}^2}\hspace{0.1cm}d\xi}=2.\]
\end{proof}
\begin{lem}\label{lemma-integral-quotient-gradient-t-+}
	If $u=t^+,$ then
	\[\frac{\displaystyle \int_{\partial B^{\mathbb{H}^1}_1(0)}\frac{\left|\nabla_{\mathbb{H}^1}u\right|^2}{\sqrt{x^2+y^2}}\hspace{0.1cm}d\sigma_{\mathbb{H}^1}(\xi)}{\displaystyle \int_{B^{\mathbb{H}^1}_1(0)}\frac{\left|\nabla_{\mathbb{H}^1}u\right|^2}{\left|\xi\right|_{\mathbb{H}^1}^2}\hspace{0.1cm}d\xi}=4.\]
\end{lem}
\begin{proof}
	We point out first that in this case we have
	\[\nabla_{\mathbb{H}^1}t^+\equiv 2(y,-x)\chi_{\left\{t>0\right\}}.\]
	Therefore
	\begin{align*}
	&\frac{\displaystyle \int_{\partial B^{\mathbb{H}^1}_1(0)}\frac{\left|\nabla_{\mathbb{H}^1}u\right|^2}{\sqrt{x^2+y^2}}\hspace{0.1cm}d\sigma_{\mathbb{H}^1}(\xi)}{\displaystyle \int_{B^{\mathbb{H}^1}_1(0)}\frac{\left|\nabla_{\mathbb{H}^1}u\right|^2}{\left|\xi\right|_{\mathbb{H}^1}^2}\hspace{0.1cm}d\xi}=\frac{\displaystyle \int_{\partial B^{\mathbb{H}^1}_1(0)\cap \left\{t>0\right\}}\frac{4(x^2+y^2)}{\sqrt{x^2+y^2}}\hspace{0.1cm}d\sigma_{\mathbb{H}^1}(\xi)}{\displaystyle \int_{B^{\mathbb{H}^1}_1(0)\cap \left\{t>0\right\}}\frac{4(x^2+y^2)}{\left|\xi\right|_{\mathbb{H}^1}^2}\hspace{0.1cm}d\xi}\\
	&=\frac{\displaystyle \int_{\partial B^{\mathbb{H}^1}_1(0)\cap \left\{t>0\right\}}4\sqrt{x^2+y^2}\hspace{0.1cm}d\sigma_{\mathbb{H}^1}(\xi)}{\displaystyle \int_{B^{\mathbb{H}^1}_1(0)\cap \left\{t>0\right\}}\frac{4(x^2+y^2)}{\left|\xi\right|_{\mathbb{H}^1}^2}\hspace{0.1cm}d\xi},
	\end{align*}
	that is 
	\begin{equation}\label{integral-quotient-gradient-t-+-1}
	\frac{\displaystyle \int_{\partial B^{\mathbb{H}^1}_1(0)}\frac{\left|\nabla_{\mathbb{H}^1}u\right|^2}{\sqrt{x^2+y^2}}\hspace{0.1cm}d\sigma_{\mathbb{H}^1}(\xi)}{\displaystyle \int_{B^{\mathbb{H}^1}_1(0)}\frac{\left|\nabla_{\mathbb{H}^1}u\right|^2}{\left|\xi\right|_{\mathbb{H}^1}^2}\hspace{0.1cm}d\xi}=\frac{\displaystyle \int_{\partial B^{\mathbb{H}^1}_1(0)\cap \left\{t>0\right\}}4\sqrt{x^2+y^2}\hspace{0.1cm}d\sigma_{\mathbb{H}^1}(\xi)}{\displaystyle \int_{B^{\mathbb{H}^1}_1(0)\cap \left\{t>0\right\}}\frac{4(x^2+y^2)}{\left|\xi\right|_{\mathbb{H}^1}^2}\hspace{0.1cm}d\xi}.
	\end{equation}
	Now, let us compute
	\[\int_{B^{\mathbb{H}^1}_1(0)\cap \left\{t>0\right\}}\frac{4(x^2+y^2)}{\left|\xi\right|^2_{\mathbb{H}^1}}\hspace{0.1cm}d\xi.\]
	Using the change of variables in spherical coordinates \eqref{polar-coordinates} and noting that
	\begin{equation}\label{condition-t->-0}
	t=\rho^2\cos\varphi>0\Longleftrightarrow 0<\varphi<\frac{\pi}{2},
	\end{equation}
	we then achieve
	\begin{align*}
	&\int_{B^{\mathbb{H}^1}_1(0)\cap \left\{t>0\right\}}\frac{4(x^2+y^2)}{\left|\xi\right|^2_{\mathbb{H}^1}}\hspace{0.1cm}d\xi=\int_{-\pi}^{\pi}\int_0^{\frac{\pi}{2}}\int_0^1\frac{4\rho^2\sin\varphi}{\rho^2}\rho^3\hspace{0.1cm}d\rho d\varphi d\theta\\
	&=8\pi\int_0^{\frac{\pi}{2}}\sin\varphi\hspace{0.1cm}d\varphi\int_0^1\rho^3\hspace{0.1cm}d\rho=8\pi\bigg[-\cos\varphi\bigg]^{\varphi=\frac{\pi}{2}}_{\varphi=0}\left[\frac{\rho^4}{4}\right]^{\rho=1}_{\rho=0}=\frac{8\pi}{4}=2\pi,
	\end{align*}
	which gives
	\begin{equation}\label{integral-quotient-gradient-t-+-denominator-1}
	\int_{B^{\mathbb{H}^1}_1(0)\cap \left\{t>0\right\}}\frac{4(x^2+y^2)}{\left|\xi\right|^2_{\mathbb{H}^1}}\hspace{0.1cm}d\xi=2\pi.
	\end{equation}
	In parallel, in view of \eqref{heisenberg-area-element}, \eqref{quotient-vector-product-of-area-element-gradient-rho-on-the-boundary-of-koranyi-ball} and \eqref{condition-t->-0}, we also get
	\begin{align*}
	&\int_{\partial B^{\mathbb{H}^1}_1(0)\cap\left\{t>0\right\}}4\sqrt{x^2+y^2}\hspace{0.1cm}d\sigma_{\mathbb{H}^1}(\xi)=\int_{\partial B^{\mathbb{H}^1}_1(0)\cap\left\{t>0\right\}}4(x^2+y^2)\frac{1}{\left|\nabla\rho\right|}\hspace{0.1cm}d\sigma(\xi)\\
	&4\int_{-\pi}^{\pi}\int_0^{\frac{\pi}{2}}\sin\varphi\hspace{0.1cm}d\varphi\hspace{0.05cm}d\theta=8\pi\bigg[-\cos\varphi\bigg]^{\varphi=\frac{\pi}{2}}_{\varphi=0}=8\pi,
	\end{align*}
	i.e.
	\begin{equation}\label{integral-quotient-gradient-t-+-numerator-1}
	\int_{\partial B^{\mathbb{H}^1}_1(0)\cap\left\{t>0\right\}}4\sqrt{x^2+y^2}\hspace{0.1cm}d\sigma_{\mathbb{H}^1}(\xi)=8\pi.
	\end{equation}
	As a consequence, from \eqref{integral-quotient-gradient-t-+-1}, \eqref{integral-quotient-gradient-t-+-denominator-1} and \eqref{integral-quotient-gradient-t-+-numerator-1}, we finally have
	\[\frac{\displaystyle \int_{\partial B^{\mathbb{H}^1}_1(0)}\frac{\left|\nabla_{\mathbb{H}^1}u\right|^2}{\sqrt{x^2+y^2}}\hspace{0.1cm}d\sigma_{\mathbb{H}^1}(\xi)}{\displaystyle \int_{B^{\mathbb{H}^1}_1(0)}\frac{\left|\nabla_{\mathbb{H}^1}u\right|^2}{\left|\xi\right|_{\mathbb{H}^1}^2}\hspace{0.1cm}d\xi}=\frac{8\pi}{2\pi}=4,\]
	which yields
	\[\frac{\displaystyle \int_{\partial B^{\mathbb{H}^1}_1(0)}\frac{\left|\nabla_{\mathbb{H}^1}u\right|^2}{\sqrt{x^2+y^2}}\hspace{0.1cm}d\sigma_{\mathbb{H}^1}(\xi)}{\displaystyle \int_{B^{\mathbb{H}^1}_1(0)}\frac{\left|\nabla_{\mathbb{H}^1}u\right|^2}{\left|\xi\right|_{\mathbb{H}^1}^2}\hspace{0.1cm}d\xi}=4.\]
\end{proof}

\section{Evaluation of the first eigenvalue for symmetric caps}
We observe, at this point, that a symmetric cap with respect to the $t-$axis may be described by only using the variable $\varphi$ in the  change of variables $T$ see \eqref{polar-coordinates}. Therefore, the following results hold.   
\begin{lem}\label{lemma-eigenvalues-depending-on-theta-phi-and-phi}
If $u=x^+,$ then
\[\frac{\displaystyle \int_{\partial B^{\mathbb{H}^1}_1(0)\cap\left\{u>0\right\}}\frac{\left|\nabla_{\mathbb{H}^1}^{\varphi}u\right|^2}{\sqrt{x^2+y^2}}\hspace{0.1cm}d\sigma_{\mathbb{H}^1}(\xi)}{\displaystyle\int_{\partial B^{\mathbb{H}^1}_1(0)\cap\left\{u>0\right\}}u^2\sqrt{x^2+y^2}\hspace{0.1cm}d\sigma_{\mathbb{H}^1}(\xi)}=2.\]
If $u=t^+,$ then
\[\frac{\displaystyle \int_{\partial B^{\mathbb{H}^1}_1(0)\cap\left\{u>0\right\}}\frac{\left|\nabla_{\mathbb{H}^1}^{\varphi}u\right|^2}{\sqrt{x^2+y^2}}\hspace{0.1cm}d\sigma_{\mathbb{H}^1}(\xi)}{\displaystyle\int_{\partial B^{\mathbb{H}^1}_1(0)\cap\left\{u>0\right\}}u^2\sqrt{x^2+y^2}\hspace{0.1cm}d\sigma_{\mathbb{H}^1}(\xi)}=8.\]
As a consequence it results in both cases
$$
\lambda(\partial B^{\mathbb{H}^1}_1(0)\cap\{u>0\})\leq 2.
$$
\end{lem}
\begin{proof}
We start from $u=x^+.$ In particular, we want to compute
	\[\frac{\displaystyle \int_{\partial B^{\mathbb{H}^1}_1(0)\cap\left\{u>0\right\}}\frac{\left|\nabla_{\mathbb{H}^1}^{\varphi}u\right|^2}{\sqrt{x^2+y^2}}\hspace{0.1cm}d\sigma_{\mathbb{H}^1}(\xi)}{\displaystyle\int_{\partial B^{\mathbb{H}^1}_1(0)\cap\left\{u>0\right\}}u^2\sqrt{x^2+y^2}\hspace{0.1cm}d\sigma_{\mathbb{H}^1}(\xi)},\]
	with $u=x^+.$ Using the parametrization in spherical coordinates of the boundary of unitary Koranyi ball in \eqref{parametrization-of-the-boundary-of-Koranyi-ball},
	$u=x^+$ reads 
	\begin{equation}\label{x-in-spherical-coordinates-1}
	u=(\sqrt{\sin\varphi}\cos\theta)^+.
	\end{equation}
	As a consequence, from \eqref{x-in-spherical-coordinates-1}, we get
	\begin{equation}\label{u->-0-spherical-coordinates-1}
	u>0\Longleftrightarrow -\frac{\pi}{2}<\theta<\frac{\pi}{2}.
	\end{equation}
	At this point, we want to express $\left|\nabla_{\mathbb{H}^1}^{\varphi}u\right|^2$ according to \eqref{x-in-spherical-coordinates-1}. Let us recall first that if $v=\rho^{\alpha}f(\theta,\varphi),$ we have 
	\begin{equation}\label{heisenberg-gradient-rho-^-alpha-f-theta-varphi-1}
    \nabla_{\mathbb{H}^1}v=\alpha\rho^{\alpha-1}\nabla_{\mathbb{H}^1}\rho+\rho^{\alpha}\left(\frac{\partial f}{\partial \theta}\nabla_{\mathbb{H}^1}\theta+\frac{\partial f}{\partial \varphi}\nabla_{\mathbb{H}^1}\varphi\right).
	\end{equation}
	Moreover, we know by \eqref{two-components-nabla-horiz} that
	\[\nabla_{\mathbb{H}^1}^{\varphi}v=\langle\nabla_{\mathbb{H}^1}v,e_{\varphi}\rangle e_\varphi,\]
	where $e_{\varphi}=\frac{\nabla_{\mathbb{H}^1}\varphi}{\left|\nabla_{\mathbb{H}^1}\varphi\right|},$ thus, in view of \eqref{heisenberg-gradient-rho-^-alpha-f-theta-varphi-1}, we achieve
	\begin{align*}
	&\nabla_{\mathbb{H}^1}^{\varphi}v=\langle\nabla_{\mathbb{H}^1}v,e_{\varphi}\rangle e_\varphi=\langle\alpha\rho^{\alpha-1}\nabla_{\mathbb{H}^1}\rho+\rho^{\alpha}\left(\frac{\partial f}{\partial \theta}\nabla_{\mathbb{H}^1}\theta+\frac{\partial f}{\partial \varphi}\nabla_{\mathbb{H}^1}\varphi\right),\frac{\nabla_{\mathbb{H}^1}\varphi}{\left|\nabla_{\mathbb{H}^1}\varphi\right|}\rangle e_\varphi\\
	&=\frac{\rho^{\alpha}}{\left|\nabla_{\mathbb{H}^1}\varphi\right|}\left(\frac{\partial f}{\partial \theta}\langle\nabla_{\mathbb{H}^1}\theta,\nabla_{\mathbb{H}^1}\varphi\rangle+\frac{\partial f}{\partial\varphi}\left|\nabla_{\mathbb{H}^1}\varphi\right|^2\right)e_\varphi=\frac{\rho^{\alpha+2}}{2\sqrt{x^2+y^2}}\bigg(\frac{\partial f}{\partial\theta}\frac{2(x^2+y^2)}{\rho^4}\\
	&+\frac{\partial f}{\partial \varphi}\frac{4(x^2+y^2)}{\rho^4}\bigg)e_\varphi=\rho^{\alpha-2}\sqrt{x^2+y^2}\left(\frac{\partial f}{\partial\theta}+2\frac{\partial f}{\partial\varphi}\right)e_\varphi,
	\end{align*}
	namely
	\[\nabla_{\mathbb{H}^1}^{\varphi}v=\rho^{\alpha-2}\sqrt{x^2+y^2}\left(\frac{\partial f}{\partial\theta}+2\frac{\partial f}{\partial\varphi}\right)e_\varphi,\]
	which implies
	\begin{equation}\label{phi-component-of-the-heisenberg-gradient-of-v-square-norm-1}
	\left|\nabla_{\mathbb{H}^1}^{\varphi}v\right|^2=\rho^{2(\alpha-2)}(x^2+y^2)\left(\frac{\partial f}{\partial\theta}+2\frac{\partial f}{\partial\varphi}\right)^2.
	\end{equation}
	Let us note that, in the previous computation, we have used the facts that, from Lemma \ref{lemma-inner-product-nabla-phi-rho-theta} and Lemma \ref{norm-quad-nabla-horiz-theta-phi-rho},
	\begin{equation*}
	\nabla_{\mathbb{H}^1}\varphi \cdot \nabla_{\mathbb{H}^1}\rho=0,\quad
	\nabla_{\mathbb{H}^1}\varphi \cdot \nabla_{\mathbb{H}^1}\theta=\frac{2(x^{2}+y^{2})}{\rho^{4}},\quad
	\left|\nabla_{\mathbb{H}^1}\varphi\right|^{2}=\frac{4(x^{2}+y^{2})}{\rho^{4}}.
	\end{equation*}
	Now, in particular, on $\partial B^{\mathbb{H}^1}_1(0)$ \eqref{phi-component-of-the-heisenberg-gradient-of-v-square-norm-1} reads, by virtue of \eqref{parametrization-of-the-boundary-of-Koranyi-ball},
	\begin{equation}\label{phi-component-of-the-heisenberg-gradient-of-v-square-norm-on-the-boundary-of-unitary-koranyi-ball-1}
	\left|\nabla_{\mathbb{H}^1}^{\varphi}v\right|^2\restrict{\partial B^{\mathbb{H}^1}_1(0)}=\sin\varphi\left(\frac{\partial f}{\partial\theta}+2\frac{\partial f}{\partial\varphi}\right)^2.
	\end{equation}
	In addition, we can choose $f=\sqrt{\sin\varphi}\cos\theta,$ so that with $u=\sqrt{\sin\varphi}\cos\theta,$ we obtain
	\begin{align*}
	&\left|\nabla_{\mathbb{H}^1}^{\varphi}u\right|^2\restrict{\partial B^{\mathbb{H}^1}_1(0)}=\sin\varphi\bigg(-\sqrt{\sin\varphi}\sin\theta+2\frac{\cos\varphi\cos\theta}{2\sqrt{\sin\varphi}}\bigg)^2=\sin\varphi\bigg(\sin\varphi\sin^2\theta+\frac{\cos^2\varphi\cos^2\theta}{\sin\varphi}\\
	&-2\cos\varphi\cos\theta\sin\theta\bigg)=\sin^2\varphi\sin^2\theta+\cos^2\varphi\cos^2\theta-2\cos\varphi\sin\varphi\cos\theta\sin\theta\\
	&=(\sin\varphi\sin\theta-\cos\varphi\cos\theta)^2=\cos^2(\theta+\varphi),
	\end{align*}
	that is 
	\begin{equation}\label{phi-component-of-the-heisenberg-gradient-of-u-square-norm-on-the-boundary-of-unitary-koranyi-ball-1}
	\left|\nabla_{\mathbb{H}^1}^{\varphi}u\right|^2\restrict{\partial B^{\mathbb{H}^1}_1(0)}=\cos^2(\theta+\varphi).
	\end{equation}
	Let us recall, at this point, that
	\begin{equation}\label{surface-element-heisenberg-1}
	d\sigma_{\mathbb{H}^1}(\xi)=\frac{\left|\nabla_{\mathbb{H}^1}\rho\right|}{\left|\nabla\rho\right|}\hspace{0.1cm}d\sigma(\xi)=\sqrt{\sin\varphi}\hspace{0.1cm}d\theta\hspace{0.05cm}d\varphi,
	\end{equation}
	therefore, using \eqref{parametrization-of-the-boundary-of-Koranyi-ball}, \eqref{x-in-spherical-coordinates-1}, \eqref{u->-0-spherical-coordinates-1} and \eqref{phi-component-of-the-heisenberg-gradient-of-u-square-norm-on-the-boundary-of-unitary-koranyi-ball-1}, we get
	\begin{align*}
	&\frac{\displaystyle \int_{\partial B^{\mathbb{H}^1}_1(0)\cap\left\{u>0\right\}}\frac{\left|\nabla_{\mathbb{H}^1}^{\varphi}u\right|^2}{\sqrt{x^2+y^2}}\hspace{0.1cm}d\sigma_{\mathbb{H}^1}(\xi)}{\displaystyle\int_{\partial B^{\mathbb{H}^1}_1(0)\cap\left\{u>0\right\}}u^2\sqrt{x^2+y^2}\hspace{0.1cm}d\sigma_{\mathbb{H}^1}(\xi)}=\frac{\displaystyle\int_0^\pi\int_{-\frac{\pi}{2}}^{\frac{\pi}{2}}\frac{\cos^2(\theta+\varphi)}{\sqrt{\sin\varphi}}\sqrt{\sin\varphi}\hspace{0.1cm}d\theta\hspace{0.05cm}d\varphi}{\displaystyle\int_0^\pi\int_{-\frac{\pi}{2}}^{\frac{\pi}{2}}\sin\varphi\cos^2\theta\sqrt{\sin\varphi}\sqrt{\sin\varphi}\hspace{0.1cm}d\theta\hspace{0.05cm}d\varphi}\\
	&=\frac{\displaystyle\int_0^\pi\int_{-\frac{\pi}{2}}^{\frac{\pi}{2}}\cos^2(\theta+\varphi)\hspace{0.1cm}d\theta\hspace{0.05cm}d\varphi}{\displaystyle\int_0^\pi\sin^2\varphi\hspace{0.1cm}d\varphi\int_{-\frac{\pi}{2}}^{\frac{\pi}{2}}\cos^2\theta\hspace{0.1cm}d\theta},
	\end{align*}
	i.e.
	\begin{equation}\label{quotient-lambda-phi-x-spherical-coordinates-1-1}
	\frac{\displaystyle \int_{\partial B^{\mathbb{H}^1}_1(0)\cap\left\{u>0\right\}}\frac{\left|\nabla_{\mathbb{H}^1}^{\varphi}u\right|^2}{\sqrt{x^2+y^2}}\hspace{0.1cm}d\sigma_{\mathbb{H}^1}(\xi)}{\displaystyle\int_{\partial B^{\mathbb{H}^1}_1(0)\cap\left\{u>0\right\}}u^2\sqrt{x^2+y^2}\hspace{0.1cm}d\sigma_{\mathbb{H}^1}(\xi)}=\frac{\displaystyle\int_0^\pi\int_{-\frac{\pi}{2}}^{\frac{\pi}{2}}\cos^2(\theta+\varphi)\hspace{0.1cm}d\theta\hspace{0.05cm}d\varphi}{\displaystyle\int_0^\pi\sin^2\varphi\hspace{0.1cm}d\varphi\int_{-\frac{\pi}{2}}^{\frac{\pi}{2}}\cos^2\theta\hspace{0.1cm}d\theta}.
	\end{equation}
	Let us compute now the numerator and the denominator of the right term in \eqref{quotient-lambda-phi-x-spherical-coordinates-1-1} separately. In both cases, we use the duplication formulas recalled in \eqref{formula-sin-^-2-cos-^-2}.
	
	As regards the numerator, in particular, we have
	\begin{align*}
	&\int_0^\pi\int_{-\frac{\pi}{2}}^{\frac{\pi}{2}}\cos^2(\theta+\varphi)\hspace{0.1cm}d\theta\hspace{0.05cm}d\varphi=\int_0^\pi\int_{-\frac{\pi}{2}}^{\frac{\pi}{2}}\left(\frac{1+\cos(2(\theta+\varphi))}{2}\right)\hspace{0.1cm}d\theta\hspace{0.05cm}d\varphi\\
	&=\frac{1}{2}\int_0^\pi\left[\theta+\frac{\sin(2(\theta+\varphi))}{2}\right]^{\theta=\frac{\pi}{2}}_{\theta=-\frac{\pi}{2}}\hspace{0.1cm}d\varphi=\frac{1}{2}\int_0^\pi\pi\hspace{0.1cm}d\varphi=\frac{\pi^2}{2},
	\end{align*}
	which gives
	\begin{equation}\label{quotient-lambda-phi-x-spherical-coordinates-numerator-1}
	\int_0^\pi\int_{-\frac{\pi}{2}}^{\frac{\pi}{2}}\cos^2(\theta+\varphi)\hspace{0.1cm}d\theta\hspace{0.05cm}d\varphi=\frac{\pi^2}{2}.
	\end{equation}
	Concerning the denominator, instead, it holds
	\begin{align*}
	&\int_0^\pi\sin^2\varphi\hspace{0.1cm}d\varphi\int_{-\frac{\pi}{2}}^{\frac{\pi}{2}}\cos^2\theta\hspace{0.1cm}d\theta=\int_0^\pi\left(\frac{1-\cos(2\varphi)}{2}\right)\hspace{0.1cm}d\varphi\int_{-\frac{\pi}{2}}^{\frac{\pi}{2}}\left(\frac{1+\cos(2\theta)}{2}\right)\hspace{0.1cm}d\theta\\
	&=\frac{1}{4}\left[\varphi-\frac{\sin(2\varphi)}{2}\right]^{\varphi=\pi}_{\varphi=0}\left[\theta+\frac{\sin(2\theta)}{2}\right]^{\theta=\frac{\pi}{2}}_{\theta=-\frac{\pi}{2}}=\frac{\pi^2}{4},
	\end{align*}
	namely
	\begin{equation}\label{quotient-lambda-phi-x-spherical-coordinates-denominator-1}
	\int_0^\pi\sin^2\varphi\hspace{0.1cm}d\varphi\int_{-\frac{\pi}{2}}^{\frac{\pi}{2}}\cos^2\theta\hspace{0.1cm}d\theta=\frac{\pi^2}{4}.
	\end{equation}
	Consequently, from \eqref{quotient-lambda-phi-x-spherical-coordinates-numerator-1}, \eqref{quotient-lambda-phi-x-spherical-coordinates-denominator-1} and \eqref{quotient-lambda-phi-x-spherical-coordinates-1-1}, we finally obtain
	\[\frac{\displaystyle \int_{\partial B^{\mathbb{H}^1}_1(0)\cap\left\{u>0\right\}}\frac{\left|\nabla_{\mathbb{H}^1}^{\varphi}u\right|^2}{\sqrt{x^2+y^2}}\hspace{0.1cm}d\sigma_{\mathbb{H}^1}(\xi)}{\displaystyle\int_{\partial B^{\mathbb{H}^1}_1(0)\cap\left\{u>0\right\}}u^2\sqrt{x^2+y^2}\hspace{0.1cm}d\sigma_{\mathbb{H}^1}(\xi)}=\frac{\displaystyle\frac{\pi^2}{2}}{\displaystyle\frac{\pi^2}{4}}=2,\]
	which yields
	\begin{equation}\label{quotient-lambda-phi-x-spherical-coordinates-final}
	\frac{\displaystyle \int_{\partial B^{\mathbb{H}^1}_1(0)\cap\left\{u>0\right\}}\frac{\left|\nabla_{\mathbb{H}^1}^{\varphi}u\right|^2}{\sqrt{x^2+y^2}}\hspace{0.1cm}d\sigma_{\mathbb{H}^1}(\xi)}{\displaystyle\int_{\partial B^{\mathbb{H}^1}_1(0)\cap\left\{u>0\right\}}u^2\sqrt{x^2+y^2}\hspace{0.1cm}d\sigma_{\mathbb{H}^1}(\xi)}=2.
	\end{equation}	
	In case $u=t^+,$ instead, in view of \eqref{parametrization-of-the-boundary-of-Koranyi-ball}, we achieve $u^+=(\cos\varphi)^+,$ which entails
	\begin{equation}\label{u->-0-spherical-coordinates-2}
	u>0\Longleftrightarrow 0<\varphi<\frac{\pi}{2}.
	\end{equation}
	Then, keeping in mind \eqref{phi-component-of-the-heisenberg-gradient-of-v-square-norm-on-the-boundary-of-unitary-koranyi-ball-1}, \eqref{surface-element-heisenberg-1} and \eqref{u->-0-spherical-coordinates-2}, we get
	\begin{equation*}\begin{split}&\frac{\displaystyle \int_{\partial B^{\mathbb{H}^1}_1(0)\cap\left\{u>0\right\}}\frac{\left|\nabla_{\mathbb{H}^1}^{\varphi}u\right|^2}{\sqrt{x^2+y^2}}\hspace{0.1cm}d\sigma_{\mathbb{H}^1}(\xi)}{\displaystyle\int_{\partial B^{\mathbb{H}^1}_1(0)\cap\left\{u>0\right\}}u^2\sqrt{x^2+y^2}\hspace{0.1cm}d\sigma_{\mathbb{H}^1}(\xi)}=
	\frac{\displaystyle \int_0^{\frac{\pi}{2}}\left(\int_0^{2\pi}4\sin\varphi\left(\frac{\partial f}{\partial\varphi}\right)^2\hspace{0.1cm}d\theta\right) d\varphi }{\displaystyle\int_0^{\frac{\pi}{2}}\left(\int_{0}^{2\pi}\cos^2\varphi \sqrt{\sin\varphi}\sqrt{\sin\varphi}\hspace{0.1cm}d\theta\right)\hspace{0.05cm}d\varphi}\\
	&=\frac{8\pi\displaystyle \int_0^{\frac{\pi}{2}}\sin\varphi\sin^2\varphi d\varphi }{2\pi\displaystyle\int_0^{\frac{\pi}{2}}\cos^2\varphi \sin\varphi d\varphi}=\frac{8\pi\displaystyle \int_0^{\frac{\pi}{2}}\sin\varphi(1-\cos^2\varphi) d\varphi }{2\pi\left[-\frac{1}{3}\cos^3\varphi\right]_{\varphi=0}^{\varphi=\frac{\pi}{2}}}=\frac{8\pi\left[-\cos\varphi+\frac{1}{3}\cos^3\varphi\right]_{\varphi=0}^{\varphi=\frac{\pi}{2}}}{\frac{2}{3}\pi}=8,
	\end{split}
	\end{equation*}	
	that is
	\[\frac{\displaystyle \int_{\partial B^{\mathbb{H}^1}_1(0)\cap\left\{u>0\right\}}\frac{\left|\nabla_{\mathbb{H}^1}^{\varphi}u\right|^2}{\sqrt{x^2+y^2}}\hspace{0.1cm}d\sigma_{\mathbb{H}^1}(\xi)}{\displaystyle\int_{\partial B^{\mathbb{H}^1}_1(0)\cap\left\{u>0\right\}}u^2\sqrt{x^2+y^2}\hspace{0.1cm}d\sigma_{\mathbb{H}^1}(\xi)}=8.\]
	
\end{proof}
\begin{cor} \label{corolmain}
If $\beta=8$ and $u=a t^+-b t^-,$ for $a^2+b^2\not=0,$ then $J_{8,\mathbb{H}^1}'(r)=0.$ 
Moreover, for every $\beta> 8,$ $J_{\beta,\mathbb{H}^1}$ is not monotone so that, if a monotonicity formula exists, then $\beta\leq 8.$ 

If $\beta=4$  and $u=(ax+by)^+-(ax+by)^-,$ for $a^2+b^2\not=0,$   then $J_{4,\mathbb{H}^1}'(r)=0.$ Furthermore, for every $\beta> 4,$ $J_{\beta,\mathbb{H}^1}$ is not monotone so that, if a monotonicity formula exists, then $\beta\leq 4.$ 

If the minimum $\lambda_\varphi$ were realized by a function like $u=(ax+by)^+-(ax+by)^-$, or the minimum of the function $h,$ see \eqref{half}, is greater or equal than $16,$ then, in order to obtain a monotonicity formula, we have to require that $\beta=4.$ 

\end{cor}
\begin{proof}
The first two statements immediately follow by Lemma \ref{lemma-integral-quotient-gradient-x-+}, Lemma \ref{lemma-integral-quotient-gradient-t-+} and formula \eqref{quotient-derivative-of-J-value-of-J}. The last statement is a consequence of Lemma \ref{lemma-quotient-derivative-of-J-value-of-J}, Lemma \ref{go_go_1} and Lemma \ref{go_go_2}.
\end{proof}
\section{Particular cases}
\begin{lem}
	If
	\begin{equation}\label{def-u}
	u=\alpha t^{+}-\beta t^{-},\quad \alpha,\beta \in \mathbb{R},\quad\alpha,\beta>0,
	\end{equation}
	then
	\[\int_{B_R^{\mathbb{H}^1}(0)}\frac{\left|\nabla_{\mathbb{H}^1}u^{+}\right|^{2}}{\left|\xi\right|_{\mathbb{H}^1}^{2}}\hspace{0.1cm}d\xi\int_{B_R^{\mathbb{H}^1}(0)}\frac{\left|\nabla_{\mathbb{H}^1}u^{-}\right|^{2}}{\left|\xi\right|_{\mathbb{H}^1}^{2}}\hspace{0.1cm}d\xi=4\pi^2\alpha^2\beta^2R^8.\]
\end{lem}

\begin{proof}
	Let us begin pointing out that, from \eqref{def-u}, we have
	\begin{align*}
	&u^{+}=\alpha t^{+},\\
	&u^{-}=\beta t^{-},
	\end{align*}
	which implies
	\begin{align*}
	&\nabla_{\mathbb{H}^1}u^{+}=\left(Xu^{+},Yu^{+}\right)=
	\begin{cases}
	(2y\alpha,-2x\alpha)&t>0\\
	0&t<0
	\end{cases}
	=\begin{cases}
	2\alpha(y,-x)&t>0\\
	0&t<0,
	\end{cases}\\
	&\nabla_{\mathbb{H}^1}u^{-}=\left(Xu^{-},Yu^{-}\right)=\begin{cases}
	(2y\beta,-2x\beta)&t<0\\
	0&t>0
	\end{cases}
	=\begin{cases}
	2\beta(y,-x)&t<0\\
	0&t>0,
	\end{cases}
	\end{align*}
	that is
	\begin{align}\label{nabla-horiz-u-+-nabla-horiz-u--}
	&\nabla_{\mathbb{H}^1}u^{+}=
	\begin{cases}
	2\alpha(y,-x)&t>0\\
	0&t<0,
	\end{cases}\notag\\
	&\nabla_{\mathbb{H}^1}u^{-}=
	\begin{cases}
	2\beta(y,-x)&t<0\\
	0&t>0.
	\end{cases}
	\end{align}
	Consequently, in view of \eqref{nabla-horiz-u-+-nabla-horiz-u--}, we get
	\begin{align*}\
	\left|\nabla_{\mathbb{H}^1}u^{+}\right|^{2}=
	\begin{cases}
	4\alpha^{2}\left(x^{2}+y^{2}\right)&t>0\\
	0&t<0,
	\end{cases}\notag\\
	\left|\nabla_{\mathbb{H}^1}u^{-}\right|^{2}=
	\begin{cases}
	4\beta^{2}(x^{2}+y^{2})&t<0\\
	0&t>0,
	\end{cases}
	\end{align*}
	which yields
	\begin{align}\label{prod-int-u-+-u---2}
	&\int_{B_R^{\mathbb{H}^1}(0)}\frac{\left|\nabla_{\mathbb{H}^1}u^{+}\right|^{2}}{\left|\xi\right|_{\mathbb{H}^1}^{2}}\hspace{0.1cm}d\xi\int_{B_R^{\mathbb{H}^1}(0)}\frac{\left|\nabla_{\mathbb{H}^1}u^{-}\right|^{2}}{\left|\xi\right|_{\mathbb{H}^1}^{2}}\hspace{0.1cm}d\xi=16\alpha^{2}\beta^{2}\int_{B_R^{\mathbb{H}^1}(0)\cap\left\{t>0\right\}}\frac{x^{2}+y^{2}}{\left|\xi\right|_{\mathbb{H}^1}^{2}}\hspace{0.1cm}d\xi\notag\\
	&\times\int_{B_R^{\mathbb{H}^1}(0)\cap\left\{t<0\right\}}\frac{x^{2}+y^{2}}{\left|\xi\right|_{\mathbb{H}^1}^{2}}\hspace{0.1cm}d\xi.
	\end{align}
	In particular, since 
	\begin{equation}\label{norm-horiz-xi}
	\left|\xi\right|_{\mathbb{H}^1}=((x^{2}+y^{2})^{2}+t^{2})^{1/4},
	\end{equation}
	we have that the function
	\[\frac{x^{2}+y^{2}}{\left|\xi\right|_{\mathbb{H}^1}^{2}}\]
	is symmetric with respect to $\left\{t=0\right\}$ and thus
	\[\int_{B_R^{\mathbb{H}^1}(0)\hspace{0.05cm}\cap\hspace{0.05cm}\left\{t>0\right\}}\frac{x^{2}+y^{2}}{\left|\xi\right|_{\mathbb{H}^1}^{2}}\hspace{0.1cm}d\xi=\int_{B_R^{\mathbb{H}^1}(0)\hspace{0.05cm}\cap\hspace{0.05cm}\left\{t<0\right\}}\frac{x^{2}+y^{2}}{\left|\xi\right|_{\mathbb{H}^1}^{2}}\hspace{0.1cm}d\xi.\]
	This fact, together with \eqref{prod-int-u-+-u---2}, gives
	\begin{equation}\label{prod-int-u-+-u---3}
	\int_{B_R^{\mathbb{H}^1}(0)}\frac{\left|\nabla_{\mathbb{H}^1}u^{+}\right|^{2}}{\left|\xi\right|_{\mathbb{H}^1}^{2}}\hspace{0.1cm}d\xi \int_{B_R^{\mathbb{H}^1}(0)}\frac{\left|\nabla_{\mathbb{H}^1}u^{-}\right|^{2}}{\left|\xi\right|_{\mathbb{H}^1}^{2}}\hspace{0.1cm}d\xi=16\alpha^{2}\beta^{2}\bigg(\int_{B_R^{\mathbb{H}^1}(0)\hspace{0.05cm}\cap\hspace{0.05cm}\left\{t>0\right\}}\frac{x^{2}+y^{2}}{\left|\xi\right|_{\mathbb{H}^1}^{2}}\hspace{0.1cm}d\xi\bigg)^{2}.
	\end{equation}
	Let us analyze, at this point,
	\[\int_{B_R^{\mathbb{H}^1}(0)\hspace{0.05cm}\cap\hspace{0.05cm}\left\{t>0\right\}}\frac{x^{2}+y^{2}}{\left|\xi\right|_{\mathbb{H}^1}^{2}}\hspace{0.1cm}d\xi.\]
	Specifically, we use the change of variables \eqref{polar-coordinates}.
	In particular, by virtue of this and \eqref{norm-horiz-xi}, we obtain
	\begin{align*}
	&x^{2}+y^{2}=(\rho\sqrt{\sin\varphi}\cos\theta)^{2}+(\rho\sqrt{\sin\varphi}\sin\theta)^{2}=\rho^{2}\sin\varphi\cos^{2}\theta+\rho^{2}\sin\varphi\sin^{2}\theta=\rho^{2}\sin\varphi,
	\end{align*}
	and $\left|\xi\right|_{\mathbb{H}^1}^{2}=\rho^{2},$ namely
	\begin{equation}
	\begin{split}\label{x-^-2-+-y-^-2-norm-horiz-xi-^-2-spherical-coordinates}
	&x^{2}+y^{2}=\rho^{2}\sin\varphi\\
	&\left|\xi\right|_{\mathbb{H}^1}^{2}=\rho^{2}.
	\end{split}
	\end{equation}
	In addition, using \eqref{polar-coordinates} and \eqref{x-^-2-+-y-^-2-norm-horiz-xi-^-2-spherical-coordinates}, we also achieve, because $\rho>0,$
	\begin{align*}
	&B_R^{\mathbb{H}^1}(0)\cap \left\{t>0\right\}=\left\{0<\rho<R\right\}\cap \left\{\rho^{2}\cos\varphi>0\right\}=\left\{0<\rho<R\right\}\cap\left\{\cos\varphi>0\right\}\\
	&=\left\{0<\rho<R\right\}\cap\left\{0<\varphi<\frac{\pi}{2}\right\}=(0,R)\times\left(0,2\pi\right)\times\left(0,\frac{\pi}{2}\right),
	\end{align*}
	which entails
	\begin{equation}\label{B-H-0-R-cap-t->-0-spherical-coordinates}
	B_R^{\mathbb{H}^1}(0)\cap \left\{t>0\right\}=(0,R)\times\left(0,2\pi\right)\times \left(0,\frac{\pi}{2}\right).
	\end{equation}
	At this point, we compute the Jacobian matrix of \eqref{polar-coordinates}. Specifically, we rewrite \eqref{polar-coordinates} as 
	\[\xi=(x,y,t)=T(\rho,\theta,\varphi)=(\rho\sqrt{\sin\varphi}\cos\theta,\rho\sqrt{\sin\varphi}\sin\theta,\rho^{2}\cos\varphi),\]
	and we get, denoting $J_{T}$ the Jacobian matrix of $T(\rho,\theta,\varphi),$
	\begin{align*}
	&J_{T}=
	\begin{bmatrix}
	\dfrac{\partial }{\partial \rho}\left(\rho\sqrt{\sin\varphi}\cos\theta\right)&\dfrac{\partial }{\partial \theta}\left(\rho\sqrt{\sin\varphi}\cos\theta\right)& \dfrac{\partial }{\partial \varphi}\left(\rho\sqrt{\sin\varphi}\cos\theta\right)\\
	\\
	\dfrac{\partial }{\partial \rho}\left(\rho\sqrt{\sin\varphi}\sin\theta\right)&\dfrac{\partial }{\partial \theta}\left(\rho\sqrt{\sin\varphi}\sin\theta\right)& \dfrac{\partial }{\partial \varphi}\left(\rho\sqrt{\sin\varphi}\sin\theta\right)\\
	\\
	\dfrac{\partial }{\partial \rho}\left(\rho^{2}\cos\varphi\right)&\dfrac{\partial }{\partial \theta}\left(\rho^{2}\cos\varphi\right)& \dfrac{\partial }{\partial \varphi}\left(\rho^{2}\cos\varphi\right)
	\end{bmatrix}\\
	&\\
	&=
	\begin{bmatrix}
	\sqrt{\sin\varphi}\cos\theta&-\rho\sqrt{\sin\varphi}\sin\theta& \displaystyle\frac{\cos\varphi}{2\sqrt{\sin\varphi}}\rho\cos\theta\\
	\\
	\sqrt{\sin\varphi}\sin\theta&\rho\sqrt{\sin\varphi}\cos\theta& \displaystyle\frac{\cos\varphi}{2\sqrt{\sin\varphi}}\rho\sin\theta\\
	\\
	2\rho\cos\varphi&0& -\rho^{2}\sin\varphi
	\end{bmatrix},
	\end{align*}
	in other words
	\begin{equation}\label{Jacobian-matrix-change-of-variables}
	J_{T}=
	\begin{bmatrix}
	\sqrt{\sin\varphi}\cos\theta&-\rho\sqrt{\sin\varphi}\sin\theta& \displaystyle\frac{\cos\varphi}{2\sqrt{\sin\varphi}}\rho\cos\theta\\
	\\
    \sqrt{\sin\varphi}\sin\theta&\rho\sqrt{\sin\varphi}\cos\theta& \displaystyle\frac{\cos\varphi}{2\sqrt{\sin\varphi}}\rho\sin\theta\\
	\\
	2\rho\cos\varphi&0& -\rho^{2}\sin\varphi
	\end{bmatrix}.
	\end{equation}
	We now compute $\left|\det J_{T}\right|.$	Precisely, we obtain, in view of \eqref{Jacobian-matrix-change-of-variables},
	\begin{align*}
	&\left|\det J_{T}\right|=\left|-\rho^{3}\cos^{2}\varphi \cos^{2}\theta-\rho^{3}\sin^{2}\varphi \sin^{2}\theta-\rho^{3}\cos^{2}\varphi\sin^{2}\theta-\rho^{3}\sin^{2}\varphi\cos^{2}\theta\right|\\
	&=|-\rho^{3}\cos^{2}\varphi-\rho^3\sin^{2}\varphi|=|-\rho^3|=\rho^3,\\
	\end{align*}
	which gives
	\begin{equation}\label{absolute-value-det-Jacobian-matrix-spherical-coordinates}
	\left|\det J_{T}\right|=\rho^{3}.
	\end{equation}
	Therefore, using \eqref{x-^-2-+-y-^-2-norm-horiz-xi-^-2-spherical-coordinates}, \eqref{B-H-0-R-cap-t->-0-spherical-coordinates} and \eqref{absolute-value-det-Jacobian-matrix-spherical-coordinates}, we achieve
	\begin{align*}
	&\int_{B_R^{\mathbb{H}^1}(0)\hspace{0.05cm}\cap\hspace{0.05cm}\left\{t>0\right\}}\frac{x^{2}+y^{2}}{\left|\xi\right|_{\mathbb{H}^1}^{2}}\hspace{0.1cm}d\xi=\int_{(0,R)\times\left(0,2\pi\right)\times\left(0,\frac{\pi}{2}\right)}\frac{\rho^{2}\sin\varphi}{\rho^{2}}\hspace{0.05cm}\rho^{3}\hspace{0.1cm}d\rho\hspace{0.05cm}d\theta\hspace{0.05cm}d\varphi=\int_{0}^{R}\int_{0}^{2\pi}\int_{0}^{\frac{\pi}{2}}\rho^{3}\sin\varphi\hspace{0.1cm}d\rho\hspace{0.05cm}d\theta\hspace{0.05cm}d\varphi\\
	&=2\pi\int_{0}^{R}\rho^{3}\hspace{0.1cm}d\rho \int_{0}^{\frac{\pi}{2}}\sin\varphi\hspace{0.1cm}d\varphi=2\pi\bigg[\frac{\rho^{4}}{4}\bigg]^{\rho=R}_{\rho=0}\bigg[-\cos\varphi\bigg]^{\varphi=\frac{\pi}{2}}_{\varphi=0}=2\pi\frac{R^{4}}{4}=\frac{\pi R^{4}}{2},
	\end{align*}
	which yields
	\begin{equation}\label{int-u-B-H-0-R-cap-t->-0}
	\int_{B_R^{\mathbb{H}^1}(0)\hspace{0.05cm}\cap\hspace{0.05cm}\left\{t>0\right\}}\frac{x^{2}+y^{2}}{\left|\xi\right|_{\mathbb{H}^1}^{2}}\hspace{0.1cm}d\xi=\frac{\pi R^{4}}{2}.
	\end{equation}
	At this point, by virtue of \eqref{int-u-B-H-0-R-cap-t->-0}, we finally get, from \eqref{prod-int-u-+-u---3},
	\begin{align*}
	&\int_{B_R^{\mathbb{H}^1}(0)}\frac{\left|\nabla_{\mathbb{H}^1}u^{+}\right|^{2}}{\left|\xi\right|_{\mathbb{H}^1}^{2}}\hspace{0.1cm}d\xi\int_{B_R^{\mathbb{H}^1}(0)}\frac{\left|\nabla_{\mathbb{H}^1}u^{-}\right|^{2}}{\left|\xi\right|_{\mathbb{H}^1}^{2}}\hspace{0.1cm}d\xi=16\alpha^{2}\beta^{2}\left(\frac{\pi R^{4}}{2}\right)^{2}=16\alpha^{2}\beta^{2}\frac{\pi^{2}R^{8}}{4}=4\pi^{2}\alpha^{2}\beta^{2}R^{8},
	\end{align*}
	that is
	\[
	\int_{B_R^{\mathbb{H}^1}(0)}\frac{\left|\nabla_{\mathbb{H}^1}u^{+}\right|^{2}}{\left|\xi\right|_{\mathbb{H}^1}^{2}}\hspace{0.1cm}d\xi \int_{B_R^{\mathbb{H}^1}(0)}\frac{\left|\nabla_{\mathbb{H}^1}u^{-}\right|^{2}}{\left|\xi\right|_{\mathbb{H}^1}^{2}}\hspace{0.1cm}d\xi=4\pi^{2}\alpha^{2}\beta^{2}R^{8}.
	\]
\end{proof}
\section{Domains depending on $\varphi$ and $\theta$}
\begin{lem} Let $u$ be one of the two functions of the Theorem \ref{lowerbound_f}. Then
	\[\frac{\displaystyle \int_{\partial B^{\mathbb{H}^1}_1(0)\cap\left\{u>0\right\}}\frac{\left|\nabla_{\mathbb{H}^1}^{\varphi}u\right|^2}{\sqrt{x^2+y^2}}\hspace{0.1cm}d\sigma_{\mathbb{H}^1}(\xi)}{\displaystyle\int_{\partial B^{\mathbb{H}^1}_1(0)\cap\left\{u>0\right\}}u^2\sqrt{x^2+y^2}\hspace{0.1cm}d\sigma_{\mathbb{H}^1}(\xi)}=\frac{\displaystyle\int_{\Omega_{\theta,\varphi}}\sin(\varphi) \bigg(\frac{\partial f}{\partial\theta}+2\frac{\partial f}{\partial\varphi}\bigg)^2d\theta d\varphi}{\displaystyle\int_{\Omega_{\theta,\varphi}}\sin(\varphi)f^2d\theta d\varphi},\]
	where $T(\Omega_{\theta,\varphi})=\partial B^{\mathbb{H}^1}_1(0)\cap\left\{u>0\right\}$ and $u=\rho^\alpha f(\theta,\varphi).$
	In particular, if $u=x,$ it results
	\[\frac{\displaystyle \int_{\partial B^{\mathbb{H}^1}_1(0)\cap\left\{u>0\right\}}\frac{\left|\nabla_{\mathbb{H}^1}^{\varphi}u\right|^2}{\sqrt{x^2+y^2}}\hspace{0.1cm}d\sigma_{\mathbb{H}^1}(\xi)}{\displaystyle\int_{\partial B^{\mathbb{H}^1}_1(0)\cap\left\{u>0\right\}}u^2\sqrt{x^2+y^2}\hspace{0.1cm}d\sigma_{\mathbb{H}^1}(\xi)}=2.\]
	\end{lem}
	\begin{proof} The proof immediately follows from \eqref{phi-component-of-the-heisenberg-gradient-of-v-square-norm-on-the-boundary-of-unitary-koranyi-ball-1}, the fact that $\sqrt{x^2+y^2}=\sin\varphi$ using \eqref{polar-coordinates} and Lemma \ref{lemma-eigenvalues-depending-on-theta-phi-and-phi}.
\end{proof}

\end{document}